\documentclass[11pt]{amsart}
\usepackage{amsfonts}
\usepackage{amsmath}
\usepackage{color}
\usepackage[mathscr]{eucal}
\usepackage{amscd, latexsym, graphicx, mathrsfs, enumerate}
\usepackage{amssymb}
\usepackage{bbm}
\usepackage{dsfont}
\usepackage{url}

\makeindex

\newtheorem{thm}{{Theorem}}[section]
\newtheorem{prop}[thm]{{Proposition}}
\newtheorem{lem}[thm]{{Lemma}}
\newtheorem{cor}[thm]{{Corollary}}

\newtheorem{conj}[thm]{Conjecture}

\newtheorem{remark}[thm]{Remark}
\numberwithin{equation}{section}
\newtheorem{Def}[thm]{Definition}

\newcommand{\bQ}{\overline{\mathbb{Q}}}

\newcommand{\C}{\mathbb{C}}
\newcommand{\R}{\mathbb{R}}
\newcommand{\Q}{\mathbb{Q}}

\newcommand{\Z}{\mathbb{Z}}

\newcommand{\F}{\mathbb{F}}
\newcommand{\A}{\mathbb{A}}

\newcommand{\p}{\mathfrak{p}}

\newcommand{\f}{\bf{f}}

\newcommand{\diag}{{\rm diag}}
\newcommand{\G}{\Gamma}

\newcommand{\ds}{\displaystyle}

\newcommand{\Ga}{\mathbb{G}_a}

\newcommand{\lra}{\longrightarrow}

\newcommand{\sgn}{{\rm sgn}}
\newcommand{\bs}{\backslash}
\newcommand{\bfa}{{\bf a}}

\newcommand{\ve}{\varepsilon}
\newcommand{\disc}{{\rm disc}}
\newcommand{\sL}{L^{{\rm ss}}}
\newcommand{\sP}{P^{{\rm ss}}}

\newcommand{\sM}{M^{{\rm ss}}}
\newcommand{\W}{\mathcal W}
\newcommand{\Wh}{{\rm Wh}}
\newcommand{\wu}{\widetilde{u}}
\newcommand{\wSL}{\widetilde{{\rm SL}}}
\newcommand{\wT}{\widetilde{T}}

\renewcommand{\Bbb}{\mathbb}

\def\GL{{\mathop{\mathrm{GL}}}}
\def\SU{{\mathop{\mathrm{SU}}}}

\def\SL{{\mathop{\mathrm{SL}}}}

\def\SO{{\mathop{\mathrm{SO}}}}

\def\Sp{{\mathop{\mathrm{Sp}}}}
\def\GSp{{\mathop{\mathrm{GSp}}}}

\def\Sym{{\mathop{\mathrm{Sym}}}}

\title[On the Fourier expansion of Gan-Gurevich lifts on $G_2$]
{On the Fourier expansion of Gan-Gurevich lifts \\on the exceptional group of type $G_2$}
\author[Kim and Yamauchi]{Henry H. Kim and Takuya Yamauchi}
\thanks{The first author is partially supported by an NSERC grant \#482564}
\subjclass[2010]{}
\address{Henry H. Kim \\
Department of mathematics \\
 University of Toronto \\
Toronto, Ontario M5S 2E4, CANADA \\
and Korea Institute for Advanced Study, Seoul, KOREA}
\email{henrykim@math.toronto.edu}
\address{Takuya Yamauchi \\ 
Mathematical Inst. Tohoku Univ.\\
 6-3,Aoba, Aramaki, Aoba-Ku, Sendai 980-8578, JAPAN}
\email{takuya.yamauchi.c3@tohoku.ac.jp}

\begin{document}
\maketitle

\begin{abstract}
By using the degenerate Whittaker functions, we study the Fourier expansion of the Gan-Gurevich lifts which are Hecke eigen quaternionic cusp forms of weight $k$ ($k\geq 2$, even) on the split exceptional group $G_2$ over $\Q$ which come from elliptic newforms of weight $2k$ without supercuspidal local components. In particular, our results give a partial answer to Gross' conjecture. 
\end{abstract}

\tableofcontents



\section{Introduction}\label{intro}
Let $\A$ be the ring of adeles of $\Q$. 
Let $G$ be a connected reductive group over $\Q$. It is interesting and important to 
explicitly construct automorphic forms on $G(\A)$. The theory of Eisenstein series 
due to Langlands gives an explicit construction from cusp forms on
each of smaller reductive (Levi) subgroups of $G$ by induction. Therefore, the problem reduces to the construction of cusp forms.
Using theta lifting as in \cite{GRS} is one way but we  
need to check the non-vanishing and determine the image for the lifting. The trace formula would imply that cusp forms exist in abundance, but an explicit construction is a highly non-trivial matter.

When the symmetric space of $G$ is a Hermitian symmetric tube domain, Hecke eigen 
holomorphic cusp forms on $G(\A)$, whose each local representation is 
a constituent of the degenerate principal series, are constructed from 
Hecke eigen holomorphic cusp forms on $\GL_2(\A)$ by several authors 
(cf. \cite{Ik01}, \cite{Ik08}, \cite{Yamana10}, \cite{IY}, \cite{KY1}, \cite{KY2} 
and the references there for more history). Such forms are so called the Ikeda type lift. They are examples of CAP forms. 
Since $G$ has the Siegel parabolic subgroup $P=MN$ with unipotent abelian 
radical $N$, there is a good theory of Fourier expansions along $N$ which plays an important 
role in the above works. 
If the symmetric space of $G$ is Hermitian symmetric, but not a tube domain, establishing the theory of Fourier expansion is 
considerably more difficult (cf. \cite{MS}).

If $G$ does not give rise to a Hermitian structure, for example, $G=G_2$, 
we cannot have holomorphic automorphic forms but 
several people developed the theory of modular forms, and considered the Fourier expansions along a large unipotent subgroup  
after expanding along  a (``small'') abelian unipotent subgroups inside $G$ (cf. \cite{GGS}, \cite{PoFE}).

In this paper, we study the Fourier expansion of the Gan-Gurevich lifts, which are Hecke eigen quaternionic cusp forms on $G_2(\A)$ which come from elliptic newforms without supercuspidal local components. 
Even though $G_2(\R)$ does not have holomorphic discrete series representation, 
it has quaternionic discrete series representations which 
have a small Gelfand-Kirillov dimension. Several 
authors have studied quaternionic modular forms in \cite{GGS}, \cite{Narita24}, \cite{Po},
 \cite{Po-SK}, and \cite{PoFE}. 

To explain our main results, we need to set up the notations and we refer suitable 
sections for details.  
Let $G_2$ be the split exceptional group over $\Bbb Q$ which is of rank 2 and of dimension 14. Let $P=MN$ be the Heisenberg parabolic subgroup where the Levi factor $M$ is isomorphic to 
$GL_2$ and $N$ is a five dimensional Heisenberg group with the one dimensional center $Z_N$. Let $Q=LU$ be the maximal parabolic subgroup such that $L\simeq GL_2$.  
Put 
$W=N/Z_N\simeq \mathbb{G}^4_a$ where $\mathbb{G}_a$ is the 1-dimensional additive group scheme and identify $W$ with a subscheme of $N$ which will be explicitly specified later.     

For each even positive integer $k\ge 2$ and $C$, let $S_{2k}(\G_0(C))^{{\rm new}}$ be the space of all elliptic newforms of weight $2k$ with respect to $\G_0(C)\subset \SL_2(\Z)$.  
We also define its subspace $S_{2k}(\G_0(C))^{{\rm new,ns}}$ generated by  
all Hecke eigen newforms where the corresponding automorphic representation does not have 
supercuspidal local components. 
For each Hecke eigen newform $f\in S_{2k}(\G_0(C))^{{\rm new,ns}}$, if we denote by $\pi=\pi_f=\otimes'_p \pi_{p}=\pi_{\f}\otimes \pi_\infty$, the corresponding cuspidal automorphic 
representation of $\GL_2(\A)$, then there exists a finite set $S(\pi_{\f})$ of 
finite places of $\Q$ such that 
\begin{enumerate}
\item if $p\not\in S(\pi_{\f})\cup \{\infty\}$, $\pi_p=\pi(\mu_p,\mu^{-1}_p)$ 
for some unitary character $\mu_p:\Q^\times_p\lra \C^\times$;
\item if $p\in S(\pi_{\f})$, $\pi_{p}={\rm St}_p\otimes\mu_p$ is a unique subrepresentation of 
$\pi(\mu_p |\ast|^{\frac{1}{2}},\mu_p |\ast|^{-\frac{1}{2}})$ for a unitary character 
$\mu_p:\Q^\times_p\lra \C^\times$ satisfying $\mu^2_p=\bf 1$.
\end{enumerate}
If $C$ is square-free, then for each $p|C$, $\pi_p$ belongs to the second case 
(cf. \cite[Proposition 2.8-(2)]{LW}) and $S(\pi_{\f})$ is the set of all rational primes dividing $C$. 
Using these data, if $p\not \in S(\pi_{\f})$, we define 
an irreducible admissible representation $\Pi_p$ of $G_2(\Q_p)$ to be ${\rm Ind}^{G_2(\Q_p)}_{P(\Q_p)}\mu_p\circ\det$ (normalized induction). If $p\in S(\pi_{\f})$, we define $\Pi_p$ to be the unique maximal subrepresentation of 
 ${\rm Ind}^{G_2(\Q_p)}_{P(\Q_p)}(|\ast|^{\frac{1}{2}}\mu_p)(\det)$. Note that $\Pi_p$ 
 is irreducible except for $p\in S(\pi_{\f})$ and $\mu_p=\bf 1$ in which case it has 
 two irreducible constituents (see Theorem \ref{constituents}). Let 
 $\Pi_\infty=D_k$ be the quaternionic discrete series of weight $k$ and $V_k\simeq 
 \Sym^{2k} \C^2$ be its minimal $K_\infty$-type (see Section \ref{QDS}). 
Since $\Pi_p$ is of class one for all but finitely many $p$, 
we can consider an admissible representation $\Pi(f):=\otimes'_p \Pi_p$ of $G_2(\A)$. 
Henceforth, we assume the following:

\begin{equation}\label{assump}
\text{There is a non-trivial intertwining map $\Pi(f)\longrightarrow \mathcal{A}(G_2(\Q)\bs G_2(\A))$}
\end{equation}
from $\Pi(f)$ to the space of automorphic forms on $G_2(\A)$.
For $\phi\in \Pi(f)$, let $F_f(\ast;\phi)$ be its image under the above intertwining map. Since $D_k$ is tempered, by \cite{Wa84}, $F_f$ is in fact a cusp form.  
We call $F_f(\ast,\phi)$ Gan-Gurevich lift on $G_2$ from $f$. If $\phi_\infty$ is chosen from the minimal $K_\infty$-type $V_k$, then 
 $F_f(\ast;\phi)$ is a quaternionic cusp form in the sense of \cite[Section 7]{GGS} (see Section \ref{qmf}). 
Further, if $C$ is square-free, then we can choose such a $\phi\in \Pi(f)$ so that  
$F_f(g;\phi)$ is fixed by $\ds\prod_{p\nmid C}G_2(\Z_p)\times \prod_{p|C}\G_P(\Z_p)$ where 
 $\G_P(\Z_p)$ is the inverse image of $P(\F_p)$ under  
the reduction map $G_2(\Z_p)\lra G_2(\F_p)$ (see Section \ref{FS} for details). 

It is easy to see that if $p\notin S(\pi_{\f})$, $\Pi_p$ is the Langlands quotient of ${\rm Ind}_{Q(\Q_p)}^{G_2(\Q_p)}\, \pi_p\otimes |\det|^\frac 12$.
Now Gan and Gurevich \cite{GG} constructed a CAP representation $\Pi^{G}$ of $G_2$ which is nearly equivalent to a quotient of 
${\rm Ind}_{Q(\Bbb A)}^{G_2(\Bbb A)} \pi_f\otimes |\det|^{\frac 12}$ where $L(\frac 12,\pi_f)\ne 0$. It is obtained as an exceptional theta correspondence from $PGSp_6$ in the dual pair $G_2\times PGSp_6\hookrightarrow E_7$. At unramified places $p\notin S(\pi_{\f})\cup \{p|C\}$, 
it is $\Pi_p$. However, at the bad places $p\in S(\pi_{\f})\cup \{\infty\}$, it has not been proved that it is $\Pi_p$. 
If $C=1$, we will check that $\Pi^{G}_\infty=D_k$ in Appendix A by using 
Arthur's classification and Li's result \cite{Li}. 
Therefore, if $C=1$ and $L(\frac 12,\pi_f)\neq 0$, $\Pi(f)=\Pi^{G}$ and (\ref{assump}) is true. Note that for each newform $f$ of weight $2k$ $(\ge 12)$ and of level 1, 
the condition $L(\frac 12,\pi_f)\ne 0$ implies that $k$ is even. It is a well-known conjecture that the converse is also true. 
Let $S_0$ be the set of primes such that $\pi_p={\rm St}_p$. If $S_0=\emptyset$, then $\Pi(f)$ is irreducible, and if we assume Arthur's multiplicity formula, refined by Gan and Gurevich (Conjecture \ref{Mundy}),
(\ref{assump}) is true. 

For each $w=(a_1,a_2,a_3,a_4)\in W(\Q)$, let $q(w)\in \Q$ be the Freudenthal's quartic form (see (\ref{disc})).  
By using Jacquet integrals, for each $w\in W(\Q)$ with $q(w)\neq 0$, we will define in Section \ref{DWF}, the functionals 
$$\text{$\widetilde{{\bf w}}^{\mu_p}_w\in {\rm Hom}_{N(\Q_p)}(\Pi_p,\C(\psi_{w,p}))$ for $p<\infty$, \quad  
$W^{(k-\frac{1}{2})}_w\in {\rm Hom}_{N(\R)}(\Pi_\infty,\C(\psi_{w,\infty}))$,}
$$
where $\psi_w=\otimes'_p \psi_{w,p}$ for $w\in W(\Q)$ is an additive character of $N(\Q)\backslash N(\A)$ defined in Section \ref{robust}. 

By the multiplicity-freeness of the degenerate Whittaker spaces studied in Section \ref{DWF} and a robust  
theory of Fourier expansion due to Pollack, we have the following Fourier expansion of $F_f(\ast;\phi)$.
 
\begin{thm}\label{exp-thm} Assume {\rm (}\ref{assump}{\rm )}. 
For each distinguished vector $\phi=\otimes'_p \phi_p\in \Pi(f)$, $F_f(\ast;\phi)$ 
can be expanded as 
\begin{equation}\label{fseries} 
F_f(g;\phi)=\sum_{s\in \Q}F_{(s,0)}(g;\phi)+\sum_{\gamma\in w_\beta X_\beta(\Q)}
\sum_{s\in \Q^\times}F_{(s,0)}(\gamma g;\phi),\ g\in G_2(\A), 
\end{equation}
where 
$$F_{(s,0)}(g;\phi):=\sum_{w=(a_1,a_2,a_3,s)\in W(\Q)_{\ge 0}\atop q(w)<0} 
C^{\mu_{\f}}_{w}(F_f)\Big(\prod_{p<\infty} \widetilde{{\bf w}}^{\mu_p}_{{\rm Ad}(w_\alpha)w}(g_p\cdot \phi_p)\Big)
W^{(k-\frac{1}{2})}_{{\rm Ad}(w_\alpha)w}(g_\infty\cdot\phi_\infty)$$
for $g=(g_p)_p\in G_2(\A)$ and some complex numbers $\{C^{\mu_{\f}}_{w}(F_f)\}$.
Here $w_\beta$ {\rm (}resp. $w_\alpha${\rm)} is the Weyl element in $\sL:=[L,L]\simeq SL_2$ {\rm(}resp. in $M${\rm)} and $X_\beta$ 
is the upper unipotent subgroup of $\sL$. 
Furthermore, the coefficients $\{C^{\mu_{\f}}_w(F_f)\}_w$ completely characterize $F_f$. 
\end{thm}

\begin{remark}\label{compareGGS} The Fourier coefficient $C^{\mu_{\f}}_w(F_f)$ 
coincide with the one defined in \cite[Section 8]{GGS} up to a constant multiple, which 
depends on the choice of a generator of ${\rm Hom}_{N(\R)}(D_k,\C(\psi_{w,\infty}))$.  
\end{remark}

\begin{remark}\label{main-remark} Assume $k\ge 6$ is even and $C=1$. 
For each rational prime $p$, choose $\phi_p\in \Pi^{G_2(\Z_p)}_p$ with $\phi_p(1)=1$ 
and let $\phi_{\infty,I}$ be the one given in Section \ref{dwf-arch} for a non-empty 
subset $I$ of $\{v\in\Z\ |\ -k\le v\le k\}$. Let $\phi=\otimes'_{p<\infty} \phi_p\otimes \phi_{\infty,I}$.
Then $F_f(\ast;\phi)$ is a non-zero quaternionic Hecke eigen cusp form of weight $k$ 
and of level one. Furthermore, $F_f(\ast;\phi)$ generates $\Pi(f)$ by \cite{NPS}.

In a letter to {\rm(}David{\rm)} Pollack \cite{Gross}, Gross conjectured the existence of 
a Hecke eigen quaternionic cusp form of level one which is a lift from $S_{2k}(\SL_2(\Z))$, with the standard $L$-function in 
Theorem \ref{Lfunct}. By using the exceptional theta lift for the dual pair $(G_2,Sp_6)$ inside the split $E_7$, Gan-Gurevich's result as mentioned gives   
an affirmative answer to his conjecture when $L(\frac 12,\pi_f)\neq 0$ which implies that $k$ is even. Thus, the method does not work for $k$ odd.  Nevertheless, Pollack \cite{PoCoe} constructed quaternionic modular forms of odd weight $k$ using the exceptional theta lift 
for the dual pair $(G_2,F^c_4)$ inside $E_{8,4}$. 
However, in general his method alone does not suffice to show, in accordance with Gross' conjecture on the existence of $\Pi(f)$, 
that the weight $k$ form is the Hecke eigen-lift of a Hecke eigenform $f$ in $S_{2k}(\SL_2(\Z))$ 
without additional representation-theoretic input as in \cite{GG}.
\end{remark}

It is known that $F_0(\ast;\phi):=\ds\sum_{s\in \Q}F_{(s,0)}(\ast;\phi)$ completely determines $F_f$ (see \cite[Lemma 8.5]{GGS}) and 
the coefficients $\{C^{\mu_{\f}}_w(F_f)\}$ characterize $F_0(\ast;\phi)$ by the multiplicity-freeness for the Whittaker spaces. Thus, it is important to study the coefficients $C^{\mu_{\f}}_w(F_f)$ which would reflect some arithmetic nature of $\Pi(f)$. In fact, Gross conjectured that 
the square of $C^{\mu_{\f}}_w(F_f)$ satisfies a formula which is an analogue of Kohnen-Zagier formula 
\cite{KZ}. To explain it, we need a few notations. 
Let $W(\Z):=\{(a_1,a_2,a_3,a_4)\in W(\Q)\ |\ a_1,a_4\in \Z,\ a_2,a_3\in\frac{1}{3}\Z \}$. 
For each $w\in W(\Z)$, one can attach a cubic ring $A_w$ over $\Z$
(see \cite[Proposition 4.2]{GGS}) and it is known that $E_w:=A_w\otimes_\Z\Q$ is an \'etale $\Q$-algebra  if and only if 
$q(w)\neq 0$. The ring $A_w$ is said to be maximal if it is maximal in 
$E_w$. Let $\rho_{A_w}:G_\Q:={\rm Gal}(\bQ/\Q)\lra \GL_2(\C)$ be the 
Artin representation such that $\zeta_{A_w}(s)=\zeta(s)L(s,\rho_{A_w})$ 
(cf. \cite[Section 3]{SST}). Let $L(s,f\otimes\rho_{A_w})$ be the unnormalized Rankin-Selberg $L$-function so that $s=k$ is the central point.

\begin{conj}\label{conj}[Gross \cite{Gross}]\label{Gross-conj} Let $\Q_f$ be the Hecke field of $f$.  
\begin{enumerate}[a{\rm )}]
\item By rescaling, one can normalize $F_f(\ast;\phi)$ so that $C^{\mu_{\f}}_w(F_f)\in \Q_f$ for any $w$ 
such that $A_w$ is maximal.  
\item For such a $w$, it holds 
$$\ds\frac{L(k,f\otimes \rho_{A_w})}{\langle f,f \rangle}=
\frac{C^{\mu_{\f}}_{w}(F_f)^2}{\langle F_f,F_f \rangle}\cdot 
\frac{\pi^{2k}}{\Gamma(k)^2 |q(w)|^{k-\frac{1}{2}}}
$$
where $\langle \ast,\ast \rangle$ stands for the Petersson inner product.
\end{enumerate}
\end{conj}


The main purpose of our paper is to understand $C^{\mu_{\f}}_w(F_f)$. 
Let $\{c_t\}_{t\in \Q_{<0}}$ be the collection of complex numbers defined in 
Section \ref{FESC} which are closely related to the Fourier coefficients of the modular form of weight $k+\frac 12$, which corresponds to $f$ by the Shimura correspondence. Let $w=(a_1,a_2,a_3,a_4)\in W(\Z)\cap W(\Q)_{\ge 0}$ with $q(w)<0$. Assume that $E_w$ is isomorphic to 
$\Q^3$ or a product of $\Q$ and a 
quadratic field (in fact, a real quadratic field by the condition $w\ge 0$). Then, we can write as $w={\rm Ad}(m'^{-1})(t,0,S,0)$ for some 
$m\in M(\Q)$ where $m'={\rm Ad}(w_\alpha)(m)$ 
and $t,S\in \Q$ satisfying $t<0$ and $S>0$. 
 
\begin{thm}\label{main1} Assume {\rm (}\ref{assump}{\rm )}. 
For above $w={\rm Ad}(m'^{-1})(t,0,\frac{S}{3},0)\in W(\Q)$ with $m'={\rm Ad}(w_\alpha)(m)$, there exits a non-zero constant 
$C(S)$ depending only on $S$ and $k$ such that 
$$C^{\mu_{\f}}_w(F_f)=C(S)\mu_{\f}(\det(m))^{-1}\mu_{\f}(S)^{-1}c_{tS}.$$
\end{thm} 
For $w=(t,0,\frac{1}{3},0)$ with the square-free integer $t\in \Z_{<-1}$ such that $-t$ is the fundamental discriminant of the 
quadratic field $\Q(\sqrt{-t})$, the above theorem shows that the square $C^{\mu_{\f}}_w(F_f)^2=C(1)^2 c_{t}^2$ can be written in terms of $L(k,f\otimes \chi_{\Q(\sqrt{-t})/\Q})$ by using
 Kohnen-Zagier formula \cite{KZ}. Thus, once we could relate $\langle F_f,F_f \rangle$ 
 with $\langle f, f\rangle $ as shown in \cite[Corollary 1]{Kohnen}, we can deduce Conjecture \ref{conj}  from Theorem \ref{main1}. 
 
 When $E_w\simeq \Q^3$, $w=
 {\rm Ad}(m'^{-1})(-1,0,\frac{1}{3},0)$ for some $m\in M(\Q)$ and 
 $$C^{\mu_{\f}}_{w'}(F_f)=\mu_{\f}(\det(m))^{-1}C(1) c_{-1}\neq 0$$ since 
 $c_{-1}$ is proportional to $L(k,f)$ (see Section \ref{FESC}). Therefore, we have the following: 
\begin{cor}\label{main-cor}For $w\in W(\Q)$ such that $E_w\simeq \Q^3$, $C^{\mu_{\f}}_w(F_f)\neq 0$ if and only if $L(k,f)\neq 0$. 
\end{cor} 
 
The claims on Fourier coefficients will be checked by carefully studying 
the Fourier-Jacobi expansions along $\widetilde{\sL(\A)}\ltimes 
(U(\A)/Z_U(\A))$ with 
techniques in \cite{IY} and \cite{KY2}. A key is to use Whittaker functionals 
$\widetilde{{\bf w}}^{\mu_p}_w$ and $W^{(k-\frac{1}{2})}_w$ which do match 
with the representation theoretic study of local Fourier-Jacobi expansions 
(cf. Proposition \ref{d-case}). The functional $W^{(k-\frac{1}{2})}_w$ is a substitution of 
Pollack's explicit spherical functions in \cite[Theorem 3.4]{Po} and his functions are useful for explicit computation at the archimedean place (cf. \cite[Theorem 5.3 and Appendix A]{Po-aut}). Though we do not use his spherical functions directly, 
we will relate $W^{(k-\frac{1}{2})}_w$ with Pollack's functions (see Remark \ref{rel-to-pol}) to use his robust theory. 
In Lemma \ref{beta-whi1-na}, which is crucial, we relate the Jacquet integral at each finite place $p$ along $N$ to the Jacquet integral along the unipotent radical of the Borel of the double cover of $\SL_2$, for an induction on $\widetilde{\SL_2}$ determined by $\mu_p$. These Jacquet integrals for $\SL_2$ give the Fourier coefficients of the Shimura lift of $f$ to $\widetilde{\SL_2}$. We obtain that the Fourier coefficients of $F_f$ for $w$ of the form $(t, 0, \frac{S}{3}, 0)$
are related to Fourier coefficient $c_{tS}$ of the Shimura lift of $f$, thereby proving Theorem \ref{main1} and giving evidence toward Conjecture \ref{conj} for these $w$.

We should remark that Pollack \cite{Po-SK, Po-exc} constructed quaternionic cusp forms of even weight and obtained a similar result as in Theorem \ref{main1} for the Ramanujan delta function 
(see \cite[Corollary 1.2.4]{Po-exc}). In particular, since the dimension of the space of quaternionic modular forms with weight 6 and level 1 is one by Dalal's formula \cite{Dalal}, Pollack's weight 6 form is a unique generator corresponding to the Ramanujan delta function via the Gan-Gurevich lift. 

On the other hand, recently, Pollack \cite[Section 9]{Po-aut} studied 
the formal series as in (\ref{fseries}) for the quaternionic groups except for $G_2$ 
and gave a sufficient criterion for the formal series to have the automorphy. 
In his setting, the 
coefficients ``$\{C^{\mu_{\f}}_w(F_f)\}$'' (in \cite{Po-aut}, it is denoted by $a_w$) which satisfy $P$ and $R$ symmetries, play an important role in checking the automorphy. It seems that our 
strategy using the degenerate Whittaker functionals can be used to study 
$\{a_w\}$ in his setting as well. In particular, the Fourier Jacobi coefficients 
are related to half-integral modular forms as shown in the proof of Theorem \ref{main1} 
and a similar result has been obtained in \cite[Theorem 5.3]{Po-aut} for 
quaternionic groups except for $G_2$. 

On the other hand, we can study 
$\{C^{\mu_{\f}}_w(F_f)\}$ for $w$ when $E_w$ is a field by using cubic base change to 
$E_w$ of $f$ and $F_f$. This will be studied in a forthcoming paper.

We organize this paper as follows. In Section 2, we set up some notations for $G_2$. 
Our description is based on \cite{Po} but it is slightly modified. 
In Appendix C, we also give an explicit description of $G_2$ inside $SO(3,4)$ and its parabolic subgroups as well for the reproducibility. In Sections 3 and 4, we review quaternionic modular forms in the sense of \cite{GGS} and observe their basic features about the Fourier expansions and the automorphy. Through Section 5 to Section 8, we study 
the Fourier expansion of $F_f(\ast;\phi)$  and the Fourier-Jacobi expansion along 
$\widetilde{\sL(\A)}\ltimes (U(\A)/Z_U(\A))$ for the constant term 
of $F_f(\ast;\phi)$ along $Z_U$ with a similar strategy in \cite{IY} 
and \cite{KY2}. Section 5 is not directly related to the later sections but without it, it may be hard to figure out what 
local analogues of the Fourier-Jacobi expansions should be like. In Section \ref{FS}, we prove Theorem \ref{exp-thm}.
In Section \ref{FESC}, we recall the Fourier expansion of the half-integral modular form attached to $f$ by the Shimura correspondence. 
The proof of Theorem 1.4 is given in Section \ref{Pmt}. In Section 11, we compute the degree 7 standard $L$-function attached to the Gan-Gurevich lift, and obtain its Arthur parameter. 
We will study the archimedean component of the Gan-Gurevich lift in Appendix A.  
Appendix B is given to understand 
the Fourier-Jacobi expansion along $\sP(\A)$ of Eisenstein series associated to the degenerate principal series which are induced from $P$. It may be helpful to 
understand what kind of automorphic forms on $M(\A)$ show up in the expansion. 

\smallskip  

\textbf{Acknowledgments.} We would like to thank Wee-Teck Gan, Akihiko Goto, Hiroaki Narita, Tamotsu Ikeda, Aaron Pollack, Shunsuke Yamana, and Satoshi Wakatsuki for helpful discussions and their encouragement. We would like to thank G. Mui\'c who helped us with the degenerate principal series. We would also like to give a special thanks to Wee-Teck who explained a proof of Lemma \ref{irrecomp} and also to Aaron Pollack for pointing out 
a mistake which consequently yields an important observation stated in Proposition \ref{equiv}. 

\smallskip

\section{Preliminaries on the exceptional group $G_2$}\label{pre}  
Let $G_2$ be the split exceptional group over $\Bbb Q$ which 
has rank 2 and dimension 14. It can be viewed as 
a smooth group scheme over $\Z$. As mentioned, it is explicitly given inside $SO(3,4)$ in Appendix C. 

For any algebraic group (or group scheme) $H$, 
we denote by $Z_H$ the center of $H$. 
Let $\alpha,\beta$ be the simple roots of $G_2$ where $\alpha$ is the short root and $\beta$ is the long root. 
The set of positive roots of $G_2$ is given by 
$$\Phi(G_2)^+=\{\alpha,\beta,\alpha+\beta,2\alpha+\beta,3\alpha+\beta,3\alpha+2\beta\}$$
so that the set $\Phi(G_2):=\Phi(G_2)^+\cup (-\Phi(G_2)^+)$ consists of all roots.  
For each $\gamma\in \Phi(G_2)$, we denote by $w_\gamma$ the Weyl element 
corresponding to $\gamma$ and we fix its realization as an element in $G_2(\Z)$. 
As usual, we write $w_{\gamma_1\cdots \gamma_k}=\ds\prod_{i=1}^k w_{\gamma_i}$ 
for Weyl elements $w_{\gamma_1},\ldots,w_{\gamma_k}$.  

Let $B$ be the Borel subgroup with respect to $\Phi(G_2)^+$ and $T$ be the Levi factor 
of $B$ which is the diagonal torus of $G_2$.  
We view $\Phi(G_2)$ as a subset of the (algebraic) character group $X^\ast(T):={\rm Hom}_{{\rm alg}}(T,GL_1)$.   
Let $\Ga={\rm Spec}\hspace{0.5mm}\Z[u]$ be the 1-dimensional additive group scheme over $\Z$.  
For each root $\gamma$, one can associate 
a homomorphism $x_\gamma:\Ga \hookrightarrow G_2$ of group schemes over $\Z$ such that 
$$t x_\gamma(u)t^{-1}=x_\gamma(\gamma(t)),\ t\in T,\ u\in \Ga$$ and 
we put $X_\alpha:={\rm Im}(x_\alpha)$. 
We also define $h_\gamma:GL_1\lra G_2$ by 
\begin{equation}\label{h}
h_\gamma(t)=w_\gamma(t)w_\gamma(1)^{-1},\ w_\gamma(t)=x_\gamma(t)
x_{-\gamma}(-t^{-1})
x_\gamma(t),\ t\in GL_1
\end{equation}
where $w_\gamma(1)=w_\gamma$ in the Weyl group. 

\subsection{The Heisenberg parabolic subgroup}\label{P}
Let $P=MN$ be the Heisenberg parabolic subgroup associated to $\{\alpha\}$. 
Explicitly, 
$$N=\{n=n(a_1,a_2,a_3,a_4,t):=x_\beta(a_1)x_{\alpha+\beta}(a_2)x_{2\alpha+\beta}(a_3)
x_{3\alpha+\beta}(a_4)x_{3\alpha+2\beta}(t)\ |\ a_1,\ldots,a_4,t\in \Ga\}.$$
Then, we see easily that 
\begin{equation}\label{heisen1}
n(a_1,a_2,a_3,a_4,t_1)n(b_1,b_2,b_3,b_4,t_2)=n(a_1+b_1,a_2+b_2,a_3+b_3,a_4+b_4,t_1+t_2-a_4 b_1 + 
3 a_3 b_2).
\end{equation}

The unipotent group $N$ is a Heisenberg group with the center 
$Z_N=\{x_{3\alpha+2\beta}(t)\ |\ t\in \Ga\}$. 
To see it concretely, we need to modify the coordinates of $N$ by 
\begin{equation}\label{new-c}
n_1(a_1,a_2,a_3,a_4,t):=
n(a_1,a_2,a_3,a_4,\frac{1}{2}t-(\frac{1}{2} a_1 a_4 - \frac{3}{2} a_2 a_3)).
\end{equation}
We write 
$n_1({\bf a},t)=n_1(a_1,a_2,a_3,a_4,t)$ for ${\bf a}=(a_1,a_2,a_3,a_4)\in \mathbb{G}^4_a$. Then, we have   
\begin{equation}\label{heisen2}
n_1({\bf a},t_1)n_1({\bf b},t_2)=n_1({\bf a}+{\bf b},t_1+t_2+\langle {\bf a},{\bf b} \rangle)
\end{equation}
where $\langle {\bf a},{\bf b} \rangle=a_1b_4-3a_2b_3+3a_3b_2-a_4b_1$ for 
${\bf a}=(a_1,a_2,a_3,a_4)$ and ${\bf b}=(b_1,b_2,b_3,b_4)$ so that $\langle \ast,\ast\rangle$ is 
a symplectic form on 
$$W:=X_\beta X_{\alpha+\beta}X_{2\alpha+\beta}X_{3\alpha+\beta}\simeq N/Z_N.$$ 
Notice that the above Heisenberg structure is defined over $\Z[\frac 12]$ because of the 
new coordinates. For the Levi part, we have $M\simeq GL_2$ and write 
$m=
m(\begin{pmatrix}
a & b \\
c & d
\end{pmatrix})$ for $\begin{pmatrix}
a & b \\
c & d
\end{pmatrix}\in GL_2$. The above identification can be characterized by the action of 
$M$ on $N$ so that the adjoint action of $m$ is given by 
\begin{equation}\label{action1}
{\rm Ad}(m)(n_1(\bfa,z))=mn_1(\bfa,z)m^{-1}=n_1(\det(m)^{-1}\rho_3(m)\bfa,\det(m)z)
\end{equation}
where $\rho_3(m)\bfa$ is defined by the pullback of the action of $GL_2$ on the RHS of the identification  
\begin{equation}\label{iden}
W\simeq {\rm Sym}^3{\rm St_2},\ \bfa=(a_1,a_2,a_3,a_4)
\longleftrightarrow f_\bfa(u,v)=a_1u^3+3a_2u^2v+3a_3uv^2+a_4v^3
\end{equation}
with $m f_\bfa(u,v)=f(d u + b v, c u + a v)$. 
Here ${\rm Sym}^3{\rm St_2}$ is the symmetric cube of the 2-dimensional standard module 
${\rm St}_2$. Note that $\det^{-1}\otimes\rho_3$ 
corresponds to the adjoint action of $M$ on $W$ and it yields 
$W\simeq \det^{-1}{\rm St}_2\otimes{\rm Sym}^3{\rm St_2}$ as a representation of 
$M$. 
The above action of $M$ on $W$ is slightly different from the one in \cite{Po}. 
Furthermore, it is easy to check that $\langle\rho_3(m) w,x \rangle=\langle w,
\det(m)^3\rho_3(m^{-1})x \rangle$ and $\langle \det(m)^2 \rho_3(m^{-1}) w,x \rangle=\langle w, {\rm Ad}(m)x \rangle$ which will be used later.
In situations where a Haar measure is considered, the modulus character of $P$  is given by 
$\delta_P(mn)=|\det(m)|^3$. 

One can view $W$ as a vector scheme over $\Z$ and for any commutative 
algebra $R$, and $\bfa=(a_1,a_2,a_3,a_4)\in W(R)$, define 
\begin{equation}\label{disc}
q(\bfa)=-\frac{1}{27}\disc_x(f_{\bfa}(x,1))=-3 a_2^2 a_3^2 + 4 a_1 a_3^3 + 4 a_2^3 a_4 - 6 a_1 a_2 a_3 a_4 + a_1^2 a_4^2.
\end{equation}
It is easy to see that $q(\rho_3(m)\bfa)=\det(m)^6q(\bfa)$ for $m\in M$ and $\bfa\in W$ . 
We remark that $q(\bfa)$ is nothing but the $GL_2$-invariant form, so called 
Freudenthal's quartic form for 
$f_\bfa(u,v)$ given in \cite[Section 2.4]{C6}, up to scaling by positive rational numbers. 

\subsection{Another maximal parabolic subgroup}\label{Q} 
Let $Q=LU$ be the maximal parabolic subgroup associated to $\{\beta\}$. 
Explicitly, 
$$U=\{u=u(a_1,a_2,a_3,a_4,z):=x_\alpha(a_1)x_{\alpha+\beta}(a_2)x_{2\alpha+\beta}(a_3)
x_{3\alpha+\beta}(a_4)x_{3\alpha+2\beta}(z)\ |\ a_1,\ldots,a_4,z\in \Ga\}.$$
It has three step nilpotency:
$$U\supset U_1:=[U,U]=X_{2\alpha+\beta}X_{3\alpha+\beta}
X_{3\alpha+2\beta}\supset U_2:=[U_1,U]=Z_U=X_{3\alpha+\beta}
X_{3\alpha+2\beta}.$$ 
The quotient $\widetilde{U}:=U/Z_U$ is a 3-dimensional  Heisenberg group with the center 
$U_1/Z_U=Z_{\widetilde{U}}$. 
We identify  $\widetilde{U}$ (resp. $Z_{\widetilde{U}}$) with $X_\alpha X_{\alpha+\beta}X_{2\alpha+\beta}$ (resp. $X_{2\alpha+\beta}$) and put $\widetilde{u}(a_1,a_2,a_3):=u(a_1,a_2,a_3,0,0)$ for simplicity. 
 It is easy to see that 
\begin{equation}\label{heisen3}
\wu(a_1,a_2,a_3)\wu(b_1,b_2,b_3)\equiv \wu(a_1+b_1,a_2+b_2,a_3+b_3+2a_2b_1)\ {\rm mod}\ Z_U.
\end{equation}
As in the previous section, we modify the coordinates of $\widetilde{U}$ by 
$\wu_1(a_1,a_2,a_3):=\wu(a_1,a_2,a_3+a_1a_2)$. Then, we have   
\begin{equation}\label{heisen4}
\wu_1(a_1,a_2,a_3)\wu_1(b_1,b_2,b_3)\equiv \wu_1(a_1+b_1,a_2+b_2,a_3+b_3+
\langle(a_1,a_2),(b_1,b_2)\rangle)\ {\rm mod}\ Z_U 
\end{equation}
where $\langle(a_1,a_2),(b_1,b_2)\rangle=a_2 b_1 - a_1 b_2$ so that $\langle \ast,\ast\rangle$ is 
a symplectic form on $X_\alpha X_{\alpha+\beta}$. 
For the Levi part, we fix $L\simeq GL_2$ so that we have the modulus character 
$\delta_Q=|\det|^5$ in situations where a Haar measure is considered.  
If we write $\ell=
\ell(A)$ for $A=\begin{pmatrix}
a & b \\
c & d
\end{pmatrix}\in GL_2$, then it holds 
\begin{equation}\label{ml}
l(\begin{pmatrix}
a & 0 \\
0 & d
\end{pmatrix})
=
m(\begin{pmatrix}
ad & 0 \\
0 & a
\end{pmatrix}),\ l(\begin{pmatrix}
1 & b \\
0 & 1
\end{pmatrix})=n(-b,0,0,0,0),\ 
m(\begin{pmatrix}
1 & b \\
0 & 1
\end{pmatrix})=u(-b,0,0,0,0).
\end{equation} 
The identification $L\simeq GL_2,\ \ell=\ell(A)\longleftrightarrow A$ can also be characterized by 
the action on several unipotent groups. For instance, we have 
\begin{equation}\label{action-tildeU}
\ell^{-1}\wu_1(a_1,a_2,a_3)\ell\equiv \wu_1(\det(A)^{-1}(a a_1 + c a_2), 
\det(A)^{-1}(b a_1 + d a_2),\det(A)^{-1}a_3)\ 
 {\rm mod}\ Z_U  
\end{equation}
and the equality in $U$:
\begin{equation}\label{actionZ}
\ell^{-1}z(x,y)\ell=z(\det(A)^{-2}(x,y)A)\text{ for $z(x,y):=u(0,0,0,x,y)\in Z_U$} 
\end{equation}
where $(x,y)A$ means the usual matrix multiplication.


\section{Quaternionic discrete series representations}\label{QDS}
We refer \cite[Section 6]{GGS} for basic facts on quaternionic discrete series representations. 

Let $K_\infty$ be the maximal compact subgroup of $G_2(\R)$ which is explicitly given in 
\cite[Section 4.1.1]{Po}.  By construction, it is easy to see that 
$M(\R)\cap K_\infty=\{m(\pm I_2)\}$ and $L(\R)\cap K_\infty\simeq SO(2)$. 

Since $G_2(\R)/K_\infty$ does not give rise to a Hermitian symmetric domain, $G_2(\R)$ does not have a holomorphic discrete series 
representation. However, it has a quaternionic discrete series representation $D_k$, parametrized by an integer  
$k\geq 2$, with infinitesimal character $\rho+(k-2)\beta_0$, where $\beta_0=3\alpha+2\beta$ is the highest root and $\rho=5\alpha+3\beta$ is the half sum of positive roots. 
We have $K_\infty\simeq (\SU(2)_{3\alpha+2\beta}\times \SU(2)_\alpha)/\mu_2$ where the first 
(resp. the second) factor 
corresponds to the long root $3\alpha+2\beta$ (resp. the short root $\alpha$). 
Then, we have 
\begin{equation}\label{k-type}
D_k|_{K_\infty}=\bigoplus_{n\ge 0}V_{k,n},\  V_{k,n}:=\Sym^{2k+n}(\C^2)\boxtimes \Sym^n(W(\C)).
\end{equation}
Since $W\simeq \Sym^3\C^2$ as a representation $\SU(2)_\alpha$, by \cite[Theorem 1.3]{HHLS}, we have 
$$\Sym^n(W(\C))\simeq\bigoplus_{i=0}^{n}(\Sym^{3n-2i}\C^2)^{\oplus 
(\lfloor \frac{i}{2}\rfloor-\lfloor \frac{i-1}{3}\rfloor)}\bigoplus_{i=n+1}^{\lfloor \frac{3n}{2}\rfloor}
(\Sym^{3n-2i}\C^2)^{\oplus 
(\lfloor \frac{i}{2}\rfloor-\lfloor \frac{i-1}{3}\rfloor-\lfloor \frac{i-n-1}{2}\rfloor-1)},$$
where $\lfloor \ast \rfloor$ stands for the floor function so that $\lfloor -\frac{1}{3} \rfloor
=-1$. 

It is known that $D_k$ is a submodule of a degenerate principal series representation ${\rm ind}_{P(\Bbb R)}^{G_2(\Bbb R)} \lambda_k$ 
(an unnormalized induction), where $\lambda_k$ is the 1-dimensional representation of 
$\GL_2(\Bbb R)$ defined by 
$\lambda_k=\sgn^k |\det|^{k+1}$. We remark that we chose the identification of the Levi subgroup $M$ of $P$ such that 
 $\delta_P(m)=|\det(m)|^3$ while $\delta_P(m)=|\det(m)|^{-3}$ 
in the setting of \cite[Section 6]{GGS}.

The minimal $K_\infty$-type of $D_k$ is $V_k:=V_{k,0}=\Sym^{2k}(\C^2)$. We denote by 
$\tau_k:K_\infty\lra \GL_\C(V_k)$ the corresponding representation of $K_\infty$. 

\section{Modular forms on $G_2$ and their Fourier expansions}
In this section, we review the crucial results of \cite{Po} and \cite[Section 7]{GGS} for 
the Fourier (Fourier-Jacobi) expansions of modular forms on $G_2(\A)$ which generate 
the quaternionic discrete series $D_k$ as a representation of $G_2(\R)$. Let 
$(\tau_k,V_k)$ be the minimal $K_\infty$-type of $D_k$ and $(\tau^\vee_k,V^\vee_k)$ 
the dual of $(\tau_k,V_k)$.  

\subsection{Quaternionic Modular forms on $G_2$}\label{qmf}
Let us first recall the definition of modular forms on $G_2(\A)$ 
due to Gan-Gross-Savin \cite[Section 7]{GGS}. 
\begin{Def} \label{qmf}
A {\rm (}quaternionic{\rm)} modular form $F$ on $G_2$ of weight $k$ is a $V^\vee_k$-valued function on $G_2(\Bbb A)=
G_2(\A_f)\times G_2(\Bbb R)$ such that
\begin{enumerate} 
\item $F(\gamma g \kappa_\infty)=\tau^\vee_k(\kappa_\infty)^{-1}F(g),\ g\in G_2(\A)$ 
  for any $\gamma\in G_2(\Bbb Q)$ and any $\kappa_\infty\in K_\infty$;
  \item $F$ is right-invariant under some open compact subgroup of 
  $G_2(\A_f)$;
  \item $F$ is annihilated by an ideal of finite codimension in $Z(\frak g)$ where 
  $\frak g$ stands for the complexification of ${\rm Lie}(G_2(\R))$;
  \item for any $g_{\f}\in G_2(\A_{\f})$, $F(g_{\f}g_\infty)$ is of uniform moderate growth 
  in $g_\infty \in G_2(\Bbb R)$;
  \item $F$ generates $D_k$ as an admissible representation of $G_2(\R)$. 
  \end{enumerate}  
Furthermore, such an $F$ is said to be a {\rm (}quaternionic{\rm)} cusp form if the constant term along the unipotent radical of any proper $\Q$-parabolic subgroup of $G_2$
vanishes {\rm (}cf. \cite[Section 1.9]{BJ}{\rm )}.   
\end{Def}
Note that the fifth condition implies the third condition because any quaternionic modular forms are annihilated by the Schmid operator (\cite[Section 4]{Po}). 

\subsection{A robust theory of the Fourier expansion due to Pollack}\label{robust}
Let $\psi=\otimes'_p\psi_p:\Q\bs \A_\Q\lra \C^\times$ be the standard additive 
character (cf. \cite[$\psi^{{\rm st}}$ in Section 4.2]{KY1}). 
For each $t\in \Q$, define $\psi_t$ by $\psi_t(\ast)=\psi(t \ast)$. 
Let $F:G_2(\A)\lra V_k^\vee$ be a quaternionic modular form of weight $k$. 
Since $F$ is left invariant under $Z_N(\Q)$, we have the Fourier expansion along $Z_N$:
\begin{eqnarray}\label{fexp}
F=\sum_{t\in \Q}F_t=F_0+\sum_{t\in \Q^\times}F_t,\quad F_t(g):=\int_{Z_N(\Q)\bs Z_N(\A)}F(zg)\overline{\psi_t(z)}dz 
\end{eqnarray}
where $dz$ is induced from the Haar measure on $Z_N(\A)$ with 
${\rm vol}(Z_N(\widehat{\Z}))=1$.  

The constant term $F_0$ along $Z_N$ has nice properties and one of the 
most important facts is that if $F_0$ is zero, then so is $F$ (see \cite[Lemma 8.5]{GGS}).  
Pollack further expanded $F_0$ explicitly along $N/Z_N$ by using harmonic analysis of quaternionic discrete series representations (see \cite{Po} for 
$G_2$ and \cite{PoFE} for more general setting).  
He called it a robust theory of the Fourier expansion. 
To explain his results, we need more notations. 
Any additive character on $N(\Q)\bs N(\A)$ is trivial on $Z_N(\A)$ 
and it can be written as $\psi_w(n):=\psi(\langle w,x \rangle),\ n=n(x,t)\in N(\A)$ 
for some $w\in W(\Q)$. 
Thus, we have the Fourier expansion of $F_0$ along $N/Z_N$ as 
\begin{eqnarray}\label{naiveFE}
 F_0(g)=\sum_{w\in W(\Q)}F_{w}(g),\quad F_w(g)=\int_{N(\Q)\bs N(\A)}F(ng)\overline{\psi_w(n)}dn.
\end{eqnarray}
 
Since $G_2$ is a semi-simple split group over $\Q$, 
by using the strong approximation theorem and Iwasawa decomposition with respect to $P$ 
at each place, $F$ and $F_0$ are determined by the values on 
$P(\R)$ and $W(\R)M(\R)$ respectively. Then, he deduced  
(\cite[Theorem 3.4]{Po})
\begin{equation}\label{po-exp}
F_0(n(x)m)=F_{00}(m)+\ds\sum_{w\in W(\Q)\atop w\ge 0}a_F(w)e^{2\pi \sqrt{-1}\langle w,x \rangle}
\W_w(m)
\end{equation}
for $n(x):=n(x,0)\in W(\R),\ m\in M(\R)$ 
where $\W_w(m)$ is a kind of $V_k^\vee$-valued spherical functions described in terms of the modified 
Bessel functions. 
Note that we have $e^{2\pi \sqrt{-1} \langle w,x \rangle}$, but not 
$e^{-2\pi \sqrt{-1}\langle w,x \rangle}$ because of the choice of 
the additive character at the archimedean place. 
It is easy to see that $\W_w(\gamma m)=\sgn(\det(\gamma))^k |\det(\gamma)|^{k+1}
\W_{\det(\gamma)^2 \rho_3(\gamma^{-1})w}(m)$ for any $\gamma\in M(\R)$ which 
will be used in Remark \ref{rel-to-pol}. 

The initial term $F_{00}(g)=\ds\int_{N(\Q)\bs N(\A)}F(ng)dn$ is the constant term along $N$ which is, by definition, identically zero if $F$ is a cusp form. 
The condition $w\ge 0$ means that all roots of  the polynomial 
$f_{w}(z,1)$ (recall (\ref{iden})) are real.  Furthermore, the polynomial $f_w(z,1)$ with $w\ge 0$ is  
separable over $\R$ if and only if $q(w)\neq 0$. 

If $F$ is a cusp form, then he also deduced a finer expansion (\cite[Corollary 1.2.3]{PoFE}): 
\begin{equation}\label{po-exp}
F_0(n(x)m)=\ds\sum_{w\in W(\Q)\atop w\ge 0,\ q(w)<0}a_F(w)e^{2\pi \sqrt{-1} \langle w,x \rangle}
\W_w(m) .
\end{equation}
\begin{remark}\label{sign}
In (\ref{po-exp}), $q(w)$ is the negative of the discriminant of $f_w(z,1)$ up to 
the scaling of the positive rational number and 
the sign is different from the one in \cite[p.116, (4.5)]{GGS}. 
\end{remark}

\subsection{Expansion along $Z_U$}\label{sL}
In this section, we study the Fourier expansion along $Z_U$. 
Recall the coordinate $z(x,y)=u(0,0,0,x,y)$ of $Z_U$ and the (right) action (\ref{actionZ}) of $L$ on $Z_U$. 
Any additive character on $Z_U(\Q)\bs Z_U(\A)$ can be written as 
$\psi_{(s,t)}(z(x,y)):=\psi(sx+ty)=\psi((s,t){}^t(x,y)),\ z(x,y)\in Z_U(\A)$ for some $(s,t)\in \Q^2$. 
Then, we have 
\begin{equation}\label{FEQ}
F=\sum_{(s,t)\in \Q^2}F_{(s,t)},\ F_{(s,t)}(g)=\int_{Z_U(\Q)\bs Z_U(\A)}F(zg)
\overline{\psi_{(s,t)}(z)}dz.
\end{equation}

We further observe each term as follows (cf \cite[Proposition 8.2]{GGS}). 
\begin{lem}\label{chartrans}For any $\gamma=\ell(A)\in L(\Q),\ A\in \GL_2(\Q)$, it holds that 
$$F_{(s,t)}(\gamma g)=F_{\det(A)^{2}(s,t){}^t A^{-1}}(g)$$
where $(s,t){}^t A^{-1}$ is the usual matrix product. 
\end{lem}
\begin{proof}
By (\ref{actionZ}), we have $\gamma z(x,y)\gamma^{-1}=z(\det(A)^{2}(x,y)A^{-1})$. 
Then, $$\psi_{(s,t)}(\gamma z(x,y)\gamma^{-1})=\psi((s,t){}^t(\det(A)^{2}(x,y)A^{-1}))
=\psi_{\det(A)^{2}(s,t){}^t A^{-1}}(z(x,y))$$
and it yields that for $g\in G_2(\A)$,  
\begin{eqnarray}
F_{(s,t)}(\gamma g)&=& \int_{Z_U(\Q)\bs Z_U(\A)}F(z\gamma g)\overline{\psi_{(s,t)}(z)}dz 
\nonumber \\
&=&\int_{Z_U(\Q)\bs Z_U(\A)}F(\gamma^{-1}z\gamma g)\overline{\psi_{(s,t)}(z)}dz 
\text{ (the left invariance)}  \nonumber \\
&=&\int_{Z_U(\Q)\bs Z_U(\A)}F(zg)\overline{\psi_{(s,t)}(\gamma z\gamma^{-1})}dz \nonumber \\
&=& F_{\det(A)^{2}(s,t){}^t A^{-1}}(g). \nonumber
\end{eqnarray}
\end{proof} 
We have $L^{{\rm ss}}\simeq SL_2$ under the identification in Section \ref{Q} so that  
$x_\beta(b)=
\ell(
\begin{pmatrix}
1 & -b \\
0 & 1
\end{pmatrix}
)\in X_\beta\subset \sL$. 
By using Lemma \ref{chartrans}, we have an expansion which fits into $L(\Q)$-invariance. 

\begin{prop}\label{expansion1} Keep the notations in {\rm (\ref{FEQ})}. 
Then, it holds
\begin{eqnarray}\label{FEQR}
F(g)&=&F_{(0,0)}(g)+\sum_{\gamma\in X_\beta(\Q)\bs \sL(\Q)}F_{(1,0)}(\gamma g) \nonumber \\
&=&F_{(0,0)}(g)+\sum_{\gamma\in B_{\sL}(\Q)\bs \sL(\Q)}\sum_{s\in \Q^\times}F_{(s,0)}(\gamma g)  \\
&=& \sum_{s\in \Q}F_{(s,0)}(g)+\sum_{b\in \Q}
\sum_{s\in \Q^\times}F_{(s,0)}(w_\beta x_\beta(b) g),\ g\in G_2(\A). \nonumber
\end{eqnarray}
Furthermore, in terms of the expansion (\ref{fexp}), $F_0(g)=\ds\sum_{s\in \Q}F_{(s,0)}(g)$.
\end{prop}
\begin{proof}
We naturally identify $\Q^2$ with $Z_U(\Q)$. 
Since $\sL(\Q)$ acts transitively on $\Q^2\setminus\{(0,0)\}$ and the stabilizer of $(1,0)$ is $X_\beta(\Q)$, 
by (\ref{FEQ}) and Lemma \ref{chartrans}, $F_{(1,0)}(\gamma g)$ 
exhausts the second term of the first expression when $\gamma$ runs over $ X_\beta(\Q)\bs \sL(\Q)$. 
Similarly, the stabilizer of the set $\{(s,0)\ |\ s\in \Q^\times\}\subset Z_U(\Q)$ is clearly $B_{\sL}(\Q)$. Thus, 
 $F_{(s,0)}(\gamma g)$ exhausts the second term of the second expression when $s$ and $\gamma$ run over $\Q^\times$ 
and $\sL(\Q)$ respectively.   
The claim follows. 
\end{proof}

\subsection{An observation on $F_t$ for $t\ne0$}
Let $\sP=\sM N$ where $\sM=[M,M]\simeq SL_2$ is the derived group of $M$. 
For any automorphic form $F$ on $G_2(\A)$, we can regard $F_t|_{\sP(\A)},\ t\neq 0$, as an element 
in the space $C^\infty(\sP(\Q)\bs \sP(\A))_{\psi_t}$ with $C^\infty$-topology. 
Then, by \cite[Proposition 1.3]{Ik94}, $F_t|_{\sP(\A)}$ belongs to a closed span generated by 
certain functions of the form 
\begin{equation}\label{elliptic-part}
f_{\Phi}(m)\theta_{\Phi}(nm),\ m\in \sM(\A),\ n\in N(\A)
\end{equation}
where $f_\Phi$ is an automorphic form on $\sM(\Q)\bs \sM(\A)$ and 
$\theta_\Phi$ is the theta function attached to a Schwartz function $\Phi$ on 
$X_{2\alpha+\beta}(\A)X_{3\alpha+\beta}(\A)$. 
As observed in Appendix B, when $F$ is a quaternionic Eisenstein series, 
$f_\Phi$ is a sum of Eisenstein series and some kinds of theta function. 

On the other hand, recently, Narita \cite{Narita24} obtained a surprising result that for any quaternionic cusp form $F$ and $F_t$ with 
$t\ne 0$, $f_\Phi$ in (\ref{elliptic-part}) belongs to the continuous spectrum. So even if $F$ is a cusp form, $f_\Phi$ may not be a cusp form. The situation is very complicated.
 As the following result shows, $F_t$ for $t$ non-trivial, has rich information 
as much as $F_0$ and $F$.
\begin{prop}\label{equiv}Let $F$ be an automorphic form on $G_2(\A)$. The followings are equivalent:  
\begin{enumerate}
\item $F=0$.
\item $F_0=0$.
\item  $F_t=0$ for any $t\in\Q^\times$.
\item  $F_t=0$ for any fixed $t\in\Q^\times$.
\end{enumerate}
\end{prop} 
\begin{proof}A key is to use $Z_U=X_{3\alpha+\beta}X_{3\alpha+2\beta}\supset Z_N=
X_{3\alpha+2\beta}$. The equivalence  of (1) and (2) follows from \cite[Lemma 8.5]{GGS}. 
The equivalence  of (3) and (4) follows from the fact that $M(\Q)$ acts transitively on 
$\Q^\times\subset Z_N(\Q)$. 

Assume (3). Then, it is easy to see that $F_{(0,t)}=0$ for any $t\in \Q^\times$. 
By using automorphy of $F$, $F_{(t,0)}(g)=F_{(0,t)}(w_\beta g)=0$. This means that 
$F_w(g):=\ds\int_{N(\Q)\bs N(\A)}F(n g)\overline{\psi_w(n)}dn=0$ for any 
$w=(a_1,a_2,a_3,t)\in W(\Q)$ and $t$ non-zero. If $F_0\neq 0$, then there 
exists non-trivial $w'$ such that $F_{w'}$ is non-zero. By using automorphy, 
there exists $\gamma\in M(\Q)$ such that $F_{w'}(g)=F_{w''}(\gamma g)$ such that 
the last coordinate of $w''$, say $t''$, is non-zero. Thus, $F_{w''}\neq 0$ implies 
$F_{t''}\neq 0$ and we have a contradiction. Therefore, $F_0=0$. 
The implication from (2) to (3) is similarly proved. 
\end{proof}

\section{Fourier-Jacobi expansion of Eisenstein series on $G_2$ along $U/Z_U$}\label{FJEES}
 In this section, we compute the Fourier-Jacobi expansion of 
 Eisenstein series on $G_2$ along the Heisenberg group $\widetilde{U}=U/Z_U$. 
 This section will help to define both local and global analogues of the Fourier-Jacobi expansion 
 which will be studied in Section \ref{FJE}. 
  
 We follow the computation in \cite{KY1} but we refer \cite[Section 1]{Ik94} 
 for the Weil representations because in our setting, the dimension of any Lagrangian  subspace of 
 $\widetilde{U}/Z_{\widetilde{U}}\simeq X_\alpha X_{\alpha+\beta}$ is odd ($1$-dimensional) 
 while the one in \cite{KY1} is even. 
 
For any unitary character $\omega:\Q^\times\bs \A^\times\lra \C^\times$ and $s\in\Bbb C$, 
we denote by $I(s,\omega)$ the degenerate principal series representation of $G_2(\A)$ 
consisting of any smooth, $G_2(\widehat{\Z})\times K_\infty$-finite function $f:G_2(\A)\lra \C$ such that 
\begin{equation}\label{g2-prin}
f(nmg)=\delta^{\frac{1}{2}}_P(m)|\det(m)|^s \omega(\det(m))f(g),\ n\in N(\A),\ m\in M(\A),\ 
g\in G_2(\A) 
\end{equation}
where $\delta^{\frac{1}{2}}_P(m)=|\det(m)|^{\frac{3}{2}}$. 
We identify $\sL$ with $SL_2$ by $\ell=\ell(A)\mapsto A$ (recall Section \ref{Q} if necessary) and let $B$ be the upper 
Borel subgroup of $SL_2$. Let $B_{\sL}$ be the upper Borel subgroup of $\sL$ which is identified with the above $B$. 
 For any section $f\in I(s,\omega)$, we define the Eisenstein series on $G_2(\A)$ of 
 type $(s,\omega)$ by
\begin{equation}\label{ES1}
E(g;f):=\sum_{\gamma\in P(\Q)\bs G_2(\Q)}f(\gamma g),\ g\in G_2(\A).
\end{equation}
 
We write $V=\{v(x,y,z):=\widetilde{u}_1(x,y,z)\ |\ x,y,z\in \Ga\}\equiv 
X_\alpha X_{\alpha+\beta}X_{2\alpha+\beta}$ mod $Z_U$ for simplicity. 
Let $\widetilde{J}(\A)=\widetilde{\SL_2(\A)}\ltimes V(\A)$ where $\widetilde{\SL_2(\A)}$ 
is the metaplectic double cover of $\SL_2(\A)$. 
For each 
non-trivial additive character
\begin{equation}\label{extension-psi}
\psi_S=\otimes'_p \psi_{S,p}:U_1\lra U_1/Z_U \stackrel{n(0,0,z,\ast,\ast)\mapsto \psi(Sz)}{\lra}\C^\times,\  
S\in \Q^\times, 
\end{equation}      
we denote by $\omega_{\psi_S}=\otimes'_{p\le \infty}\omega_{S,p}$ the Weil 
representation of $\widetilde{J}(\A)$ 
acting on the Schwartz space $\mathcal{S}(X_{\alpha}(\A))$. 
Explicitly, for each place $p\le \infty$ and $\Phi=\otimes'_{p\le\infty}\Phi_p\in \mathcal{S}(X_{\alpha}(\A))$, it is given by 
\begin{eqnarray}\label{weil-action1}
\omega_{S,p}(v(x,y,z))\Phi_p(t)&=&\phi_p(t+x)\Psi_{S,p}(z+ty+\frac{1}{2}xy),\ 
x,y,z,t\in\Q_p  \\
\label{weil-action2}
\omega_{S,p}((\ell(\begin{pmatrix}
 a & 0 \\
 0 & a^{-1}
 \end{pmatrix}),\ve))\Phi_p(t)&=&\ve \frac{\gamma_p(S)}{\gamma_p(aS)}|a|^{\frac{1}{2}}_p\Phi_p(ta),\ t\in \Q_p,\ a\in\Q^\times_p,\ \ve=\pm1, \\ 
 \label{weil-action3}
\omega_{S,p}((\ell(\begin{pmatrix}
 1 & b \\
 0 & 1
 \end{pmatrix}),\ve))\Phi_p(t)&=&\ve\psi_{S,p}(-bt^2)\Phi_p(t),\ b,t\in \Q_p,\ \ve=\pm1, \\
 \label{weil-action4}
\omega_{S,p}(w_\beta)\Phi_p(t)&=&\ve \gamma_p (F_S\Phi_p)(t),\ t\in\Q_p,\ \ve=\pm1, 
\end{eqnarray}
where $\gamma_p:\Q^\times_p\lra \C^1:=\{z\in \C\ |\ |z|=1\}$ is the Weil constant at $p$ with respect to $\psi_p(S\ast)$ (cf. \cite[p.618]{Ik94}) and  
$$(F_S\Phi_p)(t)=\ds\int_{X_\alpha(\Q_p)} \Phi_p(x)\psi_{S,p}(tx)dx,
$$ 
where $dx$ means the Haar measure on $X_\alpha(\Q_p)$ which is self-dual with respect to the Fourier transform $F_S$. Here we intentionally put the sign $-1$ in front of $bt^2$  
in the third formula (\ref{weil-action3}). The sign yields anti-holomorphic modular 
forms on $\widetilde{\SL_2(\A)}$. 
For each $\Phi\in 
\mathcal{S}(X_{\alpha}(\A))$, we define the theta function 
$$\Theta_{\psi_S}(v(x,y,z)h;\Phi):=\sum_{\xi\in X_\alpha(\Q)}\omega_{\psi_S}(v(x,y,z)h)\Phi(\xi),\ 
v(x,y,z)\in V(\A),\ h\in \widetilde{\SL_2(\A)}$$  
 and put 
\begin{equation}\label{eisen-v-coeff}
E_{\psi_S,\Phi}(h;f):=\int_{V(\Q)\bs V(\A)}
E_{\psi_S}(vh;f)\overline{\Theta_{\psi_S}(v h;\Phi)}dv
\end{equation}
where 
$E_{\psi_S}(g;f):=\ds\int_{U_1(\Q)\bs U_1(\A)}E(u_1 g;f)\overline{\psi_S(u_1)}du_1$ is the Fourier coefficient of 
$E(g;f)$ for $\psi_S$. Put $E_{Z_U}(g;f):=\ds\int_{Z_U(\Q)\bs Z_U(\A)}E(z_U g;f)dz_U$. 
Then, we can also write $E_{\psi_S}(g;f)$ as 
\begin{equation}\label{fgc}
E_{\psi_S}(g;f)=\ds\int_{X_{2\alpha+\beta}(\Q)\bs X_{2\alpha+\beta}(\A)}E_{Z_U}(
x_{2\alpha+\beta}(u) g;f)\overline{\psi(Su)}du.
\end{equation}

Let $K=\Big(\ds\prod_{p<\infty}\SL_2(\Z_p)\Big)\times \SO(2)(\R)$ be the standard maximal compact subgroup $\SL_2(\A)$ and $\widetilde{K}$ 
be its pull-back to $\widetilde{\SL_2(\A)}$. 
It is well-known that the Weil representation $\omega_{\psi_S}$ splits over $\SL_2(\Q)$. 
For a character $\omega:
\Q^\times\bs \A^\times\lra \C^\times$ and $s\in\Bbb C$, 
we define the space $\widetilde{I}^{\psi_S}_1(s,\omega)$ consisting of any $\widetilde{K}$-finite function $\widetilde{f}:\widetilde{\SL_2(\A)}\lra \C$ such that     
\begin{equation}\label{principal-series}
\widetilde{f}((\ell,\ve)g)=\ve \frac{\gamma(1)}{\gamma(a)} \delta^{\frac{1}{2}}_B(\ell)|a|^s \omega(a)\widetilde{f}(g),\ 
 \ell=\begin{pmatrix}
 a & b \\
 0 & a^{-1}
 \end{pmatrix}
 \in B(\A),\ \ve=\pm1,\ g\in \widetilde{\SL_2(\A)}
 \end{equation}
 where $\gamma=\prod_p\gamma_p:\A^\times\lra \C^1$ is the Weil constant with respect to $\psi_S$. 
 For any section $\widetilde{f}\in \widetilde{I}^{\psi_p}_1(s,\omega)$, we define the Eisenstein series on 
 $\widetilde{\SL_2(\A)}$ of type $(s,\omega)$ by
 $$E_1(g;\widetilde{f}):=\sum_{\gamma\in B(\Q)\bs \SL_2(\Q)}\widetilde{f}(\gamma g),\ g\in \widetilde{\SL_2(\A)}.$$

Let us first expand $E_{Z_U}(g;f)$ along the maximal parabolic $Q$ and then compute the 
Fourier-Jacobi coefficient $E_{\psi_S,\Phi}(h;f)$ at $\psi_S$ with $S\in \Q^\times$. 

\begin{lem}\label{eisen-decom1} Keep the notations as above. For each section $f$, on any region in $s\in\Bbb C$ of which 
$E(g;f)$ converges absolutely,  
$E_{Z_U}(g,f)=E^{(1)}(g;f)+E^{(2)}(g;f)+E^{(3)}(g;f)$, where 
\begin{eqnarray*} 
&& E^{(1)}(g;f)=\ds\sum_{\gamma\in B_{\sL}(\Q)\bs \sL(\Q)}f(\gamma g);\\
&& E^{(2)}(g;f)=\ds\sum_{\gamma\in B_{\sL}(\Q)\bs \sL(\Q)}
\sum_{(u_1,u_2)\in \Q^2}\int_{Z_U(\Q)\bs Z_U(\A)}f(w_{\beta\alpha}\gamma x_\alpha(u_1)x_{3\alpha+\beta}(u_2)z_U g)
dz_U \\
&&\phantom{xxxxxxx}=\ds\sum_{\gamma\in B_{\sL}(\Q)\bs \sL(\Q)}
\sum_{(u_1,u_2)\in \Q^2}\int_{Z_U(\Q)\bs Z_U(\A)}f(w_{\beta\alpha}z_U \gamma x_\alpha(u_1)x_{3\alpha+\beta}(u_2)g)
dz_U; \\
&& E^{(3)}(g;f)=\ds\sum_{\gamma\in B_{\sL}(\Q)\bs \sL(\Q)}
\sum_{(u_1,u_2)\in \Q^2}\int_{Z_U(\A)}f(w_{\beta\alpha\beta\alpha}
x_\alpha(u_1)x_{2\alpha+\beta}(u_2)z_U\gamma g)
dz_U.
\end{eqnarray*}
\end{lem}
\begin{proof}
It is easy to see that $\{1,w_{\beta\alpha},w_{\beta\alpha\beta\alpha}\}$ 
is a complete system of representatives of the double coset 
$P(\Q)\bs G_2(\Q)/Q(\Q)=P(\Q)\bs G_2(\Q)/\sL(\Q)U(\Q)$. 
The claims follow from the equalities $P(\Q)\bs Q(\Q)=B_{\sL}(\Q)\bs \sL(\Q)$, 
\begin{eqnarray}\label{coset1}
P(\Q)\bs w_{\beta\alpha}Q(\Q)&=&w_{\beta\alpha}
(B_{\sL}(\Q)\bs \sL(\Q))X_\alpha(\Q)
X_{3\alpha+\beta}(\Q)\nonumber \\
&=&w_{\beta\alpha}X_\alpha(\Q)
X_{3\alpha+\beta}(\Q)
(B_{\sL}(\Q)\bs \sL(\Q)),\nonumber 
\end{eqnarray}
 and also
\begin{eqnarray}\label{coset2}
P(\Q)\bs w_{\beta\alpha\beta\alpha}Q(\Q)&=&w_{\beta\alpha\beta\alpha}
(B_{\sL}(\Q)\bs \sL(\Q))X_\alpha(\Q)
X_{2\alpha+\beta}(\Q)Z_U(\Q)\nonumber \\
&=&w_{\beta\alpha\beta\alpha}X_\alpha(\Q)
X_{2\alpha+\beta}(\Q)Z_U(\Q)
(B_{\sL}(\Q)\bs \sL(\Q)).\nonumber 
\end{eqnarray}
\end{proof}
In what follows, for any smooth function $f$ on $G_2(\A)$, we define the (left) action of 
$\widetilde{\SL_2(\A)}$ on $f|_{\SL_2(\A)\ltimes V(\A)}$ via the natural projection 
$\widetilde{\SL_2(\A)}\lra \SL_2(\A)$. 

\begin{thm}\label{Eisen-exp}Keep the notations
in Lemma \ref{eisen-decom1}. Put $\iota:=w_{\beta\alpha\beta\alpha}w^{-1}_\beta$ 
for simplicity. Then, it holds that 
\begin{eqnarray*}
&& (1)\, R(h;f,\Phi)=
 \ds\int_{X_{\alpha+\beta}(\A)}\int_{X_{2\alpha+\beta}(\A)}\int_{Z_U(\A)}f(\iota z_U v(0,y,z) w_\beta h)\overline{(\omega_{\psi_S}(h)\Phi)(y)\psi(Sz)}d_{z_U}
dydz \\
&& \phantom{xxxxsssssss}=\ds\int_{X_{\alpha+\beta}(\A)}\int_{X_{2\alpha+\beta}(\A)}\int_{Z_U(\A)}f(\iota 
w_\beta z_U v(y,0,z) h)\overline{(\omega_{\psi_S}(h)\Phi)(y)\psi(Sz)}d_{z_U}
dydz \\
&&\text{ is a section of $\widetilde{I}^{\psi_S}_1(s,\chi_S\omega)$};\\ 
&& \text{(2) $E_{\psi_S,\Phi}(\ast;f)$ is the Eisenstein series on $\widetilde{\SL_2(\A)}$ 
of type $(s,\chi_{S}\omega)$ defined by $R(h;f,\Phi)$.} 
\end{eqnarray*}
Here $\chi_{S}(a):=\langle -S,a  \rangle\in \{\pm 1\},\ a\in \A^\times$ where 
$\langle \ast,\ast \rangle$ stands for the quadratic Hilbert symbol on $\A^\times\times \A^\times$.
\end{thm} 
\begin{proof}By Lemma \ref{eisen-decom1} and (\ref{fgc}), we have 
\begin{equation}\label{eisen-FJ}
E_{\psi_S,\Phi}(h;f)=\sum_{i=1}^3
\ds\int_{V(\Q)\bs V(\A)}E^{(i)}_{Z_U}(vh;f)
\overline{\Theta_{\psi_S}(vh;\Phi)}dv.
\end{equation}
We shall try to prove the vanishing of terms for $i=1,2$. 
Since $\psi_S$ is non-trivial and $Z_U$ is stable under the (conjugate) action of $\sL$, clearly, the first term is vanishing. 
For $i=2$, by the unfolding technique, the second term becomes 
$$\ds\sum_{\gamma\in B_{\sL}(\Q)\bs \sL(\Q)}
\int_{X_{\alpha+\beta}(\Q)X_{2\alpha+\beta}(\Q)\bs V(\A)}\int_{X_{3\alpha+2\beta}(\Q)\bs Z_U(\A)}f(w_{\beta\alpha}z_U v \gamma h)\overline{\Theta_{\psi_S}(v\gamma h;\Phi)}
dz_Udv.$$
By a similar computation in \cite[p.242, the proof of Theorem 7.1]{KY1}, it is equal to 
\begin{eqnarray*} 
&& \ds\sum_{\gamma\in B_{\sL}(\Q)\bs \sL(\Q)}
\int_{X_{\alpha+\beta}(\Q)X_{2\alpha+\beta}(\Q)\bs V(\A)}\int_{X_{3\alpha+2\beta}(\Q)\bs Z_U(\A)}f(w_{\beta\alpha}z_U v \gamma h)\\
&& \phantom{xxxxxxxxxxxxxxxsssssssssssssssssssss} \times \overline{\sum_{u\in \Q}F_S(\omega_{\psi_S}(x_{\alpha+\beta}(u)v\gamma h)\Phi(0))}
dz_U dv.
\end{eqnarray*}
We now substitute $v$ for $x_{\alpha+\beta}(u)^{-1}v$ and use the fact that $w_{\beta\alpha}$ 
commutes with $x_{\alpha+\beta}(u)^{-1}$ so that $x_{\alpha+\beta}(u)^{-1}$ 
trivially comes out from inside $f$. Then, by the unfolding technique in the coordinate 
of $X_{\alpha+\beta}$, it becomes 
$$\ds\sum_{\gamma\in B_{\sL}(\Q)\bs \sL(\Q)}
\int_{X_{2\alpha+\beta}(\Q)\bs V(\A)}\int_{X_{3\alpha+2\beta}(\Q)\bs Z_U(\A)}f(w_{\beta\alpha}z_U v \gamma h)\overline{\omega_{\psi_S}(w_\beta v\gamma h)\Phi(0))}
dz_Udv.
$$
By substituting $v$ for $w^{-1}_\beta v w_\beta$, finally,
it becomes
$\ds\sum_{\gamma\in B_{\sL}(\Q)\bs \sL(\Q)}R^{(2)}(\gamma h;f,\Phi)$, where 
$$R^{(2)}(h;f,\Phi)=
\int_{X_{2\alpha+\beta}(\Q)\bs V(\A)}\int_{X_{3\alpha+2\beta}(\Q)\bs Z_U(\A)}f(w_{\beta\alpha}w^{-1}_\beta z_U v w_\beta h)\overline{\omega_{\psi_S}(v w_\beta\gamma h)\Phi(0))}
dz_Udv.$$
We shall prove $R^{(2)}(h;f,\Phi)=0$. Put $w=w_{\beta\alpha}w^{-1}_\beta$. It is easy to see that 
$w v(x,0,0)=v(0,0,-x)w$ and $v(x,y,z)=v(x,0,0)v(0,y,z+xy)$. Furthermore, $v(0,0,-x)$ trivially comes 
out inside $f$. Thus, we have 
\begin{eqnarray*}
&& R^{(2)}(h;f,\phi)= \\
&&\int_{X_{2\alpha+\beta}(\Q)\bs V(\A)}\int_{X_{3\alpha+2\beta}(\Q)\bs Z_U(\A)}f(w z_U v(0,y,z+xy) w_\beta h)\overline{\omega_{\psi_S}(w_\beta h)\Phi(x)\psi(Sz+\frac{1}{2}Sxy)}
dz_U dv,
\end{eqnarray*}
where $v=v(x,y,z)$. 
By substituting $z$ for $z-xy$, it becomes 
\begin{equation}\label{z-xy}\int_{X_{2\alpha+\beta}(\Q)\bs V(\A)}\int_{X_{3\alpha+2\beta}(\Q)\bs Z_U(\A)}f(w z_U v(0,y,z) w_\beta h)\Big(\overline{\omega_{\psi_S}(w_\beta h)\Phi(x)\psi(\frac{1}{2}Sxy)}\Big)\cdot 
\overline{\psi(Sz)}
dz_U dv.
\end{equation}
Now, observe $$X_{2\alpha+\beta}(\Q)\bs V(\A)=X_\alpha(\A)X_{\alpha+\beta}(\A) 
(X_{2\alpha+\beta}(\Q)\bs X_{2\alpha+\beta}(\A)),\ v(0,y,z)=v(0,0,z)v(0,y,0),
$$ 
and 
$wv(0,0,z)w^{-1}=m(\begin{pmatrix}
1 & -z \\
0 & 1
\end{pmatrix}
)$. Therefore, by using the Fourier transform, 
$$(\ref{z-xy})=
\int_{X_{\alpha+\beta}(\A)}\int_{X_{3\alpha+2\beta}(\Q)\bs Z_U(\A)}
f(w z_U v(0,y,0) w_\beta h)\overline{\omega_{\psi_S}(h)\Phi(-\frac{y}{2})}\cdot
\overline{\Big(\int_{X_{2\alpha+\beta}(\Q)\bs X_{2\alpha+\beta}(\A)}\psi(Sz)dz\Big)} 
dz_U dy.
$$
Since $S\neq 0$, clearly 
$\ds\int_{X_{2\alpha+\beta}(\Q)\bs X_{2\alpha+\beta}(\A)}\psi(Sz)dz=0$. 
Hence, we have $R^{(2)}(h;f,\Phi)=0$ and it yields the vanishing of the second term.  

Finally, we handle the case of $i=3$. We shall prove both claims simultaneously. 
As in the previous case, we have  
$$E_{\psi_S,\Phi}(h;f)=\ds\int_{V(\Q)\bs V(\A)}E^{(3)}_{Z_U}(vh;f)
\overline{\Theta_{\psi_S}(vh;\Phi)}dv=\sum_{\gamma\in B_{\sL}(\Q)\bs \sL(\Q)}
R(\gamma h;f,\phi),
$$
where $R(h;f,\Phi)$ is the one in the first claim. Here, we used  
$\iota v(x,y,z)\iota^{-1}=m(\begin{pmatrix}
1 & 0 \\
-z & 1
\end{pmatrix}
)v(0,y,z)$. 

We now check each action of 
$$\ell_b:=(\ell(\begin{pmatrix}
1 & b \\
0 & 1
\end{pmatrix}
),\ve),\ 
\ell_a:=(\ell(\begin{pmatrix}
a & 0 \\
0 & a^{-1}
\end{pmatrix}
),\ve)\in \widetilde{\SL_2(\A)}
.$$
As in \cite[Lemma 7.4-(1),(2)]{KY1}, we observe 
$$\iota z_U(z_1,z_2) v(0,y,z)w_\beta \ell(\begin{pmatrix}
1 & b \\
0 & 1
\end{pmatrix}
)=m(\begin{pmatrix}
1 & by \\
0 & 1
\end{pmatrix}
)X_{3\alpha+\beta}(b)\iota z_U(z_1-bz_2+b^2y^3,z_2-2by^3)v(0,y,z-by^2)w_\beta$$
where we write $z_U=z_U(z_1,z_2)$ 
and 
$$\iota z_U(z_1,z_2)(0,y,z)w_\beta 
\ell(\begin{pmatrix}
a & 0 \\
0 & a^{-1}
\end{pmatrix}
)
=
m(\begin{pmatrix}
a& 0 \\
0 & 1
\end{pmatrix}
)
\iota z_U(\frac{z_1}{a},az_2)v(0,ay,z)w_\beta.$$
By using these relations and (\ref{weil-action3}), first we have 
\begin{eqnarray*}
&& R(\ell_b h;f,\Phi)= \\
&& \ve\ds\int_{X_{\alpha+\beta}(\A)}\int_{X_{2\alpha+\beta}(\A)}\int_{Z_U(\A)}f(\iota z_U v(0,y,z-by^2) w_\beta h)\overline{(\omega_{\psi_S}(h)\Phi)(y)\psi(S(z-by^2))}d{z_U}
dydz.
\end{eqnarray*}
By substituting $z$ with $z+by^2$, we see 
$R(\ell_b h;f,\Phi)=\ve R(h;f,\Phi)$. 

Finally, as for $\ell_a$, we have 
\begin{eqnarray*}
&& R(\ell_a h;f,\Phi)= \\
&& \ve\ds\int_{X_{\alpha+\beta}(\A)}\int_{X_{2\alpha+\beta}(\A)}\int_{Z_U(\A)}
f(m(\begin{pmatrix}
a& 0 \\
0 & 1
\end{pmatrix}
)
\iota z_U(\frac{z_1}{a},az_2)v(0,ay,z)w_\beta h) \overline{(\omega_{\psi_S}(p_a h)\Phi)(y)\psi(Sz)}dz_U
dydz.
\end{eqnarray*}
By using (\ref{g2-prin}) and (\ref{weil-action2}) and changing the variables 
as $(z_1,z_2,y)\mapsto (az_1,z_2/a,y/a)$, we have 
\begin{eqnarray*}
&& R(\ell_a h;f,\Phi)=\ve \delta_P(a)^{\frac{1}{2}}|a|^s \omega(a) |a|^{-1}
\frac{\gamma(-S)}{\gamma(-Sa)}|a|^{\frac{1}{2}}
R(h;f,\Phi) \\
&& \phantom{xxxxxxxxx} =\ve \delta_B(a)^{\frac{1}{2}}|a|^s \frac{\gamma(1)}{\gamma(a)}\omega(a)\chi_S(a) 
R(h;f,\Phi).
\end{eqnarray*}
Here, we used the formula $\gamma(-S)\gamma(a)=\langle -S,a \rangle 
\gamma(1)\gamma(-Sa)$ 
for the gamma constant.  
\end{proof}

\begin{remark} For non-archimedean place, the local representation theoretic analogue of 
Theorem \ref{Eisen-exp} has been proved by G. Savin in \cite[Theorem B.2.2]{Po-modular}.
\end{remark}

\section{Degenerate Whittaker functions}\label{DWF}
In this section, we consider a local representation theoretic analogue of $F_0$ 
(see (\ref{naiveFE}) or (\ref{po-exp})). 

\subsection{Degenerate principal series representations: The nonarchimedean case}\label{dnonar}
Let $p$ be a rational prime.  For a unitary character 
$\mu_p:\Q^\times_p\lra \C^\times$ and $s\in\Bbb C$, let us consider the degenerate principal series representation $I(s,\mu_p):={\rm Ind}^{G_2(\Q_p)}_{P(\Q_p)} \, (\mu_p\circ {\rm det}) |{\rm det}|^s$
consisting of any smooth $G_2(\Z_p)$-finite functions $\phi:G_2(\Q_p)\lra \C$ such that 
$$\phi(nmg)=\delta^{\frac{1}{2}}_P(m)\mu_p(\det(m))|\det(m)|^s \phi(g),\ nm\in P(\Q_p)=N(\Q_p)M(\Q_p),\ 
g\in G_2(\Q_p)$$
where $\delta^{\frac{1}{2}}_P(m)=|\det(m)|^{\frac{3}{2}}_p$. Here we write $\det(m)$ for $\det(A)$ when 
$m=m(A),\ A\in \GL_2(\Q_p)$. 
In terms of notations in \cite{M}, it is $I_\alpha(-s,\mu_p)$. 
We denote by $\bf 1$ the trivial character of $\Q^\times_p$.

\begin{thm}\label{constituents}\cite[Theorem 3.1, p.472, Proposition 4.1, p.475, Proposition 4.3-(ii), p.476]{M} Let $s\in\Bbb R$. It holds that 
\begin{enumerate}
\item $I(0,\mu_p)$ is irreducible: 
\item $I(s,\mu_p)$ reduces if and only if $s=\pm\frac 12$, $\mu_p^2=\bf 1$, or $s=\pm\frac 32$, $\mu_p=\bf 1$, or $s=\pm\frac 12$, $\mu_p^3=\bf 1$;
\item When $\mu^2_p=\bf 1$, $I(\frac 12,\mu_p)$ has a unique maximal subrepresentation $A(|\cdot|^{\frac{1}{2}}_p \mu_p)$, and a unique irreducible quotient $J_\beta(1,\pi(1,\mu_p))$, i.e., 
$$0\lra A(|\cdot|^{\frac{1}{2}}_p \mu_p)\lra I(\tfrac 12,\mu_p)\lra J_\beta(1,\pi(1,\mu_p))\lra 0,
$$ 
and
in the notation of \cite[Proposition 4.1-(ii), Proposition 4.3-(ii)]{M}, 
$$A(|\cdot|^{\frac{1}{2}}_p \mu_p)
=\left\{\begin{array}{cc}
V & \text{if $\mu_p=\bf 1$}\\
J_\beta(\frac 12,{\rm St}_p\otimes\mu_p) & \text{if $\mu_p\neq \bf 1$},
\end{array}\right.
$$
where $V$ satisfies 
$$0\lra \pi(1)\lra V\lra J_\beta(\tfrac 12,{\rm St}_p)\lra 0.$$ 
\end{enumerate}
\end{thm}

\begin{prop}\label{Langlands} Let $I_\beta(s,\pi)=Ind_{Q(\Q_p)}^{G_2(\Q_p)} \pi\otimes |\det|^s$, where $\pi$ is a tempered representation of $GL_2(\Q_p)$, and $J_\beta(s,\pi)$ be its Langlands quotient.
Then $I(0,\mu_p)=J_\beta(\frac 12,\pi(\mu_p,\mu_p^{-1}))$.
\end{prop}
\begin{proof} Since $\mu_p\circ \det\hookrightarrow \mu_p|\ |^{-\frac 12}\otimes \mu_p|\ |^{\frac 12}$, 
$$I(0,\mu_p)\hookrightarrow Ind_{B(\Q_p)}^{G_2(\Q_p)}\, \mu_p|\ |^{-\frac 12}\otimes \mu_p|\ |^\frac 12\simeq Ind_{B(\Q_p)}^{G_2(\Q_p)}\, \mu_p^{-1}|\ |^{\frac 12}\otimes \mu_p^2,
$$
by $w_{3\alpha+\beta}$ in the notation of \cite{Z}. 
Now 
$$I_\beta(\frac 12,\pi(\mu_p,\mu_p^{-1}))=Ind_{B(\Q_p)}^{G_2(\Q_p)}\, \underline{\mu_p|\ |^{\frac 12}\otimes \mu_p|\ |^{-\frac 12}}\simeq Ind_{B(\Q_p)}^{G_2(\Q_p)}\, \mu_p^{-1}|\ |^{\frac 12}\otimes \mu_p^2
$$ 
in the notation of \cite{Z}. Since $I(0,\mu_p)$ is irreducible, our result follows.
\end{proof}

Let $\psi=\psi_p:\Q_p\lra \C^\times$ be the standard non-trivial additive character. 
For each $w\in W(\Q_p)$, we define $\psi_w(n)=\psi(\langle w,x \rangle)$ for $n=n(x,t)\in 
N(\Q_p)$ and $\psi_w$ is said to be generic if $q(w)\neq 0$. 
For any smooth representation $\Pi$ of $G_2(\Q_p)$, we put 
$$\Wh_{\psi_w}(\Pi):={\rm Hom}_{N(\Q_p)}(\Pi,\psi_w).
$$
The following claim is similar to \cite[Proposition 3.1]{KY2}. 

\begin{prop}\label{wh} Suppose $\psi_w$ is generic. Then it holds that 
\begin{enumerate}
\item ${\rm dim}\hspace{0.5mm}\Wh_{\psi_w}(I(s,\mu_p))\le 1$ for any unitary character $\mu_p$ of $\Q_p$ above and 
$s\in \C$; 
\item When $\mu^2_p=\bf 1$, if $\Wh_{\psi_w}(A(|\cdot|^{\frac{1}{2}}_p\mu_p))\neq 0$, the restriction map induces an isomorphism 
$$\Wh_{\psi_w}(I(\frac{1}{2},\mu_p))\stackrel{\sim}{\lra}\Wh_{\psi_w}(A(|\cdot|^{\frac{1}{2}}_p\mu_p)).
$$
In this case, we have ${\rm dim}\hspace{0.5mm}\Wh_{\psi_w}(I(s,\mu_p))=
{\rm dim}\hspace{0.5mm}\Wh_{\psi_w}(A(|\cdot|^{\frac{1}{2}}_p\mu_p))=1$.
\end{enumerate}
\end{prop}

\begin{proof}The first claim follows from \cite[Theorem 3.2, p.1311]{Karel}. 

For the second claim, 
let $0\lra A(|\cdot|^{\frac{1}{2}}_p\mu_p) \lra I(\frac{1}{2},\mu_p)\lra V'\lra 0$ be the exact sequence for some 
quotient $V'$. (By \cite[Proposition 4.1-ii),Proposition 4.3-(ii)]{M}, we can specify $V'$ but it is unnecessary for the argument below.) 
Then, by taking the Whittaker functor and using its exactness, we have  
$$0\lra  \Wh_{\psi_w}(V')\lra \Wh_{\psi_w}(I(\frac{1}{2},\mu_p))\lra \Wh_{\psi_w}(A(|\cdot|^{\frac{1}{2}}_p\mu_p))\lra 0$$
Then, the claim follows from this and the first claim with the assumption. 


\end{proof}

\subsection{Jacquet integrals and Siegel series}\label{JI}
For $z\in \C$, we define a function $\ve_z$ on $G_2(\Q_p)$ by 
$$\ve_z(g)=|\det(m)|^z_p,\ g=nmk\in G_2(\Q_p)=N(\Q_p)M(\Q_p)G_2(\Z_p).$$ 
For $\phi\in I(s,\mu_p)$ and $w\in W(\Q_p)$ with $q(w)\neq 0$, we define the Jacquet integral by 
$${\bf w}^{\mu_p,s,z}_w(\phi):=\int_{N(\Q_p)}(\phi\cdot \ve_z)(\iota n)\overline{\psi_w(n)}dn,\ 
\iota=w_{\beta\alpha\beta\alpha}w^{-1}_\beta,
$$
which is motivated by \cite[(16),\ p.292]{JR}.
It is absolutely convergent for ${\rm Re}(z)>\frac{3}{2}-{\rm Re}(s)$ and for each $s\in C$ 
one can check that it is a polynomial in $\C[p^{\pm z}]$ by using 
\cite[Corollary 3.6.1]{Karel}. Thus, we can substitute $z=0$ into 
${\bf w}^{\mu_p,s,z}_w(\phi)$. Then, we define, for ${\rm Re}(s)>-\frac{1}{2}$, 
\begin{equation}\label{wCF1}
\widetilde{{\bf w}}^{\mu_p,s}_w(\phi):=|q(w)|^{3/4}_p \frac{
L(s+\frac{1}{2},\mu_{p})L(s+\frac{3}{2},\mu_p)
L(2s+1,\mu_p) L(3s+\frac{3}{2},\mu_p)}{L(s+\frac{1}{2},\mu_{p,{E_w}})} {\bf w}^{\mu_p,s,0}_w(\phi) 
\end{equation}
where $L(s,\mu_p)=(1-\mu_p(p)p^{-s})^{-1}$ 
and $L(s,\mu_{p,{E_w}})$ is the $L$-function of the base change of $\mu_p$ 
to the cubic \'etale algebra 
$$E_w:=\left\{\begin{array}{cc}
\Q_p[x]/(f_w(x,1)) & \text{if $\deg_x(f_w(x,1))=3$} \\
\Q_p[x]/(f_w(x,1))\times \Q_p & \text{if $\deg_x(f_w(x,1))=2.$}
\end{array}\right.
$$ 
As for the factors in front of the Jacquet integral, we follow the normalization of Eisenstein series in \cite[p.226-237]{Xiong}. We also remark that 
in \cite{JR}, the induced representation is unnormalized while ours is normalized and 
the variable $s$ there should be replaced with $\ds\frac{s}{3}+\frac{1}{2}$, as in 
\cite{Xiong}, to get our setting.
Then, finally, we write 
\begin{equation}\label{wCF2}
\widetilde{{\bf w}}^{\mu_p}_w(\phi):=
\left\{\begin{array}{cl}
\widetilde{{\bf w}}^{\mu_p,0}_w(\phi) & (\phi\in I(0,\mu_p)) \\
\widetilde{{\bf w}}^{\mu_p,\frac{1}{2}}_w(\phi) & 
(\phi\in I(\frac{1}{2},\mu_p) \text{ with $\mu^2_p=\bf 1$})
\end{array}\right..
\end{equation} 

The following is an analogue of \cite[Lemma 3.3, p.590]{KY2}.
\begin{lem}\label{est-coe} Keep the notations in $($\ref{wCF2}$)$.
Assume $q(w)\neq 0$. 
Then, there exist constants $C_1,C_2>0$ depending only on $\phi$ such that 
$$|\widetilde{{\bf w}}^{\mu_p}_w(\phi)|\le C_1 \max\{|q(w)|^{\frac{3}{4}}_p,\ |q(w)|^{-C_2}_p\}.$$
\end{lem}
\begin{proof}We borrow an idea of the proof in \cite[Lemma 3.3]{Yamana17} and 
an argument in \cite[p.53]{Ik17}, but the proof here is slightly different and applicable to many cases where the unipotent radical is not abelian. 

We need a bound on $|{\bf w}^{\mu_p,0,0}(\phi)|$ 
(the case $|{\bf w}^{\mu_p,\frac{1}{2},0}(\phi)|$ is similarly handled and omitted). 
By the argument in \cite[Section 1 and 2]{CS}, ${\bf w}^{\mu_p,s,0}_w(\phi)$ is 
a holomorphic function in $s\in \C$ (we note that the Jacquet integral here is interpreted as a Cauchy principal value integral, following \cite{Karel} and \cite{CS}. 
Hence, in the region of absolute convergence for $s$, it coincides with the local integral considered in \cite{JR}.). 
Fix a positive real number $\sigma>\frac{3}{2}$ and define 
$D_\sigma:=\{s\in \C\ |\ -\sigma \le {\rm Re(s)\le \sigma}\}$. 
Applying the maximal modulus principle to ${\bf w}^{\mu_p,s,0}_w(\phi)$ on $D_\sigma$ as a function in $s$, we have 
$$|{\bf w}^{\mu_p,0,0}_w(\phi)|\le \max_{{\rm Re}(s)=
\pm \sigma}\{|{\bf w}^{\mu_p,s,0}_w(\phi)|\}.$$
When ${\rm Re}(s)=\sigma$, as in the proof in \cite[Lemma 3.3]{Yamana17}, 
there exists a constant $C_1>0$ depending on $\phi$ (and $\sigma$) such that 
$|{\bf w}^{\mu_p,s,0}(\phi)|\le C_1$. 

Next, we consider the case when ${\rm Re}(s)=-\sigma$. 
Let 
$M(s):I(s,\mu_p)\lra I(-s,\mu^{-1}_p),\ f\mapsto [g\mapsto \ds\int_{N(\Q_p)}f(\iota n g)dn$]. 
It is well-known that for each $\phi\in I(s,\mu_p)$, 
$M(s)(\phi)$ extends meromorphically on the  whole space in $s\in \C$. 
By Proposition \ref{wh}-(1), 
there exists a meromorphic function $\kappa_w(s)$ on $s\in \C$ such that 
$${\bf w}^{\mu^{-1}_p,-s,0}_w\circ M(s)=\kappa_w(s) {\bf w}^{\mu_p,s,0}_w.$$ 
Let us evaluate $\kappa_w(s)$.  
Let $C_W$ be a complete system of representatives of $\{w\in W(\Q_p)\ |\ q(w)\neq 0\}/
M(\Q_p)$ where $M(\Q_p)$ acts as the adjoint action. It is well-known that 
$C_W$ is finite (\cite[Section 2.4]{JR}) and the upper bound of its cardinality is independent of $p$.
We can write $w=\det(m)^2\rho_3(m^{-1})w_0$ for some $m\in M(\Q_p)$ and $w_0\in C_W$ so that $\langle w,n \rangle=
\langle w_0,{\rm Ad}(m)n \rangle$ and $q(w)=\det(m)^2 q(w_0)$. 
By transformation law, we have 
$${\bf w}^{\mu^{-1}_p,-s,0}_{w_0}\circ M(s)(m\cdot f)=\mu^2_p(\det(m)) |\det(m)|^{2s}_p \kappa_{w}(s) 
{\bf w}^{\mu_p,s,0}_{w_0}(m\cdot f),\ f\in I(s,\mu_p).$$ 
Thus, we have 
$$\kappa_w(s)=\mu^{-2}_p(\det(m)) |\det(m)|^{-2s}_p\kappa_{w_0}(s)=
\mu^{-2}_p(\det(m)) |q(w)|^{-s}_p(|q(w_0)|^{s}_p\kappa_{w_0}(s))$$
and then,  
\begin{eqnarray*}
&& {\bf w}^{\mu_p,-s,0}_w(\phi)=\kappa^{-1}_w(-s)
{\bf w}^{\mu^{-1}_p,s,0}_w\circ M(-s)(\phi) \\
&& \phantom{xxxxxxxxx}=|q(w)|^{-s}_p\{\mu_p(\det(m)^2) (|q(w_0)|^{s}_p\kappa_{w_0}(-s)^{-1}){\bf w}^{\mu^{-1}_p,s,0}_w\circ M(-s)(\phi)\}.
\end{eqnarray*}
Since $\mu_p$ is unitary, $|\mu_p(\det(m)^2)|=1$. 
The set consisting of all poles of $\kappa_{w_0}(-s)^{-1}$ for any 
$w_0\in C_W$, and the normalizing factor of $\widetilde{{\bf w}}^{\mu_p,s}_w(\phi)$ and $M(-s)$, is finite. Therefore, one can re-choose 
$\sigma>\frac{3}{2}$ if necessary, so that any $s\in \C$ with ${\rm Re}(s)=\sigma$ does not 
contribute to any such poles. Thus, the claim follows from the previous argument. 
\end{proof}

\begin{lem}\label{transformation1}For $w\in W(\Q_p)$ and ${\rm Re}(s)>-\frac{1}{2}$, 
it holds that 
\begin{enumerate}
\item the functional $\widetilde{{\bf w}}^{\mu_p,s}_w\in \Wh_{\psi_w}(I(s,\mu_p))$ is non-zero 
and thus, ${\rm dim}\hspace{0.5mm}\Wh_{\psi_w}(I(s,\mu_p))=1$.  
Further, the restriction of $\widetilde{{\bf w}}^{{\bf 1}}_w$ to $A(|\cdot|^{\frac{1}{2}}_p\mu_p)$ is also 
non-zero if $\Wh_{\psi_w}(A(|\cdot|^{\frac{1}{2}}_p\mu_p))\neq 0$; 
\item  for any $m\in M(\Q_p),\ n\in N(\Q_p)$ and $\phi\in I(s,\mu_p)$, 
$$\widetilde{{\bf w}}^{\mu_p,s}_w(nm\cdot \phi)= \psi_w(n)\mu_p(\det(m))^{-1}|\det(m)|^{-s}
\widetilde{{\bf w}}^{\mu_p,s}_{(\det(m))^2\rho_3(m^{-1})w}(\phi).$$
\end{enumerate}
\end{lem}
\begin{proof}The first claim is proved by choosing a section $\phi$ suitably and Proposition \ref{wh}-(1). It is standard 
and thus omitted. 
For the second claim, the action of $n$ is easy to handle. 
Therefore, we only check the action of $m$. 
If we write $m=m(
\begin{pmatrix}
a & b \\
c & d
\end{pmatrix}
)$, then 
\begin{equation}\label{imi}
\iota m\iota^{-1}=m((ad-bc)^{-1}
\begin{pmatrix}
a & -b \\
-c & d
\end{pmatrix}
).
\end{equation}
Thus, $\det \iota m\iota^{-1}=\det(m)^{-1}$ and it yields 
$\phi(\iota n m)=\mu_p(\det(m))^{-1}|\det(m)|^{-s-3/2}\phi(\iota m n')$ where 
$n'=m^{-1}nm$.
By (\ref{action1}), 
\begin{eqnarray*}
&& \psi_w(n)=\psi(\langle w,n \rangle)= 
\psi(\langle w,mn'm^{-1} \rangle) = \psi(\langle w,\det(m)^{-1}\rho_3(m)n \rangle) \\
&& \phantom{xxxxx} =\psi(\langle \det(m)^2\rho_3(m^{-1})w,n' \rangle)=
\psi_{\det(m)^2\rho_3(m^{-1})w}(n').
\end{eqnarray*}
Further, $dn=d(mn'm^{-1})=\delta_P(m)dn'$ by (\ref{action1}) again 
and $$|q(\det(m)^2\rho_3(m^{-1})w)|^{3/4}_p=|\det(m)|^{3/2}_p |q(w)|^{3/4}_p.$$
Summing up, we have
\begin{eqnarray*}
&& \widetilde{{\bf w}}^{\mu_p,s}_w(m\cdot \phi)
=\mu_p(\det(m))^{-1}|\det(m)|^{-s-3/2}\delta_P(m)|\det(m)|^{-3/2}_p 
\widetilde{{\bf w}}^{\mu_p,s}_{(\det(m))^2\rho_3(m^{-1})w}(\phi) \\
&& \phantom{xxxxxxxxx} =\mu_p(\det(m))^{-1}|\det(m)|^{-s}
\widetilde{{\bf w}}^{\mu_p,s}_{(\det(m))^2\rho_3(m^{-1})w}(\phi).
\end{eqnarray*}
\end{proof}

\subsection{Degenerate principal series representations: The archimedean case} 
Recall the notations in Section \ref{QDS}. We regard the quaternionic discrete series 
representation $D_k$ as a submodule of 
$$\Pi_k:={\rm ind}^{G_2(\R)}_{P(\R)}\lambda_k=
{\rm Ind}^{G_2(\R)}_{P(\R)}\sgn^k |\det|^{k-\frac{1}{2}},\ k\ge 2,
$$
where the latter is a normalized induced representation. 

Let $\psi=\psi_\infty=\exp(2\pi\sqrt{-1}\ast):\R\lra \C^\times$ be the standard non-trivial additive character. 
For each $w\in W(\R)$, we define $\psi_w(n)=\psi_\infty(\langle w,x \rangle)$ for $n=n(x,t)\in 
N(\R)$ and $\psi_w$ is said to be generic if $q(w)\neq 0$. We say $w\in W(\R)$ is generic if $q(w)\ne 0$, or equivalently $\psi_w$ is generic. 
The following claim is due to Wallach \cite[Theorem 13, p.301 and Theorem 16, p.302]{Wallach} (see also \cite[Proposition 6.1]{GGS}) though 
the sign is opposite in $q(w)$ because $q(w)=-\Delta(\psi_w)$ in the notation there.   
\begin{prop}\label{wh-arch} Suppose $\psi_w$ is generic for $w\in W(\R)$.  Then, it holds that 
\begin{enumerate}
\item ${\rm dim}\hspace{0.5mm}{\rm Hom}_{N(\R)}(\Pi_k,\C(\psi_w))=1$ if $q(w)\neq 0$; 
\item ${\rm dim}\hspace{0.5mm}{\rm Hom}_{N(\R)}(D_k,\C(\psi_w))=\begin{cases} 0, &\text{if $q(w)>0$}\\ 
1, &\text{if $q(w)<0$}\end{cases}$.
\end{enumerate}
\end{prop}
Recall $K_\infty\simeq (\SU(2)_{\beta_0}\times\SU(2)_\alpha)/\mu_2$ where 
we insert subscripts into $\SU(2)$-factors to indicate the roots. 
\begin{prop}\label{k-type-dp}As a representation of $K_\infty$, it holds that 
$$\Pi_k|_{K_\infty}\simeq \bigoplus_{m,n\in \Z_{\ge 0}\atop m:\text{even}}T_{m,n},\quad 
T_{m,n}:=(\Sym^m \C^2)^{\oplus(m+1)}\boxtimes \Sym^{2n}\C^2. 
$$
\end{prop}
\begin{proof}We see that $G_2(\R)=P(\R)K_\infty$ and $P(\R)\cap K_\infty=M(\R)\cap K_\infty 
\simeq \{1_2\}\times U(1)_\alpha/\mu_2\subset  (\SU(2)_\beta\times \SU(2)_\alpha)/\mu_2$ where $U(1)_\alpha$ is diagonally embedded into 
$\SU(2)_\alpha$ as $u\mapsto \diag(u,u^{-1})$. It follows from this that   
$$\Pi_k|_{K_\infty}\simeq {\rm Ind}^{K_\infty}_{P\cap K_\infty}1\simeq 
\Big({\rm Ind}^{\SU(2)_\beta}_{\{1_2\}}1\boxtimes 
{\rm Ind}^{\SU(2)_\alpha}_{U(1)_\alpha} 1\Big)^{\mu_2},
$$
where the last one is the $\mu_2$-fixed part of the representation of 
$\SU(2)_\beta\times\SU(2)_\alpha$. 
By the Peter-Weyl theorem, ${\rm Ind}^{\SU(2)_\beta}_{\{1_2\}}1\simeq 
\bigoplus_{m\ge 0}(\Sym^m \C^2)^{\oplus(m+1)}$. On the other hand, it is well-known that 
${\rm Ind}^{\SU(2)_\alpha}_{U(1)_\alpha}1\simeq L^2(\SO(3)/\SO(2))
\simeq \bigoplus_{n\ge 0}\Sym^{2n}\C^2$. To have the same $\mu_2$-action on 
both factors, $m$ has to be even. The claim follows.
\end{proof}

The following result is well-known in more general setting by 
\cite[Chapter III]{KS}. Recall $W(\R)_{\ge 0}$ is the set of $w\in W(\R)$ with $w\geq 0$, i.e., all roots of the polynomial $f_w(z,1)$ are real.
\begin{prop}\label{conv-Jacquet-int}
Let $\mu:\R^\times\lra \C^\times$ be a unitary character, $w\in W(\R)_{\ge 0}$ and $s\in \C$. Then, 
for any $\phi\in {\rm Ind}^{G_2(\R)}_{P(\R)}\mu(\det)|\det|^s$, the integral
\begin{equation}\label{Ws}
W^{(s)}_w(g;\phi):=|q(w)|^{-\frac{s}{2}+\frac{3}{4}}\int_{N(\R)}\phi(\iota ng)\overline{\psi_w(n)}dn,\ g\in G_2(\R),
\end{equation}
converges absolutely if ${\rm Re}(s)>0$. 
\end{prop}

\begin{cor}\label{trans-arch} Let $k\ge 1$ and $w\in W(\R)_{\ge 0}$. 
Then, for any $\phi\in\Pi_k$, 
$W^{(k-\frac{1}{2})}_w(g;\phi),\ g\in G_2(\R)$ converges absolutely   and it holds that 
\begin{equation}\label{inf-trans}
W^{(k-\frac{1}{2})}_w(nmg;\phi)=\psi_w(n)\sgn(\det m)^k W^{(k-\frac{1}{2})}_{\det m^2 \rho_3(m^{-1})w}(g;\phi). 
\end{equation}
\end{cor}
\begin{proof}
Since $k-\frac{1}{2}>0$, the convergence is clear. 
The latter claim is proved as in the proof of Lemma \ref{transformation1}-(2).
\end{proof}

\begin{cor}\label{w-funct}Let $k\ge 1$. Assume $w\in W(\R)$
is generic and $q(w)<0$. Then, the functional 
\begin{equation}\label{w-functional}\Pi_k\lra \C,\ \phi\mapsto W^{(k-\frac{1}{2})}_w(1;\phi)
\end{equation}
is a generator of ${\rm Hom}_{N(\R)}(\Pi_k,\C(\psi_w))$. 
\end{cor}
\begin{proof} Non-vanishing of the functional over $\Pi_k$ can be checked by 
a standard argument (cf. \cite[the proof of Proposition 7.1, line -7 in p.141]{Wa88}).
Then, by invoking Proposition \ref{wh-arch}, we have the claim. 
\end{proof}
\begin{lem}\label{irrecomp}Assume $k\ge 2$. Then, $\Pi_k$ has possibly three 
irreducible components including $D_k$. Furthermore, any irreducible component $V$ 
except for $D_k$, it holds ${\rm Hom}(V,\C(\psi_w))=0$ for any generic $w\in W(\R)$ with $q(w)<0$. 
\end{lem}
\begin{proof}
Let $W(\R)^{\rm gen}$ be the set of all $w\in W(\R)$ such that $q(w)\neq 0$ or 
equivalently $\psi_w$ is generic by definition.  
Recall the adjoint action of $M(\R)$ on $W(\R)$ preserves  $W(\R)^{\rm gen}$.  
There are exactly two orbits of $W(\R)^{\rm gen}$ such that 
a representative $w$ satisfies $q(w)>0$ or $q(w)<0$ respectively (see Proposition \ref{wh-arch}). 
The quaternionic discrete series representation $D_k$ is supported 
in $\psi_w$ for any $w\in W(\R)^{\rm gen}$ with $q(w)<0$. 
On the other hand, by Proposition \ref{wh-arch}, there exists an 
irreducible constituent of $\Pi_k$, say $D'_k$, which is supported in  
$\psi_w$ for any $w\in W(\R)^{\rm gen}$ with $q(w)>0$. 

On the other hand, by \cite[Lemma 2.3.4]{Gomez} and 
using the fact $P(\R)\bs G_2(\R)/P(\R)=\{1,w_\beta,w_{\beta\alpha\beta},\iota\}$ 
\cite[(1), p.260]{GJ}, we see that 
${\rm dim}_\C{\rm End}_{G_2(\R)}(\Pi_k)\le 3$. Since $D_k\not\simeq D'_k$, 
${\rm dim}_\C{\rm End}_{G_2(\R)}(\Pi_k)\ge 2$. 
Therefore, ${\rm End}_{G_2(\R)}(\Pi_k)$ is isomorphic to either of 
 $\C^2,\ \C^3$, or $\Big\{
 \begin{pmatrix}
 a & b \\
 0 & c
 \end{pmatrix}
 \Big|\ a,b,c\in \C \Big\}$. The claim follows from this. 
\end{proof}

\begin{prop}\label{nvJI1}Assume $k\ge 2$. Assume $w\in W(\R)$ is generic and $q(w)<0$.  
Then, the restriction of the functional {\rm (}\ref{w-functional}{\rm )} to 
$D_k$ is not identically zero. In particular, the functional gives a generator of 
${\rm Hom}_{N(\R)}(D_k,\C(\psi_w))$. 
Further, for each non-zero $\phi \in D_k$ as above, $W^{(k-\frac{1}{2})}_w(g;\phi)$ is not identically zero.   
\end{prop}
\begin{proof}The claim follows from Lemma \ref{irrecomp} and Proposition \ref{wh-arch}. 
\end{proof}

\begin{remark}\label{nvJI2}
If $k\ge 2$ is even, we can give another proof of the above proposition by using 
a global method without using Lemma \ref{irrecomp}. Perhaps, it may be useful for another setting: 
Let $E_{k}$ be a quaternionic Eisenstein series of weight $k$ 
considered in \cite[Section 9]{GGS}. Note that ``$2k$'' in loc.cit. is ``$k$'' in 
our notation. By \cite[Theorem 3.1]{Gan} and \cite[Corollary 1.2.3, p.1216]{PoFE}, there exists a generic $w_0\in W(\Q)$ with $q(w_0)<0$ such that 
the $w_0$-th Fourier 
coefficient of $E_{k}$ is non-zero. 
Suppose the restriction of the functional {\rm (}\ref{w-functional}{\rm )} to 
$D_k$ is identically zero. Then,   
as observed in \cite[p.130, line -6 to the bottom]{GGS}, all Fourier coefficients of 
$E_k$ vanish. Thus, we have a contradiction. 
The latter claim follows from the irreducibility of  
$D_k$ and the equality $W^{(k-\frac{1}{2})}_w(g;\phi_{\infty,I})=W^{(k-\frac{1}{2})}_w(1;g\cdot\phi_{\infty,I})$ for any $g\in G_2(\R)$. 
\end{remark}

\subsection{Degenerate Whittaker functions: The archimedean case}\label{dwf-arch}
Recall the minimal $K_\infty$-type $V_k\subset D_k$ from Section \ref{QDS}. 
Fix the basis $\{e_v=x^{k+v}y^{k-v}\}_{-k\le v\le k}$ of $V_k$ as in \cite[p.391]{Po} and  
we denote by $\{e^\vee_v\}_{-k\le v\le k}$ its dual basis. 
Let $\langle \ast,\ast \rangle:V_k\times V^\vee_k\lra \C$ be the natural pairing. 
For  each $-k\le v,v'\le k$, we define the section $f_{v,v'}$ of   
$D_k\subset {\rm Ind}^{G_2(\R)}_{P(\R)}\sgn^k |\det|^{k-\frac{1}{2}}$ (the normalized induction) by 
\begin{equation}\label{section1}
f_{v,v'}(g)=\sgn(\det(m))^k|\det(m)|^{k+1}
 \langle\tau_k(\kappa)e_v, e^\vee_{v'}\rangle,\ g=mn \kappa\in M(\R)N(\R)K_\infty.
\end{equation}
This is well-defined since $M(\R)\cap K_\infty=\{m(\pm I_2)\}$. 
For each non-empty subset $I$ of $\{v\in \Z\ |\ -k\le v\le k\}$, 
put $\phi_{\infty,I}:=\ds\sum_{v\in I}f_{v,v}e^\vee_v\in D_k$ defined by using 
$f_{v,v}$. 
We define the $V_k^\vee$-valued function
\begin{equation}\label{section2}
W^{(k-\frac{1}{2})}_w(g_\infty;\phi_{\infty,I})=\ds\sum_{v\in I}W^{(k-\frac{1}{2})}_w(g_\infty;f_{v,v})e^\vee_{v},\ 
g_\infty\in G_2(\R),
\end{equation}
which plays a role in the Fourier expansion of modular forms on $G_2$ and 
it is an analogue of exponential functions 
(see \cite[Section 3.3]{IY} and \cite[Section 3.4, Theorem 3.4.1]{Liu}). 

\begin{remark}\label{rel-to-pol}
Let $k\ge 2$ be any integer.  
Let $I=\{v\in\Z\ |\ -k\le v\le k\}$ and $w\in W(\R)_{\ge 0}$ 
with $q(w)<0$. 
Put $\widetilde{W}^{(k-\frac{1}{2})}_w(g_\infty;\phi_{\infty,I}):=|q(w)|^{\frac{k+1}{2}}
W^{(k-\frac{1}{2})}_w(g_\infty;\phi_{\infty,I})$. Then, for 
$g_\infty=n_\infty m_\infty k_\infty \in G_2(\R)=P(\R)K_\infty$, 
$\widetilde{W}^{(k-\frac{1}{2})}_w(g_\infty;\phi_{\infty,I})$ and 
$e^{2\pi \sqrt{-1}\langle w,n_\infty \rangle}\tau^\vee_k(k_\infty)
\mathcal{W}_w(m_\infty)$ appeared in (\ref{po-exp}) has the same 
transformation law in the left $P(\R)$-action.  By Proposition \ref{nvJI1},  
there exists a non-zero constant $c_{\infty,w}$ depending on $w$ and $k$ such that 
$$\widetilde{W}^{(k-\frac{1}{2})}_w(g_\infty;\phi_{\infty,I})=c_{\infty,w} e^{2\pi \sqrt{-1}\langle w,n_\infty \rangle}\tau^\vee_k(k_\infty)
\mathcal{W}_w(m_\infty).$$ 
Since the set $\{w\in W(\R)_{\ge 0}\ |\ q(w)<0\}$ has a single $M(\R)$-orbit and both sides 
have the same transformation law in $M(\R)$, $c_{\infty,w}=:c_\infty$ is, in fact, a uniform constant.   
\end{remark}

\section{Fourier expansion of quaternionic modular forms: Proof of Theorem \ref{exp-thm}}\label{FS} 
In this section, we will prove Theorem \ref{exp-thm}. 
Let $f$ be a new form in $S_{2k}(\Gamma_0(C))^{{\rm new,ns}}$ and 
$\Pi(f)=\otimes'_{p<\infty}\Pi_p=\Pi_{\f}\otimes D_k$ be the
admissible representation of $G_2(\A)$ as in Section \ref{intro}. Assume (\ref{assump}). Then, we have an intertwining map 
$$\Pi(f)\hookrightarrow \mathcal{A}_{{\rm cusp}}(G_2(\Q)\bs G_2(\A)),\ \phi\mapsto F_f(\ast;\phi).$$
Let us consider the constant term $F_f(g;\phi)_0:=\ds\int_{Z_N(\Q)\bs Z_N(\A)}
F_f(zg;\phi)dz$ along $Z_N$. 

Recall the finite set $S(\pi_{\f})$ of rational primes in Section \ref{intro} such that 
$$\Pi_p=\begin{cases} \text{$I(0,\mu_p)$ with 
a unitary character $\mu_p:\Q^\times_p\lra \C^\times$}, &\text{if $p\not\in S(\pi_{\f})\cup
\{\infty\}$}\\ 
\text{$A(|\cdot|^{\frac{1}{2}}\mu_p)\subset I(\frac{1}{2},\mu_p)$ with $\mu^2_p=\bf 1$}, &\text{if $p\in S(\pi_{\f})$}
\end{cases}.
$$ 
(See the notation in Section \ref{dnonar}.)  
Put 
\begin{equation}\label{mu-finite}
\mu_{\f}=\otimes'_{p\not\in S(\pi_{\f})\cup\{\infty\}}\mu_p
\otimes_{p\in S(\pi_{\f})}(|\ast|^{\frac{1}{2}}_p\mu_p).
\end{equation}
For $w\in W(\Q)$ and a distinguished vector 
$\phi_{\f}=\otimes'_{p<\infty}\phi_p\in \Pi_{\f}$ such that $\phi_p$ is a $G_2(\Z_p)$-fixed 
vector with $\phi_p(1)=1$ for all but finitely many $p$, put 
\begin{equation}\label{finite-coeff}
\widetilde{{\bf w}}_{w}(\phi_{\f}):=\prod_{p<\infty}\widetilde{{\bf w}}^{\mu_p}_w(\phi_p)
=\prod_{p\not\in S(\Pi_{\f})\cup\{\infty\}}\widetilde{{\bf w}}^{\mu_p,0}_w(\phi_p)\times 
\prod_{p\in S(\Pi_{\f})}\widetilde{{\bf w}}^{\mu_p,\frac{1}{2}}_w(\phi_p).
\end{equation}
By \cite[Theorem 2.4-(4), p.292-293]{JR} or \cite[Theorem 1.1]{Xiong}, 
{$\widetilde{{\bf w}}^{\mu_p,0}_w(\phi_p)$ is trivial 
for all but finitely many $p\not\in S(\Pi_{\f})\cup\{\infty\}$. Thus, the above infinite product is well-defined. 

For each $w\in W(\Q)$, we denote by $x_{3\alpha+\beta}(w)$, the $x_{3\alpha+\beta}$-component of $w$ 
according to the decomposition $W=X_\beta X_{\alpha+\beta}X_{2\alpha+\beta}
X_{3\alpha+\beta}$. 
Put $\phi:=\phi_{\f}\otimes \phi_{\infty}$ where $\phi_\infty:=\phi_{\infty,I}$ 
with $I=\{v\ |\ -k\le v\le k\}$ is defined in the previous section. 
Then, $F_f(\ast;\phi)$ is a quaternionic modular forms of weight $k$. 
Let $F_f(g;\phi)_0=\ds\int_{Z_N(\Q)\bs Z_N(\A)}F_f(zg;\phi)dz$ for $g\in G_2(\A)$. By  Remark \ref{rel-to-pol} and a robust theory of Pollack (\ref{po-exp}),  we have 
$$F_f(g;\phi)_0=\sum_{w\in W(\Q)_{\ge 0}\atop q(w)<0}F_f(g;\phi)_{\psi_w},\ F_f(g;\phi)_{\psi_w}=\int_{N(\Q)\bs N(\A)}F_f(ng;\phi)\overline{\psi_w(n)}dn.$$ 
We have only to focus on $w$ such that $F_f(1;\ast)_{\psi_w}$ is not identically zero on $\Pi(f)$. 
For any place $v$, put $\phi^{(v)}=\otimes'_{p\neq v}\phi_p$.  
Since the non-zero functional $\Pi_v\ni \phi'_v\mapsto F_f(1;\phi^{(v)}\otimes \phi'_v)_{\psi_w}$ belongs to $\Wh_{\psi_w}(\Pi_v)$, 
if $F_f(g;\ast)_{\psi_w}\neq 0$, 
the assumption in Proposition \ref{wh}-(2) is fulfilled for such a $w$. 
By Proposition \ref{wh}, Lemma \ref{transformation1}-(1), Proposition \ref{wh-arch}, Remark \ref{rel-to-pol}, we have 
$$F_f(g;\phi)_0=\sum_{w\in W(\Q)_{\ge 0}\atop q(w)<0} 
C^{\mu_{\f}}_{w}(F_f)\Big(\prod_{p<\infty} \widetilde{{\bf w}}^{\mu_p}_{{\rm Ad}(w_\alpha)w}(g_p\cdot \phi_p)\Big)W^{(k-\frac{1}{2})}_{{\rm Ad}(w_\alpha)w}(g_\infty\cdot\phi_\infty),\ 
g=g_{\f}g_\infty\in G_2(\A)$$
for some constants $C^{\mu_{\f}}_w(F_f)$. 
Then, by Proposition \ref{expansion1}, we can recover the expansion (\ref{fseries}) 
for above $\phi$ and the general case follows from the left action of $G_2(\A)$ and 
irreducibility of $\Pi(f)$. 
This proves Theorem \ref{exp-thm}.   

By $M(\Q)$-left invariance of $F_f(\ast;\phi)_0$ which follows from the automorphy of $F$ and the transformation law in Lemma \ref{transformation1}-(2) and Corollary \ref{trans-arch}, we can easily check the following property:
\begin{equation}\label{cw}
C^{\mu_{\f}}_w(F_f) \mu_{\f}(\det(m')^{-1}) 
\sgn(\det(m'))^k=C^{\mu_{\f}}_{\det(m')^2 \rho_3(m'^{-1})w}(F_f),\ m':={\rm Ad}(w_\alpha)m,\ m\in M(\Q). 
\end{equation} 

Finally, we explain how to choose a distinguished vector $\phi$ so that 
$F_f(g;\phi)$ is fixed by $\ds\prod_{p\nmid C}G_2(\Z_p)\times \prod_{p|C}\G_P(\Z_p)$.  
If $C$ is square-free, then $S(\pi_{\f})=\{p|C\}$ and  $\mu_p$ is unramified for any rational prime $p$ 
(cf. \cite[Proposition 2.8-(2)]{LW}). If $p\nmid C$, choose $\phi_p\in \Pi^{G_2(\Z_p)}_p$ such that $\phi_p(1)=1$.
Let $r_\alpha(\Pi_p)$ be the Jacquet modules of $\Pi_p$ with respect to $M$ which is explicitly given in \cite[(4.20), p.477]{M}. 
If $p|C$, then 
$$r_\alpha(\Pi_p)=\mu_p(\det) |\det|^{\frac{1}{2}}.$$
Since $\Pi^{\G_P(\Z_p)}_p\stackrel{\sim}{\lra} r_\alpha(\Pi_p)^{M(\Z_p)}$ by 
\cite[Theorem 2.1]{MY}, we may choose $\phi_p\in \Pi^{\G_P(\Z_p)}_p$ 
corresponding to a spherical non-zero vector of $r_\alpha(\Pi_p)^{M(\Z_p)}$.

\section{Fourier-Jacobi expansions along $\widetilde{L^{{\rm ss}}}\ltimes U/Z_U$ in adelic setting}\label{FJE}
Let us recall the notations in Section \ref{FJEES}. In this section, 
we consider the Fourier-Jacobi expansions along $\widetilde{U}$ in adelic setting. 
\subsection{The non-archimedean case}\label{nac}
Let $p$ be a rational prime. 
Let $I(s,\mu_p)$ be the degenerate principal series representation in Section \ref{dnonar}. 
Note that $\mu_p$ is unitary in our setting. 
Let $\widetilde{\SL_2(\Q_p)}$ be the metaplectic double cover of $\sL(\Q_p)=\SL_2(\Q_p)$.
For a character $\delta_p:\Q^\times_p\lra \C^\times$ and $s\in\Bbb C$, let 
$\widetilde{I}^{\psi_p}_1(s,\delta_p)$ be the degenerate principal series representation of $\widetilde{\SL_2(\Q_p)}$  which is defined similarly as in 
(\ref{principal-series}). 

For each $\phi \in I(s,\mu_p)$, $\Phi\in \mathcal{S}(X_\alpha(\Q_p))$, $S\in \Q^\times_p$, 
and $h'=(h,\ve)\in \widetilde{\SL_2(\Q_p)}$, we define the integral 
\begin{eqnarray}\label{beta-local} 
&& \beta^{\psi_p}_S(h';\phi\otimes\overline{\Phi}):=
\frac{L(s+\frac{1}{2},\mu_p)L(s+\frac{3}{2},\mu_p)
L(2s+1,\mu_p) L(3s+\frac{3}{2},\mu_p)L(s+\frac{1}{2},\mu_p\chi_{S,p})}
{L(s+\frac{1}{2},\mu_{p,E_w})L(2s+1,\mu^2_p)} \nonumber  \\
&& \phantom{xxxxxx} \times \int_{X_{\alpha+\beta}(\Q_p)}\int_{X_{2\alpha+\beta}(\Q_p)}\int_{Z_U(\Q_p)}
\phi(\iota w_\beta z_U v(y,0,z) h)\overline{(\omega_{S,p}(v(y,0,z)h')\Phi)(0)}
dz_U dy dz 
\end{eqnarray}
where $\chi_{S,p}(a)=\langle -S,a \rangle_{p},\ a\in \Q^\times_p$ is 
defined by using the local quadratic Hilbert symbol $\langle \ast,\ast \rangle_{p}$ on 
$\Q^\times_p\times \Q^\times_p$. 
This is a local analogue of ``$R(h;f,\Phi)$'' in Theorem \ref{Eisen-exp} up to local 
$L$-factors. 

\begin{prop}\label{beta-intertwining} Keep the notations as above. 
Then, $\beta^{\psi_p}_S(h';\phi\otimes\overline{\Phi})$ is absolutely convergent 
if ${\rm Re}(s)>-\frac{2}{3}$ and it yields a $V(\Q_p)$-invariant 
and $\widetilde{\SL_2(\Q_p)}$-equivariant $\C$-bilinear map 
$$\beta^{\psi_p}_S:I(s,\mu_p)\otimes_\C \mathcal{S}(X_\alpha(\Q_p))
\lra \widetilde{I}^{\psi_p}_1(s,\mu_p\chi_{S,p}).$$
Namely, $\beta^{\psi_p}_S(v h';\phi\otimes\omega_{S,p}(\gamma)\overline{\Phi})
=\widetilde{I}^{\psi_p}_1(s,\mu_p\chi_{S,p})(\gamma)
\beta^{\psi_p}_S(h';\phi\otimes\overline{\Phi})$ for 
any $\gamma\in \widetilde{\SL_2(\Q_p)}$ and $v\in V(\Q_p)$. 
\end{prop}
\begin{proof} By a similar argument as in the proof of Theorem \ref{Eisen-exp} 
(or \cite[p.243]{KY1}), for $h'=(h,\ve)\in \widetilde{\SL_2(\Q_p)}$, 
we have 
\begin{eqnarray}\label{some-imp}
 \int_{X_{\alpha+\beta}(\Q_p)}\int_{X_{2\alpha+\beta}(\Q_p)}\int_{Z_U(\Q_p)}
\phi(\iota w_\beta z_U v(y,0,z) h)\overline{(\omega_{S,p}(v(y,0,z)h')\Phi)(0)}
dz_U dy dz \nonumber \\
=\int_{V(\Q_p)}\int_{Z_U(\Q_p)}
\phi(\iota w_\beta z_U v h)\overline{(\omega_{S,p}(w_\beta v h')\Phi)(0)}
dz_U dv.
\end{eqnarray}
The convergence follows from the smoothness of $\phi$, and the $L$-factors are finite by the 
condition on ${\rm Re}(s)$. The latter claim is 
similarly proved by the argument in the proof of Theorem \ref{Eisen-exp}. 
\end{proof}
For each $\phi\in \widetilde{I}^{\psi_p}_1(s,\delta_p)$ with a unitary character 
$\delta_p:\Q^\times_p\lra \C^\times$ and $t\in \Q^\times_p$, we define 
a normalized local Whittaker functional 
\begin{equation}\label{nlwf}
w^{\delta_p,s}_t(\phi):=|t|^{\frac{1}{2}}_p\frac{L(2s+1,\delta^2_p)}{L(s+\frac{1}{2},\delta_p)}\int_{X_\beta(\Q_p)}\phi((w_\beta x_\beta,1))
\overline{\psi_p(tx_\beta)}dx_\beta
\end{equation}
which can be extended holomorphically at $s=0$. Put 
$w^{\delta_p}_t(\phi):=w^{\delta_p,0}_t(\phi)$. 

\begin{lem}\label{local-non-arch}For $t\in \Q_p^\times$ and a unitary character 
$\delta_p$, it holds that 
\begin{enumerate}
\item ${\rm Hom}_{X_\beta(\Q_p)}(\widetilde{I}^{\psi_p}_1(0,\delta_p),\psi_{p}(t\ast))$ is 
non-zero and it is generated by $w^{\mu_p}_t$; 
\item If $\delta^2_p=|\cdot|$, $\widetilde{I}^{\psi_p}_1(0,\delta_p)$ has 
a unique irreducible subrepresentation $\widetilde{A}^{\psi_p}_1(\delta_p)$ which is unitary. 
Further, ${\rm Hom}_{X_\beta(\Q_p)}(\widetilde{A}^{\psi_p}_1(\delta_p),\psi_{p}(t\ast))$ is 
non-zero if and only if $\delta_p\neq \chi_{t,p}$. 
In that case, the restriction of $w^{\delta_p}_t$ to $\widetilde{A}^{\psi_p}_1(\delta_p)$ generates 
${\rm Hom}_{X_\beta(\Q_p)}(\widetilde{A}^{\psi_p}_1(\delta_p),\psi_{p}(t\ast))$. 
\end{enumerate}
\end{lem}
\begin{proof} The claims follow from \cite[Proposition 5.1]{IY}. 
\end{proof}

\begin{lem}\label{beta-whi1-na}For each $\phi\in I(s,\mu_p)$ and $S,t\in \Q^\times_p$, and $\Phi\in \mathcal{S}(X_\alpha(\Q_p))$, the following equality holds
$$w^{\mu_p\chi_{S,p},s}_t(\beta^{\psi_p}_S(\ast;\phi\otimes\overline{\Phi}))
=\Big|\frac{4}{27}S^3\Big|^{-\frac{3}{4}}_p|t|^{-\frac{1}{4}}_p\int_{X_\alpha(\Q_p)}\overline{\Phi(x)}
\widetilde{{\bf w}}^{\mu_p}_{{\rm Ad}(w_\alpha)(t,0,\frac{1}{3}S,0)}(v(x,0,0)\cdot\phi)dx$$
for ${\rm Re}(s)>-\frac{2}{3}$. 
\end{lem}
\begin{proof}
Put $A=\frac{L(s+\frac{1}{2},\mu_{p})L(s+\frac{3}{2},\mu_p)
L(2s+1,\mu_p) L(3s+\frac{3}{2},\mu_p)L(s+\frac{1}{2},\mu_p\chi_{S,p})}
{L(s+\frac{1}{2},\mu_{p,E_w})L(2s+\frac{1}{2},\mu^2_p)}$ and $B=\frac{L(2s+1,(\mu_p\chi_{S,p})^2)}{L(s+\frac{1}{2},\mu_p\chi_{S,p})}=
\frac{L(2s+1,\mu_p^2)}{L(s+\frac{1}{2},\mu_p\chi_{S,p})}$ for simplicity. 
First, we note that $w^{-1}_\beta V(\Q_p) w_\beta=V(\Q_p)$ and $w_\beta Z_U(\Q_p)
w^{-1}_\beta=Z_U(\Q_p)$. Then, by direct computation,  we have 
\begin{eqnarray}\label{eq1}
&&(AB)^{-1}|t|^{-\frac{1}{2}}_p w^{\mu_p\chi_{S,p},s}_t(\beta^{\psi_p}_S(h';\phi\otimes\overline{\Phi})) \nonumber\\
&=& A^{-1}\ds\int_{X_\beta(\Q_p)}
(\beta^{\psi_p}_S((w_\beta x_\beta,1);\phi\otimes\overline{\Phi})
\psi_p(-t x_\beta)dx_\beta   \nonumber\\
&\stackrel{\text{(\ref{some-imp})}}{=}& 
\ds\int_{X_\beta(\Q_p)}\int_{V(\Q_p)}\int_{Z_U(\Q_p)}
\phi(\iota w_\beta z_U v w_\beta x_\beta)
\overline{(\omega_{S,p}(w_\beta vw_\beta x_\beta)\Phi)(0)}\psi_p(-t x_\beta)
dz_U dv dx_\beta   \nonumber\\ 
&& \text{(noting $w^2_\beta$ is in the center of $\sL$ and substituting 
$w^{-1}_\beta v w_\beta\mapsto v,\ w_\beta z_U w^{-1}_\beta\mapsto z_U$,) } \nonumber \\
&=& 
\ds\int_{X_\beta(\Q_p)}\int_{V(\Q_p)}\int_{Z_U(\Q_p)}
\phi(\iota z_U v x_\beta)
\overline{(\omega_{S,p}(vx_\beta)\Phi)(0)}\psi_p(-t x_\beta)
dz_U dv dx_\beta.    \nonumber
\end{eqnarray} 
Since $x_\beta \in L^{{\rm ss}}(\Q_p)$, $x_\beta V(\Q_p)x^{-1}_\beta=V(\Q_p)$. 
By (\ref{heisen4}), $v=v(x,y,z)=v(0,y,z-xy)v(x,0,0)$. 
After using these, further, we substitute $x_\beta v x^{-1}_\beta$ and $z-xy$ with $v$ 
and $z$ respectively. Then, it proceeds as  
\begin{eqnarray}\label{eq2}
 &=& 
\ds\int_{X_\beta(\Q_p)}\int_{V(\Q_p)}\int_{Z_U(\Q_p)}
\phi(\iota z_U  x_\beta v)
\overline{(\omega_{S,p}(x_\beta v)\Phi)(0)}\psi_p(-t x_\beta)
dz_U dv dx_\beta.    \nonumber  \\
&=&\ds\int_{X_\beta(\Q_p)}\int_{V(\Q_p)}\int_{Z_U(\Q_p)}
\phi(\iota z_U  x_\beta v(0,y,z)v(x,0,0))
\overline{\Phi(x)\psi_p(Sz)}\psi_p(-t x_\beta)
dz_U dv(x,y,z) dx_\beta.  \nonumber
\end{eqnarray} 
Notice that $z_U  x_\beta v(0,y,z)=x_\beta v(0,y,z)z_U$ runs over all elements in $N(\Q_p)$. 
We remark that  
$$\psi_p(Sz)\psi_p(t x_\beta)=\psi_p(\langle (0,-\frac{1}{3}S,0,-t),(x_\beta,0,z,0) \rangle)=\psi_{{\rm Ad}(w_\alpha)(t,0,\frac{1}{3}S,0)}(n),\ n=n(x_\beta,\ast,z,\ast).$$
Therefore, the above integral becomes  
\begin{eqnarray} &=&
\ds\int_{n\in N(\Q_p)}\int_{x\in X_\alpha(\Q_p)}
\phi(\iota n v(x,0,0))
\overline{\Phi(x)\psi_{{\rm Ad}(w_\alpha)(t,0,\frac{1}{3}S,0)}(n)}
dx dn.   \nonumber 
\end{eqnarray} 
The integral converges absolutely. Hence, we can change the order of the double integral 
and it proceeds as 
\begin{eqnarray}
&=&
\ds\int_{x\in X_\alpha(\Q_p)}\overline{\Phi(x)}\Bigg(\int_{n\in N(\Q_p)}
\phi(\iota n v(x,0,0))
\psi_{{\rm Ad}(w_\alpha)(t,0,\frac{1}{3}S,0)}(n)
dn\Bigg) dx   \nonumber  \\
&=&
\ds\int_{x\in X_\alpha(\Q_p)}\overline{\Phi(x)}\Big(
(AB)^{-1}|q({\rm Ad}(w_\alpha)(t,0,\frac{1}{3}S,0))|^{-\frac{3}{4}}_p\widetilde{{\bf w}}^{\mu_p}_{{\rm Ad}(w_\alpha)(t,0,\frac{1}{3}S,0)}(v(x,0,0)\cdot\phi)\Big) dx   \nonumber  \\
&=&(AB)^{-1}\Big|\frac{4}{3^3}tS^3\Big|^{-\frac{3}{4}}_p
\ds\int_{x\in X_\alpha(\Q_p)}\overline{\Phi(x)}
\widetilde{{\bf w}}^{\mu_p}_{{\rm Ad}(w_\alpha)(t,0,\frac{1}{3}S,0)}(v(x,0,0)\cdot\phi) dx.   \nonumber 
\end{eqnarray} 
Cleaning up factors, we have the desired equality. 
\end{proof} 
Let $\Pi_p$ (resp. $\widetilde{A}^{\psi_p}_1(s,\mu_p\chi_{S,p})$) be a unique irreducible submodule of $I(s,\mu_p)$ (resp. $\widetilde{I}^{\psi_p}_1(s,\mu_p\chi_{S,p})$). 
\begin{cor}\label{beta-surj-arc} Keep the notations in Proposition \ref{beta-intertwining}. 
Then, 
$\beta^{\psi_p}_S:\Pi_p\otimes_\C \mathcal{S}(X_\alpha(\Q_p))\lra \widetilde{A}^{\psi_p}_1(s,\mu_p\chi_{S,p})$ is surjective. 
\end{cor}
\begin{proof}The claim follows from Lemma \ref{beta-whi1-na} by choosing $\Phi$ 
and $\phi$ suitably.
\end{proof}

\subsection{The archimedean case}\label{ac}
For a unitary character $\mu:\R^\times\lra \C^\infty$ and $s\in \C$,  
put $I(s,\mu):={\rm Ind}^{G_2(\R)}_{P(\R)}\mu(\det)|\det|^s$ (the normalized induction) 
by abusing the notation. 
Let $\widetilde{I}^{\psi}_1(s,\mu)$ be a principal series representation of $\widetilde{\SL_2(\R)}$, 
which is similarly defined as in (\ref{principal-series}). 

Let $\psi=\psi_{\infty}$ and $\psi_S:=\psi(S\ast)$ for 
$S\in \R$.  
For each $\phi \in I(s,\mu)$, $\Phi\in \mathcal{S}(X_\alpha(\R))$, $S\in \R^\times$, 
and $h'=(h,\ve)\in \widetilde{\SL_2(\R)}$, we define the integral 
\begin{eqnarray}\label{beta-local}
\beta^\psi_S(h';\phi\otimes\overline{\Phi}):=
\int_{X_{\alpha+\beta}(\R)}\int_{X_{2\alpha+\beta}(\R)}\int_{Z_U(\R)}
\phi(\iota w_\beta z_U v(y,0,z) h)\overline{(\omega_{S,\infty}(v(y,0,z)h')\Phi)(0)}
dz_U dy dz. \nonumber 
\end{eqnarray}
This is also a local analogue of ``$R(h;f,\Phi)$'' in Theorem \ref{Eisen-exp}. 

\begin{prop}\label{bia} Keep the notations above. 
Then, $\beta^{\psi}_S(h';\phi\otimes\overline{\Phi})$ is absolutely convergent 
if ${\rm Re}(s)>0$ and it yields a $V(\R)$-invariant 
and $\widetilde{\SL_2(\R)}$-equivariant $\C$-bilinear map 
$$\beta^{\psi}_S:I(s,\mu)\otimes_\C \mathcal{S}(X_\alpha(\R))
\lra \widetilde{I}^{\psi}_1(s,\mu\chi_{S,\infty}),
$$
where $\chi_{S,\infty}(a)=\langle -S,a \rangle_{\infty},\ a\in \R^\times$ is 
defined by using the local quadratic Hilbert symbol $\langle \ast,\ast \rangle_{\infty}$ on 
$\R^\times\times\R^\times$.
Namely, $\beta^{\psi}_S(v h';\phi\otimes\omega_{S,\infty}(\gamma)\overline{\Phi})
=\widetilde{I}^{\psi}_1(s,\mu\chi_{S,\infty})(\gamma)
\beta^{\psi}_S(h';\phi\otimes\overline{\Phi})$ for 
any $\gamma\in \widetilde{\SL_2(\R)}$ and $v\in V(\R)$. 
\end{prop}
\begin{proof} The claim is proved by a similar argument in the proof of 
Theorem \ref{Eisen-exp} as in Proposition \ref{beta-intertwining}. 
\end{proof}

For each section $\phi\in \widetilde{I}^{\psi}_1(s,\mu)$ and $t\in \R^\times$, 
define the Jacquet integral of $\phi$ as 
\begin{equation}\label{j-int}w^{\mu,s}_t(h';\phi):=\int_{X_\beta(\R)}\phi((w_\beta x_\beta h,\ve))
\overline{\psi_{\infty}(t x_\beta)} dx_\beta,\ h'=(h,\ve)\in \widetilde{\SL_2(\R)}.
\end{equation}

\begin{lem}\label{formula-for-w} Keep the notations as in Proposition \ref{bia}. 
For each $\phi\in I(s,\mu)$ and $S,t\in \R^\times$, and $\Phi\in \mathcal{S}(X_\alpha(\R))$, the following equality holds
$$w^{\mu\chi_{S,\infty},s}_t(h',\beta^{\psi_\infty}_S(\ast;\phi\otimes\overline{\Phi}))
=\int_{X_\alpha(\R)}\overline{\omega_{S,\infty}(h')\Phi(x)}
W^{(s)}_{{\rm Ad}(w_\alpha)(t,0,\frac{1}{3}S,0)}(1;v(x,0,0)h'\cdot\phi)dx
$$
for ${\rm Re}(s)>0$.  
\end{lem}
\begin{proof}Recall (\ref{Ws}). Then, the claim is proved similarly  
as in the proof of Lemma \ref{beta-whi1-na}. 
\end{proof}

We recall the basic facts from \cite[Chapter III, Proposition 7]{Wal80} in our setting. 
Let $\nu=\pm\frac{1}{2}$, $S\in \R^\times$ and put $\mu={\bf 1}$ and $s=k-\frac{1}{2}$ 
with an integer $k\ge 2$ in the setting of Proposition \ref{bia}. 
By computing the Weil constant at the infinite place with respect to $\psi_S$, we see 
$$\frac{\gamma_\infty(1)}{\gamma_\infty(-1)}\chi_{S,\infty}(-1)=-\sgn(S)e^{\frac{\pi \sqrt{-1}}{2}}=
e^{\pi \sqrt{-1}(-\sgn(S))\frac{1}{2}}
.$$
Thus, $\widetilde{I}^{\psi}_1(k-\frac{1}{2},\chi_{S,\infty})= \mathcal{B}(k-\frac{1}{2}, \nu)$ 
for $\nu=-\sgn(S)\frac{1}{2}$
in the notation of \cite[Chapter III, Section 1]{Wal80}. 

The following facts follow from \cite[Chapter III, Proposition 7]{Wal80} 
with a standard argument for principal series representations (cf. \cite[Section 2.2]{Szpruch}): 
When $S>0$ (hence $\nu=-\frac{1}{2}$), 
$\widetilde{I}^{\psi}_1(k-\frac{1}{2},\chi_{S,\infty})= \mathcal{B}(k-\frac{1}{2},-\frac{1}{2})$ 
admits a unique irreducible subrepresentation 
$\widetilde{\pi}^{-\sgn(-1)^k,-}_{k-\frac{1}{2}}$ such that 
\begin{enumerate}
\item (the case of (iii) in \cite[Proposition 6, p.22]{Wal80}) when $k$ is even, $\widetilde{\pi}^{-,-}_{k-\frac{1}{2}}$ is a unique $\psi_t$-generic irreducible component of $\widetilde{I}^{\psi}_1(k-\frac{1}{2},\chi_{S,\infty})$ for $t\in \R_{<0}$  
which has the highest weight $-k-\frac{1}{2}$;
\item  (the case of (ii) in \cite[Proposition 6, p.22]{Wal80}) when $k$ is odd, $\widetilde{\pi}^{+,-}_{k-\frac{1}{2}}$ is a unique $\psi_t$-generic irreducible component of $\widetilde{I}^{\psi}_1(k-\frac{1}{2},\chi_{S,\infty})$ for $t\in \R_{>0}$ 
which has the lowest weight $k+\frac{1}{2}$.
\end{enumerate}

When $S<0$ (hence $\nu=\frac{1}{2}$), 
$\widetilde{I}^{\psi}_1(k-\frac{1}{2},\chi_{S,\infty})= \mathcal{B}(k-\frac{1}{2},\frac{1}{2})$ 
admits a unique irreducible subrepresentation 
$\widetilde{\pi}^{\sgn(-1)^k,+}_{k-\frac{1}{2}}$ such that 
\begin{enumerate}
\setcounter{enumi}{2}
\item (the case of (ii) in \cite[Proposition 6, p.22]{Wal80}) when $k$ is even, $\widetilde{\pi}^{+,+}_{k-\frac{1}{2}}$ is a unique $\psi_t$-generic irreducible component of $\widetilde{I}^{\psi}_1(k-\frac{1}{2},\chi_{S,\infty})$ for $t\in \R_{>0}$ 
which has the lowest weight $k+\frac{1}{2}$;
\item  (the case of (iii) in \cite[Proposition 6, p.22]{Wal80}) when $k$ is odd, $\widetilde{\pi}^{-,+}_{k-\frac{1}{2}}$ is a unique $\psi_t$-generic irreducible component of $\widetilde{I}^{\psi}_1(k-\frac{1}{2},\chi_{S,\infty})$ for $t\in \R_{<0}$ 
which has the highest weight $-k-\frac{1}{2}$.
\end{enumerate}

As in \cite[Section 5.1]{IY}, for each $t\in \R^\times$ and $\ell\in \frac{1}{2}\Z$,  
we define 
$$W^{(\ell)}_{t,\wSL_2}(h')=|t|^{\frac{\ell}{2}}\exp(t (h\sqrt{-1}))j_{\ell}(h',\sqrt{-1})^{-1},\ 
h'=(h,\ve)\in \widetilde{SL_2(\R)},
$$
where $j_\ell$ is the automorphic factor defined in loc.cit.. 
Then, for $a\in \R_{>0}$, we have 
\begin{equation}\label{trans-borel}
W^{(\ell)}_{t,\wSL_2}(
\begin{pmatrix}
a & 0 \\
0 & a^{-1}
\end{pmatrix}
h')=W^{(\ell)}_{a^2t,\wSL_2}(h'),\ 
h'=(h,\ve).
\end{equation}
We remark that both of the cases $t>0$ and $t<0$ may happen because quaternionic modular forms are not holomorphic. 

For $w=(t,0,\frac{1}{3}S,0)\in W(\R)$, the condition $q({\rm Ad}(w_\alpha)w)<0$ is 
equivalent to $q(w)<0$ and in this case, we have $tS^3<0$. Then, it yields the parity condition $\sgn(S)=-\sgn(t)$. 

Let $T=\SO(2)(\R)=L(\R)\cap K_\infty$ where the identification is given by $\kappa_\theta:=e^{\sqrt{-1}\theta}\mapsto 
\begin{pmatrix}
e^{\sqrt{-1}\theta} & 0 \\
0 & e^{-\sqrt{-1}\theta}
\end{pmatrix}$. 
Let $\widetilde{T}$($\subset \widetilde{SL_2(\R)}$) be the double covering of $T$. 
Since $\widetilde{T}\simeq T$, the covering map $\widetilde{T}\lra T$ is identified with 
$T\lra T,\ z\mapsto z^2$ under  $\widetilde{T}\simeq T$.  
For each $j\in \Z_{\ge 0}$, 
there exists $\Phi^j_S\in \mathcal{S}(X_\alpha(\R))$ such that 
\begin{equation}\label{phi-wt}
\omega_{S,\infty}(\kappa_\theta)\Phi^j_S(x)=e^{2\pi \sqrt{-1}(\frac{1}{2}+j)\theta}\Phi^j_S(x),\  \kappa_\theta\in \widetilde{T}\simeq T.
\end{equation}
It is well-known that $\{\Phi^j_S\}_{j\ge 0}$ makes up an orthonormal basis of 
$\mathcal{S}(X_\alpha(\R))$ (cf. \cite[Section 2.1]{Takase}). 

\begin{prop}\label{d-case}
Let $S,t\in \R^\times$ with $\sgn(S)=-\sgn(t)$ and $k\ge 2$. 
Then, for each $\Phi^j_S$ with $j\ge 0$ and the function {\rm (}\ref{section2}\hspace{0.5mm}{\rm )}, there exists 
a non-zero constant $C_0(S)$ independent of $t$ such that   
\begin{eqnarray}
&&w^{\chi_{S,\infty},k-\frac{1}{2}}_t(h',\beta^{\psi_\infty}_S(\ast;
W^{(k-\frac{1}{2})}_{{\rm Ad}(w_\alpha)(t,0,\frac{1}{3}S,0)}(\ast;\phi_{\infty,I})\otimes\overline{\Phi^j_S})) \nonumber \\
&&\phantom{xx}=\left\{
\begin{array}{ll}
 C_0(S)|t|^{\frac{1}{4}}\times W^{k+\frac{1}{2}}_{t,\wSL_2}(h'), & \text{if $t<0$ 
 {\rm (}equivalently, $S>0${\rm )}, $k$ is even, and 
 $-k+j\in I$} \\
0, & \text{otherwise}
\end{array}\right.
\nonumber 
\end{eqnarray}
for $h'=(h,\ve)\in \widetilde{\SL_2(\R)}$. 
\end{prop}
\begin{proof} Assume $S>0$ (equivalently, $t<0$).   
It is easy to see that $V^\vee_k|_{T}\simeq \ds\bigoplus_{v=-k}^k \C(2v)$ 
where each element $z$ of $T$ acts on $\C(2v)\simeq \C$ by $z^{2v}$. 
Let $\wT$ act on $V^\vee_k$ via the covering map $\wT\lra T$. Then 
\begin{equation}\label{cv}
V^\vee_k|_{\wT}\simeq \ds\bigoplus_{v=-k}^k \C(v)
\end{equation}
 as a representation of $\wT$.  

By (\ref{phi-wt}), $\Phi^j_S(x)$ is of weight $\frac{1}{2}+j$ with respect to the action of $\wT$.  
On the other hand, 
by the formula in Lemma \ref{formula-for-w} and the definition of (\ref{section2}), the right hand side of (\ref{cv}) 
 can be written as a direct sum of the vectors of 
weights $v-(\frac{1}{2}+j),\ v\in I$ with respect to the action of $\wT$ and it also belongs to the image of the Whittaker model of $\widetilde{\pi}^{-\sgn(-1)^k,-}_{k-\frac{1}{2}}$ with respect to $\psi_t$. However, $k$ has to be even, since $t<0$.  
As observed, 
$\widetilde{\pi}^{-,-}_{k-\frac{1}{2}}$ has a vector of highest weight 
$-k-\frac{1}{2}$ and the equality 
$v-(\frac{1}{2}+j)=-k-\frac{1}{2},\ v\in I,\ j\ge 0$ holds exactly when 
$-k+j=v\in I$. Other vectors of weight $v-(\frac{1}{2}-j)$ with 
$-k+j\not \in I$ will be gone 
after taking the Jacquet integral (\ref{j-int}). 

Put $A_t=l(
\begin{pmatrix}
\sqrt{-t}^{-1} &0  \\
0  & \sqrt{-t}
\end{pmatrix}
)=m(\begin{pmatrix}
 1 &0  \\
0  & \sqrt{-t}
\end{pmatrix})$. 
Using (\ref{action1}) and (\ref{imi}) for the expression 
$A_t=m(\begin{pmatrix}
 1 &0  \\
0  & \sqrt{-t}
\end{pmatrix})$, and (\ref{action-tildeU}) for 
the expression 
$A_t=l(
\begin{pmatrix}
\sqrt{-t}^{-1} &0  \\
0  & \sqrt{-t}
\end{pmatrix}
)$, it is easy to see that 
$$W^{(k-\frac{1}{2})}_{{\rm Ad}(w_\alpha)(t,0,\frac{1}{3}S,0)}(1;v(x,0,0)A_t h'\cdot\phi)=W^{(k-\frac{1}{2})}_{{\rm Ad}(w_\alpha)(-1,0,\frac{1}{3}S,0)}(1;v(\frac{x}{\sqrt{-t}},0,0)h'\cdot\phi).$$
Note that 
$\omega_{S,\infty}(A_t h')\Phi^j_S(x)=(-t)^{-\frac{1}{4}}\omega_{S,\infty}(h')\Phi^j_S(\frac{x}{\sqrt{-t}})$ by (\ref{phi-wt}). 
Then, we have 
\begin{eqnarray}
&&w^{\chi_{S,\infty},k-\frac{1}{2}}_t(A_th',\beta^{\psi_\infty}_S(\ast;
W^{(k-\frac{1}{2})}_{{\rm Ad}(w_\alpha)(t,0,\frac{1}{3}S,0)}\otimes\overline{\Phi^j_S})) \nonumber \\
&=& (-t)^{-\frac{1}{4}}
\int_{X_\alpha(\R)}\overline{\omega_{S,\infty}(h')\Phi^j_S(\frac{x}{\sqrt{-t}})}
W^{(k-\frac{1}{2})}_{{\rm Ad}(w_\alpha)(-1,0,\frac{1}{3}S,0)}(1;v(\frac{x}{\sqrt{-t}},0,0)h'\cdot\phi)dx \nonumber  \\
&& \text{(substituting $x$ with $\sqrt{-t}x$)} \nonumber \\
&=& (-t)^{\frac{1}{4}(1+2j)}  
\int_{X_\alpha(\R)}\overline{\omega_{S,\infty}(h')\Phi^j_S(x)}
W^{(k-\frac{1}{2})}_{{\rm Ad}(w_\alpha)(-1,0,\frac{1}{3}S,0)}(1;v(x,0,0)h'\cdot\phi)dx. \nonumber 
\end{eqnarray} 
Therefore, we have 
\begin{eqnarray*}
&& (-t)^{-\frac{1}{4}}w^{\chi_{S,\infty},k-\frac{1}{2}}_t(A_th',\beta^{\psi_\infty}_S(\ast;
W^{(k-\frac{1}{2})}_{{\rm Ad}(w_\alpha)(t,0,\frac{1}{3}S,0)}(\ast;\phi_{\infty,I})\otimes\overline{\Phi^j_S})) \\
&& = w^{\chi_{S,\infty},k-\frac{1}{2}}_{-1}(h',\beta^{\psi_\infty}_S(\ast;
W^{(k-\frac{1}{2})}_{{\rm Ad}(w_\alpha)(-1,0,\frac{1}{3}S,0)}(\ast;\phi_{\infty,I})\otimes\overline{\Phi^j_S})).
\end{eqnarray*}
The right hand side is independent of $t$. Applying \cite[Lemma 12, p.24]{Wal80} to  
$\psi_{-1}$ with a uniqueness of a Whittaker model, 
there exists a constant $C_0(S)$ such that the RHS can be written as 
$C_0(S)W^{k+\frac{1}{2}}_{-1,\wSL_2}(h')$. By (\ref{trans-borel}), 
$W^{k+\frac{1}{2}}_{-1,\wSL_2}(A^{-1}_t h')=W^{k+\frac{1}{2}}_{t,\wSL_2}(h')$. 
Summing up, we have 
$$w^{\chi_{S,\infty},k-\frac{1}{2}}_t(h',\beta^{\psi_\infty}_S(\ast;
W^{(k-\frac{1}{2})}_{{\rm Ad}(w_\alpha)(t,0,\frac{1}{3}S,0)}(\ast;\phi_{\infty,I})\otimes
\overline{\Phi^j_S}))=C_0(S)
(-t)^{\frac{1}{4}}W^{k+\frac{1}{2}}_{t,\wSL_2}(h')$$
as desired. 
The case $S<0$ is easily handled. In fact, the vanishing follows from the  
parity condition. 

Finally, we check $C_0(S)\neq 0$ when $S>0$. It suffices to check $\beta^{\psi_\infty}_S$ induces 
a surjective map from $D_k\otimes \mathcal{S}(X_\alpha(\R))$ to 
$\widetilde{\pi}^{-\sgn(-1)^k,-}_{k-\frac{1}{2}}$. 
For any non-zero function $F$ on $X_\alpha(\R)=\R$ which has the moderate growth, one 
can choose $\Phi\in \mathcal{S}(X_\alpha(\R))$ to have a sufficiently small support so that $\ds\int_{\R}\Phi(x)F(x)dx\neq 0$. Applying this to Lemma \ref{formula-for-w}, we see that 
$\beta^{\psi_\infty}_S$ is non-zero and the claim follows from the irreducibility of $\widetilde{\pi}^{-\sgn(-1)^k,-}_{k-\frac{1}{2}}$.   
\end{proof}

\subsection{The global case}\label{gc}
Let us keep the notations in Section \ref{FS}. 
Let $\psi=\otimes'_p\psi_p$ be the standard additive character and put 
$\psi_S=\psi(S\ast)$ for $S\in \Q^\times$. Recall the global Weil representation $\omega_{\psi_S}=\otimes'_p\omega_{S,p}
=\omega_{S,\f}\otimes\omega_{S,\infty}$ defined in Section \ref{FJEES}.  
For $\Phi_{\f}\in  \mathcal{S}(X_\alpha(\A_{\f}))$, $S\in \Q_{>0}$, and $j\in \Z_{\ge 0}$
define $\Phi^j_S\in\mathcal{S}(X_\alpha(\A))$ by 
$$\Phi^j_S(x)=\Phi_{\f}(x_{\f})\Phi^j_{S,\infty}(x_\infty),
$$
where $\Phi^j_{S,\infty}$ is defined with the property (\ref{phi-wt}). 

We recall the expansion (\ref{fseries}) of $F_f(\ast;\phi)$ for a 
distinguished vector $\phi=\phi_{\f}\otimes \phi_{\infty,I}$.  
Let $F_{(0,0)}(g;\phi):=\ds\int_{Z_U(\Q)\bs Z_U(\A)}F_f(zg;\phi)dz$ 
be the constant term along $Z_U$. 
Then, we have 
\begin{equation}\label{f00}
F_{(0,0)}(g;\phi)=\sum_{w\in W(\Q)_{\ge 0}\atop q(w)<0,
x_{3\alpha+\beta}(w)=0}
c_w \widetilde{{\bf w}}_{{\rm Ad}(w_\alpha)w}(g_{\f}\cdot \phi_{\f})W^{(k-\frac{1}{2})}_{{\rm Ad}(w_\alpha)w}(g_\infty;\phi_{\infty,I}),\ g=g_{\f}g_\infty 
\in G_2(\A)
\end{equation}
where we put $c_w:=C^{\mu_{\f}}_w(F_f)$ for simplicity. 
As in (\ref{eisen-FJ}), we can define 
\begin{equation}\label{f00-FJ}
F_{(0,0)}(\ast;\phi)_{\psi_S,\Phi^j_S}(h')=
\ds\int_{V(\Q)\bs V(\A)}F_{(0,0)}(vh';\phi)
\overline{\Theta_{\psi_S}(vh';\Phi^j_S)}dv,\ h'\in\widetilde{\SL_2(\A)}.
\end{equation}

\begin{prop}\label{auto-criterion} Keep the notations as above. In particular, $S>0$. It holds that there exists a non-zero constant $C_1(S)$ independent of $t$ such that 
 if $-k+j\not\in I$, $F_{(0,0)}(\ast;\phi)_{\psi_S,\Phi^j_S}(h')=0$. Otherwise, for $h'=h'_{\f}h'_\infty=(h'_p)_p\in \widetilde{\SL_2(\A)}$,  
$$F_{(0,0)}(\ast;\phi)_{\psi_S,\Phi^j_S}(h')=C_1(S)\ds\sum_{t\in \Q_{<0}}
c_{(t,0,\frac{S}{3},0)}
w^{\mu_{\f}\chi_{S,\f}}_t(\beta^{\psi_{\f}}_S(\ast h'_{\f};\phi_{\f}
\otimes\overline{\Phi_{\f}}))
 W^{k+\frac{1}{2}}_{t,\wSL_2}(h'_\infty),
$$
where 
$$w^{\mu_{\f}\chi_{S,p}}_t(\beta^{\psi_{\f}}_S(\ast h'_{\f};\phi_{\f}
\otimes\overline{\Phi_{\f}})):=
\prod_{p<\infty}w^{\mu_p\chi_{S,p},s_p}_t(\beta^{\psi_p}_S(\ast h'_p;\phi\otimes\overline{\Phi_p}))$$
and $s_p=\begin{cases} \frac{1}{2}, &\text{if $p\in S(\pi_{\f})$}\\ 0, &\text{otherwise}\end{cases}$. 

Further, $F_{(0,0)}(\ast;\phi)_{\psi_S,\Phi^j_S}$ is an automorphic form on $\widetilde{\SL_2(\A)}$.  
\end{prop} 
\begin{proof}
By definition, we have 
\begin{eqnarray*}
&& F_{(0,0)}(\ast;\phi)_{\psi_S,\Phi_S}(h') \\
&&\phantom{xxx}=\ds\int_{V(\Q)\bs V(\A)}
\Bigg(\sum_{w\in W(\Q)_{\ge 0}\atop q(w)<0,
x_{3\alpha+\beta}(w)=0}
c_w \widetilde{{\bf w}}_{{\rm Ad}(w_\alpha)w}(v_{\f}h'_{\f}\cdot \phi_{\f})W^{(k-\frac{1}{2})}_{{\rm Ad}(w_\alpha)w}(v_\infty h'_\infty)\Bigg)
\overline{\Theta_{\psi_S}(vh';\Phi_S)}dv.
\end{eqnarray*}
Substituting $v$ with $v+v(0,0,z_\infty)$, since $v(0,0,z_\infty)=n(0,0,z_\infty,0,0)\in V(\R)$ 
for any $z_\infty\in\R$, if we put $w=(w_1,w_2,w_3,0)$, the factor 
$$\psi_{{\rm Ad}(w_\alpha)w}(n(0,0,z_\infty,0,0))\overline{\psi(Sz_\infty)}
=\psi((3w_3-S)z_\infty)$$
comes out from the right hand side. Therefore, $w_3=\frac{1}{3}S$. 
Thus, we can express the above integral as   
$$I_1:=\ds\int_{V(\Q)\bs V(\A)}
\Bigg(\sum_{w\in W(\Q)_{\ge 0},\ q(w)<0 \atop
x_{2\alpha+\beta}(w)=\frac{1}{3}S,\ x_{3\alpha+\beta}(w)=0}
c_w W_{{\rm Ad}(w_\alpha)w}(vh')\Bigg)
\overline{\Theta_{\psi_S}(vh';\Phi_S)}dv,
$$
where we put 
$$W_{{\rm Ad}(w_\alpha)w}(vh'):=\widetilde{{\bf w}}_{{\rm Ad}(w_\alpha)w}(v_{\f}h'_{\f}\cdot \phi_{\f})W^{(k-\frac{1}{2})}_{{\rm Ad}(w_\alpha)w}(v_\infty h'_\infty)$$
for simplicity. 
We can write $w=(\ast,\ast,\frac{1}{3}S,0)\in W(\Q)$ as 
$$w={\rm Ad}(v(\lambda,0,0))(t,0,\frac{1}{3}S,0),
$$
for some $\lambda,t\in \Q$ with $t\neq 0$ and $tS<0$ (hence, $t<0$). 
Then, 
$$w=(t+S\lambda^2,\frac{2}{3}S\lambda,\frac{1}{3}S,0)=:w(t,\lambda,S).
$$
Using this and $v(x,y,z)=v(0,y,z-xy)v(x,0,0)=n(0,y,z-xy,0,0)v(x,0,0)$, we have 
\begin{eqnarray*}
&& I_1=\ds\int_{V(\Q)\bs V(\A)}
\Bigg(\sum_{
w=w(t,\lambda,S)\atop 
t,\lambda\in \Q,\ t<0}\psi(S(z-xy))\psi(2S\lambda y ) 
c_w W_{{\rm Ad}(w_\alpha)w}(v(x,0,0)h')\Bigg) \\
&& \phantom{xxxxxxxxssssssss}\times
\overline{\sum_{\xi\in X_\alpha(\Q)}(\omega_{\psi_S}(h')\Phi_S)(x+\xi)
\psi(S(z-xy))\psi(S\xi y)}dv \\
&& =\ds\int_{V(\Q)\bs V(\A)}
\Bigg(\sum_{
t,\lambda\in \Q,\ t<0}
c_{w(t,0,S)} W_{{\rm Ad}(w_\alpha)w(t,0,S)}(v(x+2\lambda,0,0)h')\Bigg)
\psi(2S\lambda y ) \\
&& \phantom{xxxxxxxxsssssssss}\times \Big(\overline{\sum_{\xi\in X_\alpha(\Q)}(\omega_{\psi_S}(h')\Phi_S)(x+\xi)}
\Big)\psi(-S\xi y )dv.
\end{eqnarray*}
The integral $\psi(S(2\lambda-\xi)y )$ over 
$X_{\alpha+\beta}(\Q)\bs X_{\alpha+\beta}(\A)$ is zero unless  
$2\lambda=\xi$. 
Therefore, the above integral becomes 
$$=\int_{X_\alpha(\Q)\bs X_\alpha(\A)}\sum_{\xi\in X_\alpha(\Q)}
\Bigg(\sum_{t,\xi\in \Q,\ t<0}
c_{w(t,0,S)} W_{{\rm Ad}(w_\alpha)w(t,0,S)}(v(x+\xi,0,0)h')\Bigg)
\overline{(\omega_{\psi_S}(h')\Phi_S)(x+\xi)}dx.
$$
Here we also used the fact that ${\rm vol}(X_{2\alpha+\beta}(\Q)\bs X_{2\alpha+\beta}(\A))=1$. 
By the unfolding technique in $\xi$, the above integral is
$$=\int_{X_\alpha(\A)}
\Bigg(\sum_{t\in \Q_{<0}}
c_{w(t,0,S)} W_{{\rm Ad}(w_\alpha)w(t,0,S)}(v(x,0,0)h')\Bigg)
\overline{(\omega_{\psi_S}(h')\Phi_S)(x)}dx$$
$$= \sum_{t\in \Q_{<0}}
c_{w(t,0,S)} 
\Bigg(
\ds\int_{X_\alpha(\A_{\f})}
\widetilde{{\bf w}}^{\mu_{\f}}_{{\rm Ad}(w_\alpha)(t,0,\frac{1}{3}S,0)}(v(x_{\f},0,0)h'_{\f}\cdot\phi_{\f})
\overline{(\omega_{S,\f}(h_{\f})\Phi_{\f})(x_{\f})}dx_{\f}\Bigg)$$
$$
\phantom{xxxssxz} \times\Bigg(\ds\int_{X_\alpha(\R)}W^{(k-\frac{1}{2})}_{{\rm Ad}(w_\alpha)(t,0,\frac{1}{3}S,0)}
(v(x_\infty,0,0) h'_\infty)
\overline{(\omega_{S,\infty}(h_{\infty})\Phi^j_{S,\infty})(x_{\infty})}dx_{\infty}
\Bigg).
$$
By Proposition \ref{d-case}, it vanishes unless $-k+j\in I$. 
In the remaining case, by Proposition \ref{d-case} again, the above integral is 
$$=C_0(S) \sum_{t\in \Q_{<0}}
c_{w(t,0,S)}|t|^{\frac{1}{4}} W^{k+\frac{1}{2}}_{t,\wSL_2}(h'_\infty) 
\Bigg(
\ds\int_{X_\alpha(\A_{\f})}
\widetilde{{\bf w}}^{\mu_{\f}}_{{\rm Ad}(w_\alpha)(t,0,\frac{1}{3}S,0)}(v(x_{\f},0,0)h'_{\f}\cdot\phi_{\f})
\overline{(\omega_{S,\f}(h_{\f})\Phi_{\f})(x_{\f})}dx_{\f}\Bigg).$$
By Lemma \ref{beta-whi1-na}, 
$$\ds\int_{X_\alpha(\A_{\f})}
\widetilde{{\bf w}}^{\mu_{\f}}_{{\rm Ad}(w_\alpha)(t,0,\frac{1}{3}S,0)}(v(x_{\f},0,0)h'_{\f}\cdot\phi_{\f})
\overline{(\omega_{S,\f}(h_{\f})\Phi_{\f})(x_{\f})}dx_{\f}$$
$$\phantom{xxxxxxxxxx} =|t|^{-\frac{1}{4}}\Big(\frac{4}{27}|S|^3\Big)^{-\frac{3}{4}} w^{\mu_{\f}\chi_{S,\f}}_t(\beta^{\psi_{\f}}_S(\ast;\phi_{\f}
\otimes\overline{\Phi_{\f}})).
$$
Summing up, we have the desired claim with $C_1(S)=C_0(S)\Big(\frac{4}{27}|S|^3\Big)^{-\frac{3}{4}}$. 

The later claim is proved similarly as in the proof of \cite[Lemma 5.4-(2)]{KY2}. 
\end{proof}

\section{Fourier expansion of Shimura correspondence}\label{FESC}
In this section, we refer \cite[Section 1,5, and 8.2]{IY} for the treatment of adelic modular forms of half-integral weight. We remark that, in \cite{IY}, the authors used the additive 
character $\psi_p(-\ast)$ at finite place $p$ to get positive indices in the 
Fourier expansion while negative indices show up in our setting as below.

Let $f$ be the newform in Section \ref{intro}. Recall the notation in Lemma \ref{local-non-arch} and put 
$$A^{\psi_p}_1(\mu_p)=\begin{cases} \text{$\widetilde{I}^{\psi_p}_1(0,\mu_p)$ with 
a unitary character $\mu_p:\Q^\times_p\lra \C^\times$}, &\text{if $p\not\in S(\pi_{\f})\cup
\{\infty\}$}\\ 
\text{$\widetilde{A}^{\psi_p}_1(\mu_p)$ with $\mu^2_p=|\cdot|_p$}, &\text{if $p\in S(\pi_{\f})$} \\
\widetilde{\pi}^{-\sgn(-1)^k,-}_{k-\frac{1}{2}}  & \text{$p=\infty$}
\end{cases}. 
$$ 
Then, $\otimes'_p A^{\psi_p}_1(\mu_p)$ is a cuspidal automorphic representation of 
$\widetilde{\SL_2(\A)}$ which corresponds to $f$ by the Shimura correspondence and for each distinguished vector $\phi^{\widetilde{{\rm SL}}_2}_{\f}=\otimes'_p \phi^{\widetilde{{\rm SL}}_2}_p$ in $\otimes'_{p<\infty} A^{\psi_p}_1(\mu_p)$,  by using (\ref{nlwf}), we have embedding 
from $\otimes'_{p<\infty} A^{\psi_p}_1(\mu_p)$ into the space of 
automorphic forms on  $\widetilde{\SL_2(\A)}$ by 
\begin{equation}\label{SFC}
{\rm Ah}_f(h;\phi^{\widetilde{{\rm SL}}_2}_{\f}):=\sum_{t\in \Q_{<0}}c_t \Big( \prod_{p}
w^{\mu_p,s_p}_t(h_p\cdot \phi^{\widetilde{{\rm SL}}_2}_p)\Big)W^{k+\frac{1}{2}}_{t,\wSL_2}
(h'_\infty),\ h=(h_p)_\p\in \widetilde{\SL_2(\A)}
\end{equation}
for some $c_t\in \C$ ($t\in \Q_{<0}$), where 
$s_p=\begin{cases} \frac{1}{2}, &\text{if $p\in S(\pi_{\f})$}\\ 0, 
&\text{otherwise}\end{cases}$. 
We should remark that the above automorphic form is anti-holomorphic. 
For any $S\in \Q_{>0}$, let $\chi_S=\otimes'_p\chi_{S,p}:\Q^\times\bs\A^\times_\Q\lra \C^\times,\ 
a\mapsto \langle -S,a \rangle$ where $\langle \ast,\ast\rangle$ is the quadratic Hilbert symbol on $\A^\times\times\A^\times$. 
Then, by using a double covering of $\GL_2(\A)$ which contains 
$\widetilde{\SL_2(\A)}$ as a normal subgroup as in the proof of \cite[Lemma 5.6-(5)]{IY}, 
one can define  
\begin{eqnarray}\label{SFCtwist}
{\rm Sh}^S_f(h;\phi^{\widetilde{{\rm SL}}_2}_{\f})&:=&{\rm Sh}_f(\diag(1,S)\cdot h\cdot \diag(1,S)^{-1};\phi^{\widetilde{{\rm SL}}_2}_{\f}) \\
&=& \mu^{-1}_{\f}(S)S^{k+\frac{1}{2}}
\sum_{t\in \Q_{<0}}c_{St} \Big( \prod_{p}
w^{\mu_p\chi_{S,p},s_p}_t(h_p\cdot \phi^{\widetilde{{\rm SL}}_2}_{p})\Big)W^{k+\frac{1}{2}}_{t,\wSL_2}
(h'_\infty), \nonumber 
\end{eqnarray}
for $h=(h_p)_p\in \widetilde{\SL_2(\A)}$
and it generates the cuspidal representation 
\begin{equation}\label{twist-auto}
\otimes'_p A^{\psi_p}_1(\mu_p \chi_S)=\Big(\otimes'_{p<\infty} A^{\psi_p}_1(\mu_p \chi_S)\Big)\otimes \widetilde{\pi}^{-\sgn(-1)^k,-}_{k-\frac{1}{2}}.
\end{equation}

If we specify a distinguished section suitably, then we can recover the complex conjugation of the classical Shimura 
correspondence 
$${\rm Sh}_f(\tau)=\sum_{n\in \Z_{> 0}\atop 
\text{$n\equiv 0$ or 1 mod $4$ }}c(n)\overline{q}^{n},\ q=e^{2\pi\sqrt{-1}\tau},\ \tau\in 
\mathbb{H}:=\{\tau\in \C\ |\ {\rm Im}(\tau)>0\}$$
and if $n$ is $1$ or the fundamental discriminant of a real quadratic field, then 
$c(n)=c_{-n}$. 
In particular, $c(1)$ is proportional to $L(k,f)$ by a non-zero constant 
(\cite[Corollaire 2, p.379]{Wal81}). 

\section{Proof of Theorem \ref{main1}}\label{Pmt} 
We are now ready to prove Theorem \ref{main1}. Assume $k\ge 2$ is even. 
Starting with a Hecke eigen newform $f\in S_k(\Gamma_0(C))^{{\rm new,ns}}$, 
we defined $\Pi(f)=\otimes'_p\Pi_p=\Pi_{\f}\otimes \Pi_\infty$ in Section \ref{intro} and 
$\mu_{\f}$ by (\ref{mu-finite}) from $\Pi_{\f}$ in Section \ref{FS}.  

For each distinguished vector $\phi^{\widetilde{{\rm SL}}_2}_{\f}\in \otimes'_{p<\infty} A^{\psi_p}_1(\mu_p \chi_S)$, by Corollary \ref{beta-surj-arc}, there exists a 
distinguished vector $\phi_{\f}\in \Pi_{\f}$ and the Schwartz function $\Phi_{\f}$ on 
$X_\alpha(\A_{\f})$ such that  $\beta^{\psi_{\f}}_S(\ast h'_{\f};\phi_{\f}
\otimes\overline{\Phi_{\f}})=\phi^{\widetilde{{\rm SL}}_2}_{\f}$. 
Applying $I=\{-k\}$, $\phi_{\f}$, and $\Phi_{\f}$ to Proposition \ref{auto-criterion}, we have 
\begin{eqnarray}
F_{(0,0)}(\ast;\phi)_{\psi_S,\Phi^0_S}(h')&=&C_1(S)\ds\sum_{t\in \Q_{<0}}
c_{(t,0,\frac{S}{3},0)}w^{\mu_{\f}\chi_{S,\f}}_t(\beta^{\psi_{\f}}_S(\ast h'_{\f};\phi_{\f}
\otimes\overline{\Phi_{\f}}))
 W^{k+\frac{1}{2}}_{t,\wSL_2}(h'_\infty),  \nonumber \\
&=&C_1(S)\ds\sum_{t\in \Q_{<0}}
c_{(t,0,\frac{S}{3},0)}w^{\mu_{\f}\chi_{S,\f}}_t(\phi^{\widetilde{{\rm SL}}_2}_{\f})
 W^{k+\frac{1}{2}}_{t,\wSL_2}(h'_\infty),  \nonumber
 \end{eqnarray}
$h'=h'_{\f}h'_\infty\in \widetilde{\SL_2(\A)}$ and 
it generates the representation (\ref{twist-auto}) by Corollary \ref{beta-surj-arc}. 
Thus, there exists a non-zero constant $C_2(S)$ depending on $S$ such that 
$$F_{(0,0)}(\ast;\phi)_{\psi_S,\Phi^0_S}(h')=C_2(S){\rm Sh}^S_f(h;\phi^{\widetilde{{\rm SL}}_2}_{\f}).$$
Comparing coefficients, we have $c_{(t,0,\frac{S}{3},0)}=C(S)\mu^{-1}_{\f}(S)c_{St}$ 
where $C(S)=C_1(S)^{-1}C_2(S)S^{k+\frac{1}{2}}$. 
This completes the proof.

\section{Degree 7 standard $L$-function and the Arthur parameter for the Gan-Gurevich lift}

Let $f$ be a cuspidal holomorphic eigenform of weight $2k\geq 4$ and trivial nebentypus with respect to $\Gamma_0(C)$, and 
$\pi_f$ its associated automorphic representation. Let $\pi_f=\otimes_p' \pi_p\otimes \pi_\infty$. 
Recall the quaternionic cusp form $F=F_f(*;\phi)$ defined in Section 1, the Gan-Gurevich lift of $f$ on $G_2$. Let $\Pi_F$ be the irreducible representation of $G_2(\Bbb A)$ generated by $F$.

\subsection{Degree 7 standard $L$-function}
\begin{thm}\label{Lfunct} Let $S=S(\pi_{\f})\cup \{p|C\}$. Then
the degree 7 standard $L$-function of $\Pi_F$ is
$$L^S(s,\Pi_F,{\rm St})=L^S(s,\Sym^2\pi_f)L^S(s+\frac 12,\pi_f)L^S(s-\frac 12,\pi_f).
$$
where $L^S(s,\Pi_F,{\rm St})=\prod_{p\nmid S} L(s,\Pi_p,{\rm St})$ is the partial $L$-function.
\end{thm}
\begin{proof}
For $p\notin S$, $\pi_p=\pi(\mu_p,\mu_p^{-1})$ with an unramified character $\mu_p$ . Let $\mu_p(p)=\alpha_p$. Then
$\Pi_p={\rm Ind}_{P(\Bbb Q_p)}^{G_2(\Bbb Q_p)} \mu_p\circ \det$. 

Recall the parametrization in \cite{M}: $M_\alpha\simeq \GL_2$ under the map determined by
$$t\longrightarrow \diag((2\alpha+\beta)(t),(\alpha+\beta)(t)),
$$
and $\alpha$ corresponds to the standard positive root of $\GL_2$. The parametrization of the maximal torus of $G_2$ is
$$t: \GL_1 \times \GL_1\longrightarrow T,\quad (a,b)\longmapsto t(a,b),
$$
given by $\alpha(t(a,b))=ab^{-1}$ and $\beta(t(a,b))=a^{-1}b^2$. Now for $\mu_1,\mu_2$, quasi-characters of $\Bbb Q_p^\times$,
let $\mu_1=|\ |^{s_1}\mu_1'$ and $\mu_2=|\ |^{s_2}\mu_2'$, where $\mu_1',\mu_2'$ are unitary characters. 
We denote the induced representation
$$I_B(\mu_1\otimes\mu_2)={\rm Ind}_B^G |\ |^{s_1}\mu_1'\otimes |\ |^{s_2}\mu_2'={\rm Ind}_B^G\, \mu_1'\otimes\mu_2'\otimes exp(\lambda,H_B(\ )),
$$
where $\lambda=s_1(2\alpha+\beta)+s_2(\alpha+\beta)$. Now consider the degenerate principal series ${\rm Ind}_{P_\alpha}^G\, \mu_p\circ \det$ (normalized induction). Let $\mu_p=|\,|^{-s_p}$ so that $\mu_p(p)=p^{s_p}=\alpha_p$. 
Since $\mu_p\circ \det\hookrightarrow {\rm Ind}_B^{GL_2} |\ |^{-s_p-\frac 12}\otimes |\ |^{-s_p+\frac 12}$, by inducing in stages,
$${\rm Ind}_{P_\alpha}^G \, \mu_p\circ \det\hookrightarrow {\rm Ind}_B^G\, |\ |^{-s_p-\frac 12}\otimes |\ |^{-s_p+\frac 12}\simeq Ind_B^G \, exp(\lambda,H_B(\ )),
$$
where $\lambda=(-s_p-\frac 12)(2\alpha+\beta)+(-s_p+\frac 12)(\alpha+\beta)$.

Note that the weights of the degree 7 standard representation of ${}^L G_2$ are $0$, $\pm \beta^\vee$, $\pm (3\alpha+\beta)^\vee$, $\pm (3\alpha+2\beta)^\vee$ (short roots of ${}^L G_2$).
 Then
$$\langle\lambda, \beta^\vee\rangle=-s_p+\frac 12,\quad \langle\lambda, (3\alpha+\beta)^\vee\rangle=-s_p-\frac 12,\quad \langle\lambda, (3\alpha+2\beta)^\vee\rangle=-2s_p.
$$
Hence
\begin{eqnarray*}
&& L(s,\Pi_p,{\rm St})^{-1}\\
&& =(1-p^{-s})(1-\alpha_p^2 p^{-s})(1-\alpha_p^{-2} p^{-s})(1-\alpha_p p^{\frac 12-s})(1-\alpha_p^{-1} p^{\frac 12-s})(1-\alpha_p p^{-\frac 12-s})(1-\alpha_p^{-1} p^{-\frac 12-s}).
\end{eqnarray*}
Therefore, 
$$L(s,\Pi_p,{\rm St})
=L(s,{\rm Sym}^2(\pi_p))L(s+\frac 12,\pi_p)L(s-\frac 12,\pi_p).
$$
This proves the result.
\end{proof}

Let $\widetilde\Pi$ be the Langlands conjectural functorial lift of $\Pi_F$ to $GL_7$. Let 
$\widetilde\Pi=\otimes_p' \widetilde\Pi_p\otimes \widetilde\Pi_\infty$. Then the above theorem says that $\widetilde\Pi_p$ is a quotient of
$${\rm Ind}_{P_{2,3,2}}^{GL_7} \pi_p|\det|^\frac 12\otimes \Sym^2(\pi_p)\otimes \pi_p|\det|^{-\frac 12},
$$ 
where $P_{2,3,2}$ is the standard parabolic subgroup of $GL_7$ with the Levi subgroup $GL_2\times GL_3\times GL_2$.}

\subsection{Arthur parameter for the Gan-Gurevich lift}
We have the following Arthur parameter of $\Pi_F$ \cite{GG, Mu}:
Let $\mathcal L$ be the Langlands group over $\Bbb Q$, and let $\rho_f: \mathcal L\longrightarrow \SL_2(\Bbb C)$ be the two-dimensional irreducible representation of $\mathcal L$ corresponding to $\pi_f$.

Let $R_7$ be the standard representation of $G_2(\Bbb C)$. Then $R_7: G_2(\Bbb C)\longrightarrow \GL_7(\Bbb C)$. 
Let $\SL_{2,\gamma}(\Bbb C)$ be the $\SL_2$-subgroup of $G_2(\Bbb C)$ corresponding to $\gamma$. 
Let $\iota_\gamma: \SL_2(\Bbb C)\longrightarrow \SL_{2,\gamma}(\Bbb C)\subset G_2(\Bbb C)$. Since $\alpha$ and $3\alpha+2\beta$ are orthogonal,
$\SL_{2,\alpha}(\Bbb C)$ and $\SL_{2,3\alpha+2\beta}(\Bbb C)$ are mutual centralizers and we have inclusion
$$\iota_{\alpha,3\alpha+2\beta}: \SL_{2,\alpha}(\Bbb C)\times \SL_{2,3\alpha+2\beta}(\Bbb C)\longrightarrow G_2(\Bbb C).
$$
Now we have a map $\rho_f: \mathcal L\longrightarrow \SL_{2,\alpha}(\Bbb C)$, and $\rho_s: \SL_2(\Bbb C)\longrightarrow \SL_{2,3\alpha+2\beta}(\Bbb C)$ is the identity map. Then we have a map
$$\rho_f\oplus \rho_s: \mathcal L\times \SL_2(\Bbb C)\longrightarrow \SL_{2,\alpha}(\Bbb C)\times \SL_{2,3\alpha+2\beta}(\Bbb C).
$$
Let $\psi_{GG}=\iota_{\alpha,3\alpha+2\beta}\circ (\rho_f\oplus\rho_s): \mathcal L\times \SL_2(\Bbb C)\longrightarrow G_2(\Bbb C)$.
By \cite{Mu}, 
$R_7\circ \iota_{\alpha}: \SL_2(\Bbb C)\longrightarrow \GL_7(\Bbb C)$ is
$\diag(St^\vee,Ad,St)$. Hence $\psi_{GG}$ is the Arthur parameter for $\Pi_F$.

To state the conjectural Arthur multiplicity formula, we assume that $S_0=\emptyset$, where $S_0$ is the subset of $S(\pi_{\f})$ such that $\pi_p={\rm St}_p$. 
Let's review the epsilon factors $\epsilon(\frac 12,\Sym^3(\pi_p),\psi_p)$.
For $p\notin S(\pi_{\f})\cup \{\infty\}$, $\pi_p=\pi(\mu_p,\mu_p^{-1})$ for a unitary character $\mu_p$. In this case, $\epsilon(\frac 12,Sym^3(\pi_p),\psi_p)=1$. When $\mu_p$ is ramified, use \cite[p.14]{Ta}.
Let $\Pi_p^+=J_\beta(\frac 12,\pi_p)$. 

Let
$\Pi_\infty$ be the quaternionic discrete series representation of $G_2(\Bbb R)$ with Harish-Chandra parameter $(k-2) (3\alpha+2\beta)+\rho$. By \cite{CM}, $\epsilon(\frac 12,\pi_\infty,\psi_\infty)=(-1)^k$ and
$\epsilon(\frac 12, \Sym^3(\pi_\infty),\psi_\infty)= -1$.

If $p\in S(\pi_{\f})$, $\pi_p={\rm St}_p\otimes\mu_p$, where $\mu_p$ is a nontrivial quadratic character, let $\Pi_p^+=J_\beta(\frac 12,\pi_p)$, and $\Pi_p^-=J_\beta(1,\pi(1,\mu_p))$. If $\mu_p$ is unramified, by \cite{CM}, $\epsilon(\frac 12,{\rm St}_p\otimes\mu_p,\psi_p)=-1$, and $\epsilon(\frac 12, \Sym^3({\rm St}_p\otimes\mu_p),\psi_p)= 1$. If $\mu_p$ is ramified, use \cite[p. 284, Case IV-a]{RS}.

The following is a special case of \cite[\S 13.4]{GG}: 
\begin{conj}\label{Mundy}  Suppose $S_0=\emptyset$.
Let $S'\subset S(\pi_{\f})$. Then
$$\Pi=\Pi_{\infty}\otimes \otimes_{p\in S'} \Pi_p^-\otimes \otimes_{p\notin S'}' \Pi_p^+,
$$
occurs in $L^2_{\rm disc}(G_2(\Q)\backslash G_2(\Bbb A))$ with either multiplicity zero or one. It does so with multiplicity one if and only if 
$\epsilon(\frac 12,\Sym^3(\pi_f))=-(-1)^{\#S'}$, i.e., $\#S'$ is even.
\end{conj}

If $C=1$, since the Gan-Gurevich lift is a cuspidal representation and $\epsilon(\frac 12,\Sym^3(\pi_f))=-1$, the above conjecture is true. 
If $S'=\emptyset$, since $S_0=\emptyset$, $\Pi=\Pi(f)$ and $\epsilon(\frac 12,\Sym^3(\pi_f))=-1$. Therefore, the above conjecture says that $\Pi(f)$ is always a discrete automorphic representation. Hence
 Conjecture \ref{Mundy} implies (\ref{assump}).

\begin{remark} 
If $S_0\ne\emptyset$, for $p\in S_0$, we may take $\Pi_p^-\in \{\pi(1), J_\beta(1,\pi(1,1))\}$ in the notation of Theorem \ref{constituents}. Let us give an example of $S_0=\emptyset$.
We have dim $S_4(\Gamma_0(5))^{\rm new}=1$. Let $f=q-4q^2+2q^3+8q^4-5q^5-8q^6+\cdots$ be the unique Hecke eigenform in the space.
Then by \cite{LW}, $\pi_5={\rm St}_5\otimes\mu_5$, where $\mu_5(5)=-1$. Hence in this case $S_0=\emptyset$.
If we assume Conjecture \ref{Mundy}, we obtain the Gan-Gurevich lift $F_f$ of weight 2. The cuspidal representation $\Pi_F$ generated by $F_f$ is given by $\Pi_F=\Pi_\infty\otimes \otimes_p' \Pi_p$, where $\Pi_\infty$ is the quaternionic discrete series, and $\Pi_p=J_\beta(\frac 12,\pi_p)$ for all $p$.
\end{remark}

\section{Appendix A: The archimedean component of the Gan-Gurevich lift}
In this Appendix, we will prove that the archimedean component of the Gan-Gurevich lift 
generates a quaternionic discrete series by using Arthur's classification \cite{A} and 
Li's result \cite{Li}. We refer \cite{Atobe1}, \cite{Atobe2} for using Arthur's classification and we will not recall all notations. 

Let $k\ge 6$ be an even integer and $f$ be a newform in $S_{2k}(\SL_2(\Z))$. 
Let $\pi_f$ be the cuspidal automorphic representation of $\GL_2(\A)$ attached to $f$.  
Let us consider the global Arthur parameter 
$$\psi=\tau_1[d_1]\boxplus\tau_2[d_2],\ \tau_1={\rm Sym}^2 \pi_f,\ \tau_2=\pi_f,\ 
d_1=1,\ d_2=2$$ for 
the symplectic group $\Sp_6$ (of rank 3)  
which corresponds to the restriction to $\Sp_6$ of 
the cuspidal automorphic representation $\Sigma(\sigma,\tau)$ on 
$\GSp_6(\A)$ with $\tau=\pi_f$ constructed in \cite[Section 4.3]{GG}. And $\Pi^{G}= \Theta^{E_7}_{G_2}(\Sigma(\sigma,\tau))$ in their notations.
In fact, since 
$\Sigma(\sigma,\tau)$ 
is cuspidal (\cite[Theorem 4.3]{GG}) and of level one, by Theorem \ref{Lfunct} with 
\cite[Proposition 5.1]{GG} and multiplicity one for $\mathcal{A}_{{\rm cusp}}(\Sp_6(\Q)\bs \Sp_6(\A))$ \cite[Corollary 8.5.4]{ChL}, any  
irreducible component of $\Sigma(\sigma,\tau)|_{\Sp_6(\A)}$ belongs to 
the global Arthur packet associated to the above $\psi$. 
Then, the component group of $\psi$ is given by $A_\psi=(\Z/2\Z)\alpha_{\tau_1[d_1]}\oplus (\Z/2\Z)\alpha_{\tau_2[d_2]}$. 
The Arthur character $\ve_\psi:A_\psi\lra \{\pm1\}$ is given by 
$$\ve_\psi(\alpha_{\tau_i[d_i]})=\ve(\frac 12,\pi_f\times {\rm Sym}^2 \pi_f)=
\ve(\frac 12,\pi_f)\ve(\frac 12,{\rm Sym}^3 \pi_f)=(-1)^k (-1)=-1
$$
for each $i=1,2$ since $k$ is even. Here $\ve(\frac 12,\pi_f\times {\rm Sym}^2 \pi_f)$ stands for 
the Rankin-Selberg epsilon factor. 
Let $\psi_\infty$ be the localization of $\psi$ at the archimedean place. Then, we have 
$$\psi_\infty=\rho_{4k-2}\boxtimes S_{1}\oplus  \rho_{2k-1}\boxtimes S_{2}\oplus 1
$$
where $S_d$ stands for the unique irreducible algebraic representation of $\SL_2(\C)$ of 
dimension $d$ and see \cite[Section 2.5]{Atobe1} for $\rho_\ast$. Let $\Pi_{\psi_\infty}$ 
be the corresponding local A-packet, given by the Adams-Johnson packet. As explained in \cite[Section 2.5]{Atobe1}, 
there is a bijection between $\Pi_{\psi_\infty}$  and the set $\mathcal{P}(1)\times 
\mathcal{P}(2)$ where $\mathcal{P}(d)=\{(p,q)\in \Z^2_{\ge 0}\ |\ p+q=d\}$. Thus, 
$|\Pi_{\psi_\infty}|=6$. Then, we can apply 
an explicit formula (\cite[p.49]{Atobe2} or \cite[Theorem 2.9]{Atobe1}) to compute the character for each element of $\Pi_{\psi_\infty}$. 
Then, only $w_1:=\{(0,1),(2,0)\}$ and $w_2=\{(0,1),(0,2)\}$ do match with the Arthur 
character $\ve_\psi$. Let $\pi_{w_i}$ be the corresponding discrete series in 
$\Pi_{\psi_\infty}$. Then, by using explicit description of $\pi_{w_i}$ given in \cite[Section 2.5]{Atobe1}, the Harish-Chandra parameter ${\rm HC}(\pi_{w_i})$ of $\pi_{w_i}$ is given by 
$${\rm HC}(\pi_{w_1})=(k,k-1,-(2k-1)),\ {\rm HC}(\pi_{w_2})=(-(k-1),-k,-(2k-1)).$$
Then ${\rm HC}(\pi_{w_1})$ corresponds to $\pi^{3,3}_{\infty,1}$, which is an irreducible 
discrete representation of $\Sp_6(\Bbb R)$ in \cite[Section 6.3.1]{CLJ}, where $r=x=2k-1$ and $s=y=1$ in terms of the notations there.  
On the other hand, $\pi_{w_2}$ is an anti-holomorphic discrete series of $\Sp_6(\Bbb R)$  and it never goes to 
$G_2(\R)$ under exceptional theta lifts (see \cite[the bottom line of p.45]{GG}).
Thus, we have $\Sigma(\sigma,\tau)_\infty=\pi_{w_1}=\pi^{3,3}_{\infty}$, which is an 
irreducible discrete series representation of $\GSp_6(\Bbb R)$ in the 
notation of \cite[Lemma 2.3]{CLJ}. 
Since $\Pi^G$ is a non-zero global exceptional cuspidal theta lifting by \cite{GG} 
and it is irreducible by \cite{NPS}, we conclude by \cite[Theorem 1.1]{Li} (see also \cite[Proposition 6.7]{CLJ})  
that 
$\Pi^{G}_\infty$ corresponds to 
the quaternionic discrete series $D_k$ in our notation.

\section{Appendix B: The Fourier-Jacobi expansion of Eisenstein series along $P$}\label{appenA}
Recall the Heisenberg parabolic subgroup $P=MN$ where the Heisenberg structure 
is given by (\ref{heisen2}). Let $\sP=\sM\ltimes N$ be the 
Jacobi group where $\sM=[M,M]\simeq SL_2$. Put 
$X=X_{\beta} X_{\alpha+\beta}=\{x=(x_1,x_2):=
x_{\beta}(x_2) x_{\alpha+\beta}(x_1)\in N\}$, 
$Y=X_{2\alpha+\beta} X_{3\alpha+\beta}=\{y=(y_1,y_2):=x_{2\alpha+\beta}(y_1)x_{3\alpha+\beta}(y_2)\in N\},$ 
and $Z=Z_N=X_{3\alpha+2\beta}=\{z=x_{3\alpha+2\beta}(\frac{1}{2}t)\in N,\ t\in 
\Bbb G_a\}$  
so that $N=XYZ$ and $X$ is a Lagrangian subgroup of $XY$.  
We write $v=v(x,y,z)=xyz,\ x\in X,\ y\in Y,\ z\in Z$ for each element of $N$. 
Put $\sigma(x,y):=\langle x,y \rangle= x_1y_2-3x_2y_1$ for $x=(x_1,x_2)\in X$ and $y=(y_1,y_2)\in Y$. 

For each $u\in \Q^\times$, let $\psi_u=\psi(u\ast )=\otimes'_{p}\psi_{u,p}$ where $\psi$ is the 
standard additive character on $Z_N(\A)\stackrel{\sim}{\lra} \A$, 
$x_{3\alpha+2\beta}(x)\mapsto x$.  
Let $\omega^\alpha_{\psi_u}=\otimes'\omega^\alpha_{u,p}=
\omega^\alpha_{u,\f}\otimes \omega^\alpha_{u,\infty}:
\widetilde{\sP(\A)}=\widetilde{\sM(\A)}\ltimes N(\A)\lra {\rm Aut}_\C(\mathcal{S}(X(\A)))$ be the Weil representation 
associated to $\psi_u(\frac{1}{2}\ast)$ acting on the Schwartz space $\mathcal{S}(X(\A))$. 
Though $M$ acts on $N$ as $\det^{-1}\otimes\rho_3$, 
it is easy to see that this action splits over $\sP(\A)$ and we have the action of  $\sP(\A)$ on $\mathcal{S}(X(\A))$. 
Explicitly, for each place $p\le \infty$ and $\Phi=\otimes'_{p\le\infty}\Phi_p=\Phi_{\f}
\otimes\Phi_{\infty}\in \mathcal{S}(X(\A))$, it is given by 
$$
\omega^\alpha_{u,p}(v(x,y,z))\Phi_p(t)=\Phi_p(t+x)
\psi_{u,p}\Big(\frac{1}{2}z+\sigma(t,y)+\frac{1}{2}\sigma(x,y)\Big),\ 
v(x,y,z)\in N(\Q_p),\ t\in X(\Q_p)  $$
$$
\omega^\alpha_{u,p}((m(\begin{pmatrix}
 a & 0 \\
 0 & a^{-1}
 \end{pmatrix}))\Phi_p(t_1,t_2)=|a|^{-2}_p\Phi_p\Big(\frac{t_1}{a^3},\frac{t_2}{a}\Big),\ 
 (t_1,t_2)\in X(\Q_p),\ a\in\Q^\times_p,\ 
 \chi_{u,p}(a):=\langle u,a \rangle_p
$$
$$
\omega^\alpha_{u,p}(m(\begin{pmatrix}
 1 & b \\
 0 & 1
 \end{pmatrix}))\Phi_p(t_1,t_2)=
 \psi_{u,p}(\frac{1}{2}\langle 
 (t_1, b t_1 + t_2, 0, 0),(0, 0, b^2 t_1 +2 b t_2 , b^3 t_1 + 3 b^2 t_2) \rangle)
\Phi_p(t_1,t_2+b t_1)$$
$$\phantom{xxxxxxxxxxxxxxx} =\psi_{u,p}(b^3 t_1^2+3 b^2 t_1 t_2 + 3 b t_2^2)\Phi_p(t_1,t_2+b t_1) ,\ b\in \Q_p,\ 
(t_1,t_2)\in X(\Q_p),
$$
$$
\omega^\alpha_{u,p}(w_\alpha)\Phi_p(t)=(F_S\Phi_p)(t),\ t\in X(\Q_p),\ 
(F_u\Phi_p)(t)=\ds\int_{X(\Q_p)} \Phi_p(x)\psi_{u,p}(\sigma(t,x))dx,
$$
where $dx$ means the Haar measure on $X(\Q_p)$ which is self-dual with respect to the Fourier transform $F_S$. In the first formula, $\frac{1}{2}z$ but not $z$ inside 
$\psi_{u,p}$ 
shows up because of the new coordinates ``$n_1$'' in (\ref{new-c}). 

For each $\Phi\in 
\mathcal{S}(X(\A))$, we define the theta function 
$$\Theta^\alpha_{\psi_u}(v(x,y,z)h;\Phi):=\sum_{\xi\in X(\Q)}\omega^\alpha_{\psi_u}(v(x,y,z)h)\Phi(\xi),\ 
v(x,y,z)\in N(\A),\ h\in\widetilde{M(\A)}$$
$$\phantom{xxxxx}=\sum_{\xi\in X(\Q)}(\omega^\alpha_{\psi_u}(h)\Phi)(x+\xi)\psi(u\sigma(\xi,y))
\psi\Big(\frac{u}{2}(z+\sigma(x,y))\Big).$$  

Recall the Eisenstein series $E(g;f)$ defined in (\ref{ES1}) where $f$ is a section of $I(s,\omega)$.  
Let $u\in\Q^\times$. For each $\Phi=\Phi_{\f}\otimes\Phi_{\infty}\in \mathcal{S}(X(\A))$, 
we define  
\begin{equation}\label{ESFJ}
E(h;f)_{\psi_u,\Phi}:=\int_{N(\Q)\bs N(\A)}E(vh;f)
\overline{\Theta^\alpha_{\psi_u}(vh;\Phi)}dv,\ h\in \sM(\A).
\end{equation}
For a character $\omega:
\Q^\times\bs \A^\times\lra \C^\times$, 
we define the space $I^\alpha_1(s,\omega)$ consisting of any $\sM(\widehat{\Z})
\times \SO(2)$-finite function $f:\sM(\A)\lra \C$ such that     
\begin{equation}\label{ps-alpha}
f(m(\begin{pmatrix}
 a & b \\
 0 & a^{-1}
 \end{pmatrix}) g)= \delta^{\frac{1}{2}}_{B_{\sM}}(m(\begin{pmatrix}
 a & b \\
 0 & a^{-1}
 \end{pmatrix}))|a|^s \omega(a)f(g),\ 
 a\in \A^\times,\ b\in \A,\ g\in \sM(\A).
 \end{equation}
The following theorem is an analogue of Theorem \ref{Eisen-exp}. Using 
the description of $P(\Q)\bs P(\Q)w P(\Q)$ for each 
$w\in P(\Q)\bs G_2(\Q)/P(\Q)=\{1,w_\beta,w_{\beta\alpha\beta},\iota \}$ (cf. \cite[Section 3]{JR}), 
it is similarly proved and therefore, we omit the proof. 
\begin{thm} Keep the notations as above. It holds 
$E(h;f)_{\psi_u,\Phi}=E(h;f)^{(1)}_{\psi_u,\Phi}+E(h;f)^{(2)}_{\psi_u,\Phi}$, 
$$E(h;f)^{(1)}_{\psi_u,\Phi}:=\sum_{\gamma\in B_{\sM}(\Q)\bs \sM(\Q)}R(\gamma h;f,\Phi),\  
E(h;f)^{(2)}_{\psi_u,\Phi}:=\int_{N(\A)}f(\iota vh)
\overline{\Theta^\alpha_{\psi_u}(vh;\Phi)}dv $$
where 
$$R(h;f,\Phi):=\ds\int_{Y(\A)Z(\A)}f(w_{\beta\alpha\beta}w_{\alpha}^{-1}v(0,y,z) w_\alpha h)
\overline{\omega^\alpha_{\psi_u}(h)\Phi(y)\psi(\frac{u}{2}z)}dydz
$$ 
is a section of 
$I^1_\alpha(3(s+\frac{1}{2}),\omega^3)$. Namely, $E(h;f)^{(1)}_{\psi_u,\Phi}$ is 
an Eisenstein series defined by a section $R(\ast;f,\Phi)$ on $\sM(\A)$. 
\end{thm}
Since $E(h;f)^{(2)}_{\psi_u,\Phi}=\ds\int_{N(\A)}f(\iota v)
\overline{\Theta^\alpha_{\psi_u}(hv;\Phi)}dv$, it is some kind of theta function on $\sM(\A)$, 
and it has an interesting transformation law by symmetric cubic structure.

\section{Appendix C: Explicit realization of $G_2$ inside $SO(3,4)$}\label{appenB}
Let us keep the notations in Section \ref{pre}. 
We define the Lie algebra $\frak g_2$ of $G_2$ as in \cite[Section 2.2, p.382]{Po}. 
Let $X_\gamma$ be a generator of ${\rm Lie}({\rm Im}(x_\gamma))$ for each $\gamma \in 
\Phi(G_2)$. Let $\frak h$ be the Cartan algebra of $\frak g_2$.  
In terms of Pollack's notation,  we have
$$\frak h=\langle  E_{11}-E_{22},E_{22}-E_{33}\rangle,$$
$$X_\alpha=v_2,\ X_\beta=E_{12},\ X_{\alpha+\beta}=v_1,\ 
X_{2\alpha+\beta}=\delta_3,\ X_{3\alpha+\beta}=E_{23},\ X_{3\alpha+2\beta}=E_{13}, $$
$$X_{-\alpha}=-\delta_2,\ X_{-\beta}=E_{21}={}^t E_{12},\ X_{-(\alpha+\beta)}=-\delta_1,\ 
X_{-(2\alpha+\beta)}=-v_3,
$$
$$ X_{-(3\alpha+\beta)}=E_{32}={}^t E_{23},\ X_{-(3\alpha+2\beta)}=E_{31}={}^t E_{13},
$$
where the readers should be careful with the sign ``$-1$'' for some negative roots. 

Let $SO(3,4)$ be the special orthogonal group associated to 
$$S=
\left(
\begin{array}{ccc}
 0 & 0 & 1_2 \\
 0 & S_0 & 0 \\
 1_2 & 0 & 0
\end{array}
\right),\ S_0=\left(
\begin{array}{ccc}
 0 & 0 & 1 \\
 0 & -2 & 0 \\
 1 & 0 & 0
\end{array}
\right).$$
In \cite{Po}, Pollack realized $\frak g_2$ inside ${\rm Lie}(SO(3,4))$. 
For each of $E_{ii},\delta_i,v_i\ (1\le i\le 3)$, a matrix presentation  is given as follows; 
$$\tiny{E_{11}=\left(
\begin{array}{ccccccc}
 1 & 0 & 0 & 0 & 0 & 0 & 0 \\
 0 & 0 & 0 & 0 & 0 & 0 & 0 \\
 0 & 0 & 0 & 0 & 0 & 0 & 0 \\
 0 & 0 & 0 & 0 & 0 & 0 & 0 \\
 0 & 0 & 0 & 0 & 0 & 0 & 0 \\
 0 & 0 & 0 & 0 & 0 & -1 & 0 \\
 0 & 0 & 0 & 0 & 0 & 0 & 0
\end{array}
\right),\ E_{22}=\left(
\begin{array}{ccccccc}
 0 & 0 & 0 & 0 & 0 & 0 & 0 \\
 0 & 0 & 0 & 0 & 0 & 0 & 0 \\
 0 & 0 & -1 & 0 & 0 & 0 & 0 \\
 0 & 0 & 0 & 0 & 0 & 0 & 0 \\
 0 & 0 & 0 & 0 & 1 & 0 & 0 \\
 0 & 0 & 0 & 0 & 0 & 0 & 0 \\
 0 & 0 & 0 & 0 & 0 & 0 & 0
\end{array}
\right),E_{33}=\left(
\begin{array}{ccccccc}
 0 & 0 & 0 & 0 & 0 & 0 & 0 \\
 0 & -1 & 0 & 0 & 0 & 0 & 0 \\
 0 & 0 & 0 & 0 & 0 & 0 & 0 \\
 0 & 0 & 0 & 0 & 0 & 0 & 0 \\
 0 & 0 & 0 & 0 & 0 & 0 & 0 \\
 0 & 0 & 0 & 0 & 0 & 0 & 0 \\
 0 & 0 & 0 & 0 & 0 & 0 & 1
\end{array}
\right)},$$

$$X_\alpha=v_2=\left(
\begin{array}{ccccccc}
 0 & 0 & 0 & 0 & 0 & 0 & 0 \\
 -1 & 0 & 0 & 0 & 0 & 0 & 0 \\
 0 & 0 & 0 & 0 & 0 & 0 & 0 \\
 0 & 0 & -1 & 0 & 0 & 0 & 0 \\
 0 & 0 & 0 & -2 & 0 & 0 & 0 \\
 0 & 0 & 0 & 0 & 0 & 0 & 1 \\
 0 & 0 & 0 & 0 & 0 & 0 & 0
\end{array}
\right),\ X_\beta=E_{12}=
\left(
\begin{array}{ccccccc}
 0 & 0 & 0 & 0 & -1 & 0 & 0 \\
 0 & 0 & 0 & 0 & 0 & 0 & 0 \\
 0 & 0 & 0 & 0 & 0 & 1 & 0 \\
 0 & 0 & 0 & 0 & 0 & 0 & 0 \\
 0 & 0 & 0 & 0 & 0 & 0 & 0 \\
 0 & 0 & 0 & 0 & 0 & 0 & 0 \\
 0 & 0 & 0 & 0 & 0 & 0 & 0
\end{array}
\right),$$

$$X_{\alpha+\beta}=v_1=\left(
\begin{array}{ccccccc}
 0 & 0 & 0 & 2 & 0 & 0 & 0 \\
 0 & 0 & 0 & 0 & -1 & 0 & 0 \\
 0 & 0 & 0 & 0 & 0 & 0 & 1 \\
 0 & 0 & 0 & 0 & 0 & 1 & 0 \\
 0 & 0 & 0 & 0 & 0 & 0 & 0 \\
 0 & 0 & 0 & 0 & 0 & 0 & 0 \\
 0 & 0 & 0 & 0 & 0 & 0 & 0
\end{array}
\right),\ 
X_{2\alpha+\beta}=\delta_3=\left(
\begin{array}{ccccccc}
 0 & 0 & -1 & 0 & 0 & 0 & 0 \\
 0 & 0 & 0 & 2 & 0 & 0 & 0 \\
 0 & 0 & 0 & 0 & 0 & 0 & 0 \\
 0 & 0 & 0 & 0 & 0 & 0 & 1 \\
 0 & 0 & 0 & 0 & 0 & 1 & 0 \\
 0 & 0 & 0 & 0 & 0 & 0 & 0 \\
 0 & 0 & 0 & 0 & 0 & 0 & 0
\end{array}
\right),$$

$$X_{3\alpha+\beta}=E_{23}=\left(
\begin{array}{ccccccc}
 0 & 0 & 0 & 0 & 0 & 0 & 0 \\
 0 & 0 & -1 & 0 & 0 & 0 & 0 \\
 0 & 0 & 0 & 0 & 0 & 0 & 0 \\
 0 & 0 & 0 & 0 & 0 & 0 & 0 \\
 0 & 0 & 0 & 0 & 0 & 0 & 1 \\
 0 & 0 & 0 & 0 & 0 & 0 & 0 \\
 0 & 0 & 0 & 0 & 0 & 0 & 0
\end{array}
\right),\ X_{3\alpha+2\beta}=E_{13}=\left(
\begin{array}{ccccccc}
 0 & 0 & 0 & 0 & 0 & 0 & -1 \\
 0 & 0 & 0 & 0 & 0 & 1 & 0 \\
 0 & 0 & 0 & 0 & 0 & 0 & 0 \\
 0 & 0 & 0 & 0 & 0 & 0 & 0 \\
 0 & 0 & 0 & 0 & 0 & 0 & 0 \\
 0 & 0 & 0 & 0 & 0 & 0 & 0 \\
 0 & 0 & 0 & 0 & 0 & 0 & 0
\end{array}
\right),$$
$$X_{-\alpha}=-\delta_2=\left(
\begin{array}{ccccccc}
 0 & -1 & 0 & 0 & 0 & 0 & 0 \\
 0 & 0 & 0 & 0 & 0 & 0 & 0 \\
 0 & 0 & 0 & -2 & 0 & 0 & 0 \\
 0 & 0 & 0 & 0 & -1 & 0 & 0 \\
 0 & 0 & 0 & 0 & 0 & 0 & 0 \\
 0 & 0 & 0 & 0 & 0 & 0 & 0 \\
 0 & 0 & 0 & 0 & 0 & 1 & 0
\end{array}
\right),\ X_{-(\alpha+\beta)}=-\delta_1=
\left(
\begin{array}{ccccccc}
 0 & 0 & 0 & 0 & 0 & 0 & 0 \\
 0 & 0 & 0 & 0 & 0 & 0 & 0 \\
 0 & 0 & 0 & 0 & 0 & 0 & 0 \\
 1 & 0 & 0 & 0 & 0 & 0 & 0 \\
 0 & -1 & 0 & 0 & 0 & 0 & 0 \\
 0 & 0 & 0 & 2 & 0 & 0 & 0 \\
 0 & 0 & 1 & 0 & 0 & 0 & 0
\end{array}
\right),$$
$$X_{-(2\alpha+\beta)}=-v_3=
\left(
\begin{array}{ccccccc}
 0 & 0 & 0 & 0 & 0 & 0 & 0 \\
 0 & 0 & 0 & 0 & 0 & 0 & 0 \\
 -1 & 0 & 0 & 0 & 0 & 0 & 0 \\
 0 & 1 & 0 & 0 & 0 & 0 & 0 \\
 0 & 0 & 0 & 0 & 0 & 0 & 0 \\
 0 & 0 & 0 & 0 & 1 & 0 & 0 \\
 0 & 0 & 0 & 2 & 0 & 0 & 0
\end{array}
\right)$$
and $X_{-\ast}={}^t X_\ast$ for $\ast\in\{\beta,3\alpha+\beta,3\alpha+2\beta\}$. 

We define, for $\gamma\in \Phi(G_2)$,
$$x_\gamma(u):=\exp(u X_\gamma):=\ds\sum_{n\ge 0}\frac{u^n X^n_\gamma}{n!},\ 
u\in \mathbb{G}_a.
$$ 

Put, for $t\in GL_1$,
$$f_{E_{11}}(t)=\diag(t,1,1,1,1,t^{-1},1),\ f_{E_{22}}(t)=\diag(1,1,t^{-1},1,t,1,1),$$
$$f_{E_{33}}(t)=\diag(1,t^{-1},1,1,1,1,t),\ 
\exp(\log t(E_{ii}-E_{jj})):=f_{E_{ii}}(t)f_{E_{jj}}(t)^{-1}.
$$ 

Let $P=MN$ be the Heisenberg parabolic subgroup such that 
$$\frak n:={\rm Lie}N=\langle \ X_\beta=E_{12},\ X_{\alpha+\beta}=v_1,\ 
X_{2\alpha+\beta}=\delta_3,\ X_{3\alpha+\beta}=E_{23},\ X_{3\alpha+2\beta}=E_{13}  \rangle,
$$
$$\frak m:={\rm Lie}M=\langle E_{11}-E_{22},E_{22}-E_{33},X_\alpha=v_2,X_{-\alpha}=\delta_2 \rangle.
$$
Then, the root spaces give a structure of $N$ in $SO(3,4)$ as 
$$n(a_1,a_2,a_3,a_4,t):=
\exp (a_1 X_{\beta})
\exp (a_2 X_{\alpha+\beta})
\exp (a_3 X_{2\alpha+\beta})
\exp (a_4 X_{3\alpha+\beta})\exp (t X_{3\alpha+2\beta})=$$
$$\left(
\begin{array}{ccccccc}
 1 & 0 & -a_3 & 2 a_2 & -a_1 & a_2^2-a_1 a_3 & 2 a_2 a_3-a_1 a_4-t \\
 0 & 1 & -a_4 & 2 a_3 & -a_2 & -a_2 a_3+t & a_3^2-a_2 a_4 \\
 0 & 0 & 1 & 0 & 0 & a_1 & a_2 \\
 0 & 0 & 0 & 1 & 0 & a_2 & a_3 \\
 0 & 0 & 0 & 0 & 1 & a_3 & a_4 \\
 0 & 0 & 0 & 0 & 0 & 1 & 0 \\
 0 & 0 & 0 & 0 & 0 & 0 & 1
\end{array}
\right).
$$
The Levi factor $M$ is realized by the Zariski closure of the set consisting of 
$$\exp(\log a(E_{22}-E_{33}))\exp(\log d(E_{11}-E_{22}))\exp(b v_2)\exp(-c\delta_2)=$$
$$\left(
\begin{array}{ccccccc}
 d & -c d & 0 & 0 & 0 & 0 & 0 \\
 -a b & a (b c+1) & 0 & 0 & 0 & 0 & 0 \\
 0 & 0 & \frac{d}{a} & -\frac{2 c d}{a} & \frac{c^2 d}{a} & 0 & 0 \\
 0 & 0 & -b & 2 b c+1 & -c (b c+1) & 0 & 0 \\
 0 & 0 & \frac{a b^2}{d} & -\frac{2 a b (b c+1)}{d} & \frac{a (b c+1)^2}{d} & 0 & 0 \\
 0 & 0 & 0 & 0 & 0 & \frac{b c+1}{d} & \frac{b}{d} \\
 0 & 0 & 0 & 0 & 0 & \frac{c}{a} & \frac{1}{a} \\
\end{array}
\right).$$
The birational transformation 
$\left(
\begin{array}{cc}
 a & b \\
 c & d
\end{array}
\right)\mapsto 
\left(
\begin{array}{cc}
 \frac{ad-bc}{a} & \frac{ac}{ad-bc}  \\
 -\frac{b}{a} & a
\end{array}
\right)$ 
yields another expression $m':GL_2\stackrel{\sim}{\lra}M$ given by   
$$m'(\left(
\begin{array}{cc}
 a & b \\
 c & d
\end{array}
\right))=\left(
\begin{array}{ccccccc}
 a & b & 0 & 0 & 0 & 0 & 0 \\
 c & d & 0 & 0 & 0 & 0 & 0 \\
 0 & 0 & \frac{a^2}{ad-bc} & \frac{2 a b}{ad-bc} & \frac{b^2}{ad-bc} & 0 & 0 \\
 0 & 0 & \frac{a c}{ad-bc} & \frac{b c+a d}{ad-bc} & \frac{b d}{ad-bc} & 0 & 0 \\
 0 & 0 & \frac{c^2}{ad-bc} & \frac{2 c d}{ad-bc} & \frac{d^2}{ad-bc} & 0 & 0 \\
 0 & 0 & 0 & 0 & 0 & \frac{d}{ad-bc} & -\frac{c}{ad-bc} \\
 0 & 0 & 0 & 0 & 0 & -\frac{b}{ad-bc} & \frac{a}{ad-bc}
\end{array}
\right).$$
Then, the coordinates of $M$ in Section \ref{pre} is defined by 
$$m(\left(
\begin{array}{cc}
 a & b \\
 c & d
\end{array}
\right))=m'(\left(
\begin{array}{cc}
 d & c \\
 b & a
\end{array}
\right))=\left(
\begin{array}{ccccccc}
 d & c & 0 & 0 & 0 & 0 & 0 \\
 b & a & 0 & 0 & 0 & 0 & 0 \\
 0 & 0 & \frac{d^2}{a d-b c} & \frac{2 c d}{a d-b c} & \frac{c^2}{a d-b c} & 0 & 0 \\
 0 & 0 & \frac{b d}{a d-b c} & \frac{a d+b c}{a d-b c} & \frac{a c}{a d-b c} & 0 & 0 \\
 0 & 0 & \frac{b^2}{a d-b c} & \frac{2 a b}{a d-b c} & \frac{a^2}{a d-b c} & 0 & 0 \\
 0 & 0 & 0 & 0 & 0 & \frac{a}{a d-b c} & -\frac{b}{a d-b c} \\
 0 & 0 & 0 & 0 & 0 & -\frac{c}{a d-b c} & \frac{d}{a d-b c} \\
\end{array}
\right).$$

Next we consider the Siegel parabolic subgroup $Q=LU$. 
$$\frak u:={\rm Lie}\hspace{0.5mm}U=\langle \ X_\alpha=v_2,\ X_{\alpha+\beta}=v_1,\ 
X_{2\alpha+\beta}=\delta_3,\ X_{3\alpha+\beta}=E_{23},\ X_{3\alpha+2\beta}=E_{13}  \rangle,
$$
$$\frak l:={\rm Lie}\hspace{0.5mm}L=\langle E_{11}-E_{22},E_{22}-E_{33},X_\beta=E_{12},X_{-\beta}=E_{21}={}^tE_{12} \rangle .$$
Then, the root spaces give a structure of $U$ in $SO(3,4)$ as 
$$u(a_1,a_2,a_3,a_4,z):=
\exp (a_1 X_{\alpha})
\exp (a_2 X_{\alpha+\beta})
\exp (a_3 X_{2\alpha+\beta})
\exp (a_4 X_{3\alpha+\beta})\exp (z X_{3\alpha+2\beta})=$$
$$
\left(
\begin{array}{ccccccc}
 1 & 0 & -a_3 & 2 a_2 & 0 & a_2^2 & 2 a_2 a_3-z \\
 -a_1 & 1 & a_1 a_3-a_4 & -2 (a_1 a_2-a_3) & -a_2 & -a_1 a_2^2-a_2 a_3+z & -2 a_1 a_2 a_3+a_1 z-a_2 a_4+a_3^2 \\
 0 & 0 & 1 & 0 & 0 & 0 & a_2 \\
 0 & 0 & -a_1 & 1 & 0 & a_2 & a_3-a_1 a_2 \\
 0 & 0 & a_1^2 & -2 a_1 & 1 & a_3-2 a_1 a_2 & a_1^2 a_2-2 a_1 a_3+a_4 \\
 0 & 0 & 0 & 0 & 0 & 1 & a_1 \\
 0 & 0 & 0 & 0 & 0 & 0 & 1 \\
\end{array}
\right).
$$
The Levi factor $L$ is realized by the Zariski closure of the set consisting of 
$$\exp(\log a(E_{22}-E_{33}))\exp(\log d(E_{11}-E_{22}))\exp(b E_{12})\exp(c E_{21})=$$
$$
\left(
\begin{array}{ccccccc}
 (1-b c)d) & 0 & 0 & 0 & -b d & 0 & 0 \\
 0 & a & 0 & 0 & 0 & 0 & 0 \\
 0 & 0 & \frac{d (1-b c)}{a} & 0 & 0 & \frac{b d}{a} & 0 \\
 0 & 0 & 0 & 1 & 0 & 0 & 0 \\
 \frac{a c}{d} & 0 & 0 & 0 & \frac{a}{d} & 0 & 0 \\
 0 & 0 & -\frac{c}{d} & 0 & 0 & \frac{1}{d} & 0 \\
 0 & 0 & 0 & 0 & 0 & 0 & \frac{1}{a} \\
\end{array}
\right).$$
The birational transformation  
$\left(
\begin{array}{cc}
 a & b \\
 c & d
\end{array}
\right)\mapsto \left(
\begin{array}{cc}
 ad-bc & -\frac{bd}{ad-bc} \\
 -\frac{c}{d} & \frac{ad-bc}{d} 
\end{array}
\right)$
yields another expression $l:GL_2\stackrel{\sim}{\lra}L$ given by   
$$l(\left(
\begin{array}{cc}
 a & b \\
 c & d
\end{array}
\right))=\left(
\begin{array}{ccccccc}
 a & 0 & 0 & 0 & b & 0 & 0 \\
 0 & a d-b c & 0 & 0 & 0 & 0 & 0 \\
 0 & 0 & \frac{a}{a d-b c} & 0 & 0 & -\frac{b}{a d-b c} & 0 \\
 0 & 0 & 0 & 1 & 0 & 0 & 0 \\
 c & 0 & 0 & 0 & d & 0 & 0 \\
 0 & 0 & -\frac{c}{a d-b c} & 0 & 0 & \frac{d}{a d-b c} & 0 \\
 0 & 0 & 0 & 0 & 0 & 0 & \frac{1}{a d-bc}
\end{array}
\right).$$
Finally, the Weyl elements defined in Section \ref{pre} is explicitly given by 
$$w_\alpha=m\left(
\begin{array}{cc}
 0 &-1 \\
 1 & 0
\end{array}
\right),\ w_\beta=\ell\left(
\begin{array}{cc}
 0 &1 \\
 -1 & 0
\end{array}
\right).
$$



\end{document}